\documentclass[11pt,reqno]{amsart}

\usepackage[T1]{fontenc}
\usepackage{lmodern}
\usepackage[margin=1.02in]{geometry}
\usepackage{amsmath,amssymb,mathtools,amscd,mathrsfs}
\usepackage{enumitem}
\usepackage[expansion=false]{microtype}
\usepackage{booktabs,array}
\usepackage{aliascnt,amsmath}
\usepackage{xcolor}
\definecolor{SylvesterLink}{RGB}{18,72,120}
\usepackage{hyperref}
\hypersetup{
  pdftitle={A proof of Sylvester's conjecture},
  pdfauthor={Ashay A. Burungale and Ye Tian},
  pdfsubject={Sylvester's conjecture on sums of two rational cubes},
  colorlinks=true,
  allcolors=SylvesterLink,
  linktoc=all,
  bookmarksnumbered=true,
  bookmarksopen=true,
  bookmarksopenlevel=3,
  bookmarksdepth=3,
  pdfdisplaydoctitle=true,
  pdfstartview=FitH
}
\usepackage[nameinlink,noabbrev]{cleveref}
\numberwithin{equation}{section}

\allowdisplaybreaks[2]
\setlist[itemize]{leftmargin=1.7em,itemsep=.2ex,topsep=.4ex}
\setlist[enumerate]{leftmargin=2em,itemsep=.25ex,topsep=.5ex}

\newtheorem{theorem}{Theorem}[section]
\newaliascnt{proposition}{theorem}
\newtheorem{proposition}[proposition]{Proposition}
\aliascntresetthe{proposition}
\newaliascnt{lemma}{theorem}
\newtheorem{lemma}[lemma]{Lemma}
\aliascntresetthe{lemma}
\newaliascnt{corollary}{theorem}
\newtheorem{corollary}[corollary]{Corollary}
\aliascntresetthe{corollary}
\theoremstyle{definition}
\newaliascnt{definition}{theorem}
\newtheorem{definition}[definition]{Definition}
\aliascntresetthe{definition}
\theoremstyle{remark}
\newaliascnt{remark}{theorem}
\newtheorem{remark}[remark]{Remark}
\aliascntresetthe{remark}
\newaliascnt{example}{theorem}

\aliascntresetthe{example}
\font\cyr=wncyr10
\newcommand{\Sha}{\hbox{\cyr X}}

\title[Sylvester's conjecture]{A proof of Sylvester's conjecture}
\author{Ashay A. Burungale}
\address{Department of Mathematics, The University of Texas at Austin, 2515 Speedway, Austin, TX 78712, USA}
\email{ashayk@utexas.edu}
\author{Ye Tian}
\address{Morningside Center of Mathematics and Institute of Mathematics, Academy of Mathematics and Systems Science, Chinese Academy of Sciences, Beijing 100190, China}
\email{ytian@math.ac.cn}
\date{September 2026}

\begin{document}

\begin{abstract}

We prove Sylvester's conjecture, originating in his 1879 study of
ternary cubic equations, that every prime
$p\equiv4,7,8\pmod9$ is a sum of two rational cubes.  
Elkies announced a proof for the classes $4$ and $7$ in 1994, and Hongbo Yin recently supplied a complete proof. 
For the remaining class $p\equiv8\pmod9$, we prove that the elliptic curve
$E_p:y^2=x^3+p^2/4$ has analytic rank one, as predicted by the Birch
and Swinnerton-Dyer conjecture, and so $p$ is a sum of two
rational cubes.

The proof begins by adapting the auxiliary Rankin--Selberg construction
from the authors' work on the rank-one converse for CM elliptic curves.
The Rankin--Selberg $L$-function factors as the $L$-function of $E_p$ times a
complementary $L$-function.  Chan's $3$-isogeny
descent and the rank-zero converse show that the complementary central
$L$-value is non-zero, and so it suffices to prove that a cubic component of the
associated Heegner point is non-torsion.  A basic difficulty is that the
unweighted Hecke trace of the CM orbit vanishes.  Our decisive idea
is to take $\lambda$-division before taking the trace, where
$\lambda=1-\omega$ and $\omega$ is a primitive cube root of unity.  We prove that the
resulting division boundary is non-zero by analysing Frobenius
at $p$.  The Galois action on the CM orbit and
ramification theory then transfer this non-vanishing to the 
cubic component.
\end{abstract}
\maketitle

\setcounter{tocdepth}{1}
\tableofcontents

\section{Introduction}\label{sec:introduction}

\subsection{The cube-sum problem}
Which primes can be written as sums of two rational cubes?  In 1879,
Sylvester studied the ternary cubic equations underlying this question
\cite{Sylvester}. 

The conjecture bearing his name\footnote{Sylvester's conjecture is a modern formulation rather than a verbatim
statement in Sylvester's papers.  More precisely, Selmer \cite{Selmer} placed the
problem in the modern context, and later Birch and Stephens
\cite{BS66} formulated the above precise conjecture in light of the
Birch and Swinnerton-Dyer conjecture.} asserts that every prime
$p\equiv4,7,8\pmod9$ is a sum of two rational cubes.  Elkies announced
the classes $4$ and $7$ in 1994 \cite{Elkies}, and Hongbo Yin recently
completed their proof \cite{Yin47,YinGZ}.  The main result of this
paper settles the remaining class $p\equiv8\pmod9$ and hence resolves
Sylvester's conjecture.  For example, $17\equiv8\pmod9$ and
\[
 17=\left(\frac{18}{7}\right)^3+\left(-\frac17\right)^3.
\]

For $m\in\mathbf Q^\times$, put
\begin{equation}\label{eq,ell}
 E_m:\quad y^2=x^3+\frac{m^2}{4}.
\end{equation}
With $[1:-1:0]$ as origin, the diagonal cubic
$X^3+Y^3=mZ^3$ is $3$-isogenous to $E_m$ over $\mathbf Q$.  For
cube-free integers $m\ge3$, classical
descent gives
\[
 m\text{ is a sum of two non-zero rational cubes}
 \quad\Longleftrightarrow\quad
 \operatorname{rank}E_m(\mathbf Q)>0;
\]
see \cite{Fueter,Selmer,Satge}.  Let $L(s,E_m)$ denote the associated
Hasse--Weil $L$-function, normalized to have center $s=1$.

We prove the following analytic rank result for $E_p$ and $E_{p^2}$.

\begin{theorem}\label{thm:main}
Let $p\equiv8\pmod9$ be a prime.  Then
\[
 \operatorname{ord}_{s=1}L(s,E_p)
 =\operatorname{ord}_{s=1}L(s,E_{p^2})=1.
\]
\end{theorem}

Consequently, the theorems of Gross--Zagier and Kolyvagin imply that $E_p$ and
$E_{p^2}$ have Mordell--Weil rank one
\cite{GrossZagier,Kolyvagin,GrossKolyvagin}.  Together with Hongbo Yin's
results for the classes $4$ and $7$, we obtain the following.

\begin{theorem}[Sylvester's conjecture]\label{thm:sylvester}
Every prime $p\equiv4,7,8\pmod9$ is a sum of two rational cubes.
\end{theorem}

\begin{corollary}\label{cor:arithmetic-rank}
For $p\equiv8\pmod9$ and $j=1,2$ one has
\[
 \operatorname{rank}_{\mathbf Z}E_{p^j}(\mathbf Q)=1,
 \qquad \#\Sha(E_{p^j}/\mathbf Q)<\infty.
\]
\end{corollary}

\subsection{Prior work and historical perspective}

\subsubsection{Descent and the BSD prediction}\label{subsec:descent-bsd}
The rational point $(0,m/2)\in E_m(\mathbf Q)$ has order three and
defines a $3$-isogeny.  Descent along this isogeny and its dual bounds
the Mordell--Weil rank by testing local solubility of explicit cubic
equations.  For a prime $p>3$, the classical descent calculation and
the root number lead to the following picture for both $E_p$ and
$E_{p^2}$ \cite{Selmer,Satge,BS66,Liverance}:
\begin{center}
\begin{tabular}{c c c l}
\toprule
$p\bmod9$ & descent bound for the rank & root number & BSD-predicted rank\\
\midrule
$2,5$ & $0$ & $+1$ & $0$\\
$4,7,8$ & $1$ & $-1$ & $1$\\
$1$ & $2$ & $+1$ & $0$ or $2$\\
\bottomrule
\end{tabular}
\end{center}
In particular, for $p\equiv4,7,8\pmod9$, one has
$L(1,E_p)=L(1,E_{p^{2}})=0$, and the problem is to prove that both zeros
are simple.  

The underlying elliptic curves have complex multiplication: letting 
$K=\mathbf Q(\omega)$ for $\omega^2+\omega+1=0$, their CM ring is
$\mathcal O_K=\mathbf Z[\omega]$.
\subsubsection{Heegner points and cube sums}
In 1952, Heegner used complex multiplication and modular functions to
construct rational solutions of a Diophantine equation \cite{Heegner52}. 
In modern terms, his construction combines CM points on modular
curves with modular parametrizations, and Shimura's reciprocity law
governs the Galois action on the CM points. The influence of
Heegner's idea has proved remarkably enduring.

Satg\'e adapted CM constructions of this kind for special cases of the
cube-sum problem \cite{Satge,SatgeHeegner}.  In 1994 Elkies announced the resolution of
the cases $p\equiv4,7\pmod9$ \cite{Elkies}.  Such a prime splits in $K$.
Write $p=\pi\bar\pi$ in $\mathbf Z[\omega]$, with
$\pi\equiv\bar\pi\equiv1\pmod3$, and put
\[
 C_\pi:\qquad u^3+v^3=\pi w^3.
\]
Elkies considers a modular curve $\mathscr X$ and defines an explicit
modular parametrization
\[
 \Phi:\mathscr X\longrightarrow C_\pi.
\]
He uses a suitable CM point to obtain a point on $C_\pi$ over
$K(\sqrt[3]{\bar\pi})$, and twists it back to a point on $E_p$ over $K$.
This construction, summarized in \cite[\S4.1]{DV09}, was not published
with a complete proof.  Dasgupta--Voight subsequently developed it and
proved the $4,7$ cases under an additional cubic-residue hypothesis
\cite{DV09,DV18}.  Recently, Hongbo Yin gave a complete proof of the $4,7$
cases by a construction inspired by, but distinct from, Elkies' sketch
\cite{Yin47,YinGZ}, thereby completing Elkies' announcement.  The
second-named author's work on the congruent number problem provides another precursor: its key is a $2$-power indivisibility of the corresponding 
Heegner point \cite{Tian14}.  For other partial and
statistical results on cube sums, see \cite{RVZ,BES,ABS}.

The class $p\equiv8\pmod9$ presents an essential difficulty.  Since such
a $p$ is inert in $K$, neither Elkies' aforementioned construction via
$\Phi$ nor Hongbo Yin's variant applies.  Moreover, the natural conductor-$p$ Hecke trace
of an underlying CM orbit vanishes. 

Another related circle of ideas concerns $\ell$-converse theorems
to the Gross--Zagier--Kolyvagin theorem.  For many primes $\ell$
of good reduction, the rank-one $\ell$-converse theorems of Skinner,
Zhang and the authors, among others, use Iwasawa theory to deduce
that a suitable Heegner point is non-torsion from the hypothesis
that the $\ell^\infty$-Selmer corank is one
\cite{WZ14,Skinner20,BT20,BSTsurvey}. 
For $E_p$, a rank-one $3$-converse
would suffice since the $3^\infty$-Selmer corank is one.  However, $3$ is a
prime of additive reduction for $E_p$, and the requisite Iwasawa theory
at additive primes remains fragmentary\footnote{See \cite{bkno1,bkno2}
for the first steps.}.

\subsection{Outline of the proof}
Fix a prime $p\equiv8\pmod9$.  In the following, we explain the strategy for $E_p$, a variant of which applies for $E_{p^2}$.

\subsubsection{The auxiliary Rankin--Selberg convolution} 
We begin with an auxiliary Rankin--Selberg construction based on the
method of \cite{BT20}.

Choose an auxiliary prime $q\equiv4\pmod9$. Note that it splits in $K$ and write
$q=\varpi\bar\varpi$, with $\varpi\equiv1\pmod3$.  For
$a\in K^\times$, let $E_a$ denote the elliptic curve given by the same equation
as \eqref{eq,ell}, and let $\nu_a$ be its CM Hecke character.  
Let
$g_q=\theta(\nu_\varpi)$ be the associated theta series, and put
$\chi_q=\nu_p\nu_\varpi^{-1}$.
The pair $(g_q,\chi_q)$ is self-dual and the associated Rankin--Selberg convolution has root number $-1$.
The Yuan--Zhang--Zhang (YZZ) generalization of the Gross--Zagier formula
therefore relates the first central derivative of the associated $L$-function to a Heegner point \cite{YZZ}.
In our CM case, the Rankin--Selberg $L$-function factors as
\begin{equation}\label{eq:intro-rankin}
 L(s,g_q\times\chi_q)=L(s,E_p)L(s,E_{qp^2}).
\end{equation}

The auxiliary prime $q$ is chosen so that
\begin{equation}\label{eq:intro-ray-condition}
 \alpha_q(p):=\left(\frac p\varpi\right)_3\ne1, 
\end{equation}
where $(\cdot/\varpi)_3$ is the cubic residue symbol.
This condition has two roles.  It gives the complementary non-vanishing
below; through cubic reciprocity \eqref{eq:cubic-reciprocity-at-p}, it also
makes $p^2$-Frobenius act non-trivially on a cube root of $\varpi$, which
forces the division boundary constructed below to be non-zero.
Chebotarev gives
infinitely many such primes.  
For such a $q$, one has 
\begin{equation}\label{eq:intro-complementary-nonvanishing}
 L(1,E_{qp^2})\ne0.
\end{equation}
Indeed, \eqref{eq:intro-ray-condition} implies that the $3$-Selmer group
associated to $E_{qp^2}$ is generated by the Kummer image of a rational
$3$-torsion point (cf.~\cite{Chan}).  Hence the authors' rank-zero
$3$-converse theorem \cite{BTconverse} yields the non-vanishing
\eqref{eq:intro-complementary-nonvanishing}.

The pair $\{g_q,g_q^c\}$ determines an abelian surface
$A_q/\mathbf Q$, where $c$ denotes coefficient conjugation.  After
base change to $K$, its two
coefficient-conjugate factors are $E_{\bar\varpi}$ and $E_\varpi$,
giving quotient maps
\[
 \operatorname{pr}_{\bar\varpi}:A_q\longrightarrow E_{\bar\varpi},
 \qquad
 \operatorname{pr}_{\varpi}:A_q\longrightarrow E_\varpi,
\]
whose product is an isogeny.  
The construction of Yuan--Zhang--Zhang \cite{YZZ} attaches to the pair
$(g_q,\chi_q)$ a Heegner point
$$\mathcal P_q\in A_q(K^{\rm ab})_{\mathbf Q}:=A_q(K^{\rm ab})
\otimes_{\mathbf Z}\mathbf Q,$$ where $K^{\rm ab}$ denotes the maximal
abelian extension of $K$.  The YZZ generalization of the
Gross--Zagier
formula \cite{YZZ}, together with \eqref{eq:intro-complementary-nonvanishing},
reduces the analytic rank one assertion for $E_p$ to the following.
\begin{theorem}\label{thm:auxiliary-heegner}
The projection $\operatorname{pr}_\varpi(\mathcal P_q)$ is non-zero in
$E_\varpi(K^{\rm ab})_{\mathbf Q}$.
\end{theorem}

\begin{remark}
The importance of an auxiliary twist with non-vanishing central
$L$-value\footnote{In their work, the auxiliary and target primes coincide.} was noticed and emphasized by Dasgupta--Voight
\cite[\S\S4.4--4.5]{DV09}. 
In our setting, the rational $3$-torsion on the $3$-isogenous model
$E_m$ makes the complementary factor accessible to an explicit
$3$-isogeny descent.
From this perspective, the condition \eqref{eq:intro-ray-condition} is a $q\neq p$-analogue 
of the cubic residue
condition in \cite{DV09}.
\end{remark}

\subsubsection{The base CM point and its first division}\label{s:div}
Let $X_\Gamma$ be the degree-three modular cover of $X_0(9q)$ attached
to the cubic residue character modulo $q$.  
Fix the normalized modular parametrizations
\[
 \varphi:X_\Gamma\longrightarrow E_{\bar\varpi},\qquad
 \varphi^c:X_\Gamma\longrightarrow E_\varpi
\]
as in \cite{Yin47}, where
the superscript $c$ denotes the coefficient-conjugate factor of
$A_q$ and both maps send the cusp $\infty$ to the origin (see
\Cref{sec:cm-data,sec:rankin-gz}).

Put \(F_q=K(t)\), where \(t^3=\varpi\), and write
\(\lambda=1-\omega\).   Then there exists a CM point \(P_0\in X_\Gamma(F_q)\) such that\footnote{The occurrence of this cubic field, rather than a ring-class field,
reflects the fine level structure at $q$: Shimura reciprocity identifies
the field of definition of this CM point with $F_q$
(cf.~\Cref{lem:fine-base-field} and \cite[Theorem~4.1]{Yin47}).}
$$
 Q_q:=\varphi(P_0)\notin E_{\bar\varpi}[\lambda],
 \qquad
 \varphi^c(P_0)=(0,\eta_q\varpi/2),
 \qquad \eta_q\in\{1,-1\}.
$$
The key point is the following cubic identity over $K$:
\[
 y(Q_q)+\eta_q\bar\varpi/2=\eta_q\varpi z_q^3,
 \qquad z_q\in K^\times.
\]
Indeed, the ratio of the corresponding cubic descent functions is the cube of
a modular function, and Shimura reciprocity shows that the value of its
normalized cube root at \(P_0\) lies in \(K^\times\).

This description also makes the first \(\lambda\)-division of \(Q_q\)
explicit.  On the Fermat cubic
$$
 \mathscr C_{\bar\varpi}:\quad V^3-U^3=\bar\varpi W^3,
$$
the cubic identity produces a point \(R_q\) mapping to \(Q_q\) under the
natural degree-three isogeny.  After identifying this isogeny with
multiplication by \(\lambda\), \(R_q\) is a first
\(\lambda\)-division of \(Q_q\).  Exactly one of its affine
co-ordinates lies in \(K\), while the other lies in \(tK\).  In view of this
asymmetry, the subsequent Frobenius calculation at \(p\)
detects the cubic residue symbol \((p/\varpi)_3\).

\subsubsection{The vanishing Hecke trace and its boundary}
The prime-to-level Hecke correspondence at $p$, applied to $P_0$,
consists of $p+1$ CM points.  They are indexed by the cyclic group
\[
 C_p=\mathbf F_{p^2}^{\times}/\mathbf F_p^{\times},
 \qquad |C_p|=p+1.
\]
Let $Q^u$ denote the $\varphi^c$-image of the point indexed by
$u\in C_p$, and put $Q=Q^1$.  Since the $T_p$-eigenvalue of $g_q^c$
is zero, the Hecke relation gives
\begin{equation}\label{eq:intro-vanishing-orbit}
 \sum_{u\in C_p}Q^u=O.
\end{equation}

Let $F$ be the completion at a prime above $p$ of the field of
definition of $Q$.  Since $E_\varpi$ has good reduction at $p$
and $p\ne3$, multiplication by $\lambda$ extends to a finite
\'etale isogeny of its N\'eron model over $\mathcal O_F$.
Hence, after passing to a finite unramified extension $F'/F$,
we may choose a point $R\in E_\varpi(F')$ such that $[\lambda]R=Q$.
Then
$[\lambda]R^u=Q^u$ for every $u\in C_p$, and the trace of these
divisions defines the fundamental invariant 
\begin{equation}\label{eq:intro-boundary-definition}
 \partial^c_{q,p}:=\sum_{u\in C_p}R^u\in E_\varpi[\lambda](K).
\end{equation}
This is the division boundary of the trace-zero orbit\footnote{It is
independent of the choice of $R$: another choice differs by a
$\lambda$-torsion point, whose trace is $(p+1)$ times that point and so zero.}.  

It has the following Tate-cohomological interpretation. 
Let $M$ denote the $C_p$-module generated by $Q$ and 
write
\[
 N_{C_p}:M\longrightarrow M,\qquad
 N_{C_p}(m)=\sum_{u\in C_p}m^u. 
\]
 In view of the vanishing Hecke relation
\eqref{eq:intro-vanishing-orbit}, one has
 \(Q\in\ker N_{C_p}\).  Hence \(Q\) determines a
class
\[
 [Q]\in
 \widehat H^{-1}(C_p,M). 
\]
Indeed, put
$\widetilde M=[\lambda]^{-1}M$ and then the exact sequence
\[
 0\longrightarrow E_\varpi[\lambda]\longrightarrow\widetilde M
 \xrightarrow{[\lambda]}M\longrightarrow0
\]
gives a connecting map in Tate cohomology, and 
$[Q]\in\widehat H^{-1}(C_p,M)$ maps to $\partial^c_{q,p}$ in
$\widehat H^0(C_p,E_\varpi[\lambda])$.

Our central result is that 
\begin{equation}\label{eq:intro-boundary-nonzero}
 \partial^c_{q,p}\ne O.
\end{equation} 
More precisely, we prove that its reduction at $p$ is non-zero. 
The conductor $3p$ ring-class extension $H_{3p}/K$ is totally ramified
at $p$ with Galois group $C_p$.
 Hence $C_p$ acts trivially on the residue field. 
Consequently, all the conjugates $R^u$ have the same reduction, and the reduced boundary equals 
$
[p+1]\bar R.
$
Then we calculate the Frobenius via the base division point $R_q$ in \S\ref{s:div}.  Write
$R_q^E$ for the corresponding point on $E_{\bar\varpi}$.  Recall that one of the affine
Fermat co-ordinates of $R_q$ lies in $K$ and the other in $tK$, on which
$p^2$-Frobenius acts by $t\mapsto\alpha_q(p)t$. Since $\alpha_q(p)\ne1$, a comparison with the CM-action shows that 
$[p+1]\bar R_q^E\in E_{\bar\varpi}[\lambda]\setminus\{O\}$.

To pass to the reduction of the conductor-$p$ point, let
$\operatorname{Fr}^{\mathrm{rel}}_p:
\bar E_{\bar\varpi}\longrightarrow\bar E_\varpi$
denote the relative $p$-Frobenius between the reduced curves.
 Since $p\equiv2\pmod3$, it conjugates
the CM action.  Hence, letting 
$
 \operatorname{Fr}^{(\lambda)}_p
 =
 \left[\frac{\bar\lambda}{\lambda}\right]
 \circ\operatorname{Fr}^{\mathrm{rel}}_p,
$
one has
$
 [\lambda]\operatorname{Fr}^{(\lambda)}_p
 =
 \operatorname{Fr}^{\mathrm{rel}}_p[\lambda].
$
By \Cref{lem:boundary-frobenius-transfer}, $R$ may be chosen so that
\[
 \bar R=\operatorname{Fr}^{(\lambda)}_p(\bar R_q^E).
\]
Since $\operatorname{Fr}^{(\lambda)}_p$ has degree $p$, it is
injective on the $\lambda$-kernel.  Therefore
\[
 [p+1]\bar R
 =
 \operatorname{Fr}^{(\lambda)}_p([p+1]\bar R_q^E)
 \in E_\varpi[\lambda]\setminus\{O\}, 
\]
concluding the proof of \eqref{eq:intro-boundary-nonzero}.

\begin{remark}
The above phenomenon is reminiscent of a principle in the work of Bertolini--Darmon \cite{BD05}: arithmetic questions regarding Heegner points may be studied geometrically on the special fibre after reduction at $p$.
\end{remark}

\subsubsection{The cubic projector}

We now pass from the non-zero division boundary to the cubic component
of the CM orbit.  Let $H_p^{(3)}$ be the maximal $3$-power
subextension of the conductor $p$ ring class field of $K$, and put
$$
 L=F_qH_p^{(3)},\qquad G=\operatorname{Gal}(L/F_q).
$$

Write $Z$ for the point $\mathcal Z^c\in E_\varpi(L)$ defined in
\eqref{eq:three-primary-trace}. 
 Its norm
$
 S=N_GZ
$
satisfies
$$
 S\in E_\varpi[3](F_q),\qquad
 [\lambda]S=-\partial^c_{q,p}\ne O.
$$
Thus $S\notin E_\varpi[\lambda]$.  Moreover, a cubic descent shows that
$S$ is not $\lambda$-divisible over $F_q$: its descent class is
represented by $\varpi\omega$ or $\varpi\omega^2$, which is non-trivial.

Let $\vartheta:G\to\mu_3$ be the cubic character describing the
action on $\sqrt[3]{p}$, and define
$$
 \Pi=\sum_{\sigma\in G}\vartheta(\sigma)^{-1}\sigma,
 \qquad
 \mathscr D=\frac{N_G-\Pi}{\lambda}\in\mathcal O_K[G].
$$
The operator $\mathscr D$ is integral as 
$\vartheta(\sigma)\equiv1\pmod{\lambda}$.  If $\gamma$ is chosen so
that $\vartheta(\gamma)=\omega$, then
\begin{equation}\label{eq:Koly}
 \lambda\mathscr D=N_G-\Pi,
 \qquad
 (\gamma-1)\mathscr D=\Pi.
\end{equation}
These identities show that the non-zero division boundary survives in the cubic component.
 Indeed, if $\Pi Z=O$, then $\mathscr DZ$ is
$G$-fixed and
$
 [\lambda]\mathscr DZ=S,
$
contradicting the non-divisibility of $S$ over $F_q$.  
Hence
$\Pi Z\ne O$.

The point $\Pi Z$ reduces to $O$ at $p$: the points in the
conductor-$p$ orbit have the same reduction, while the cubic
character weights sum to zero.  
A local argument using the
ramification degrees of $L/F_q$ at primes above $p$ and the height-two supersingular
formal group at $p$ then shows that \(\Pi Z\) cannot be a non-zero torsion point. 
Thus \(\Pi Z\) has infinite order. Finally, the finite CM average defining the auxiliary Heegner point identifies \(\Pi Z\), up to a non-zero rational factor, with \(\operatorname{pr}_{\varpi}(\mathcal P_q)\), concluding the proof of \Cref{thm:auxiliary-heegner}.

\begin{remark}
As with Kolyvagin derivatives, the group ring identity \eqref{eq:Koly} converts the norm relation into a \(\lambda\)-divisibility \cite{Kolyvagin,GrossKolyvagin}, which underlies the above argument. 
\end{remark}

\begin{remark}\label{rem,vis}
In the earlier applications of Heegner's method, including
\cite{Tian14,Yin47}, the approach proceeds by establishing the
indivisibility of a distinguished Heegner point.  Here the relevant
point is the cubic component $\Pi Z$ of a CM orbit whose unweighted
trace vanishes, so no direct analogue of the prior arguments is available.  
The
division boundary supplies a replacement. Indeed, if $\Pi Z=O$, the integral group-ring identities \eqref{eq:Koly} show
that $\mathscr DZ$ is fixed by $G$ and satisfies
$[\lambda]\mathscr DZ=S=N_GZ$.  Thus $S$ would be
$\lambda$-divisible over $F_q$, contrary to the above cubic descent.
We expect this boundary mechanism to
extend beyond the present curves; we will pursue such an extension elsewhere.
\end{remark}

\subsubsection{Explicitness}
A salient feature of our method is that it is constructive\footnote{We are
grateful to Samit Dasgupta for highlighting this and encouraging us to include
\Cref{app:table}.}.  For a fixed $p$, one chooses an auxiliary prime $q$,
computes the division boundary up to sign, determines the relevant
class-group orders, and evaluates the $E_p$-component numerically from
CM data.  The computations in \Cref{app:table} numerically identify
this component with an integer multiple of an explicit non-torsion
rational point on the corresponding cube sum elliptic curve.

\subsection{Vistas}

The methods of this paper suggest several directions for further study.
A natural one is the cube-sum problem for composite integers, beginning with
products of two primes, where several local cubic conditions appear 
simultaneously.  Another is to seek refinements towards the full BSD
formula for these cube-sum elliptic curves.

The division boundary may also admit higher analogues.  It encodes the
first layer of information retained after a vanishing trace.
In conductor
towers, compatible higher $\lambda$-divisions may reveal further
arithmetic information beyond this first layer.

\subsection{Organization}

\Cref{sec:rankin-convolution,sec:complementary-factor} introduce the
auxiliary Rankin--Selberg convolution and establish the non-vanishing
of its complementary central $L$-value.
\Cref{sec:cm-data,sec:cm-kummer,sec:fermat-cover} develop the CM
geometry behind the first $\lambda$-division and the Frobenius
calculation.
\Cref{sec:hecke-projector,sec:cubic-components} construct the division
boundary and use it to prove that the two cubic components are
non-torsion.
\Cref{sec:rankin-gz} identifies these components with the corresponding
projections of the auxiliary Heegner points, and
\Cref{sec:rankin-derivative} completes the proof of \Cref{thm:main}
using the YZZ generalization of the Gross--Zagier formula.
\Cref{app:examples} illustrates the argument with examples for 
$q=13$ and $q=31$: the first makes the division boundary and the cubic
projector explicit, while the second exhibits a non-trivial cube
factor in the cubic descent identity.  \Cref{app:table} then tabulates
the construction for all fourteen primes $p\equiv8\pmod9$ below $500$
and records associated rational-point data.

\subsection{Notation and conventions}\label{sec:reference-conventions}

\subsubsection{General}   
For
$g=\left(\begin{smallmatrix}a&b\\c&d\end{smallmatrix}\right)$ we write
$c(g)=c$ and $d(g)=d$, and
$\operatorname{diag}(a,b)=\left(\begin{smallmatrix}a&0\\0&b\end{smallmatrix}\right)$.

For a number field $F$, $\mathbf A_F$ and $\mathbf A_{F,f}$
denote its adeles and finite adeles. For a prime $\ell$ and an integer $m\ne0$, $\operatorname{ord}_\ell(m)$ is its
$\ell$-adic valuation. 

\subsubsection{\texorpdfstring{$L$}{L}-functions and scalar extensions}
The functions $L(s,E)$ and $L(s,g\times\theta)$ are primitive $L$-functions, normalized with the center at $s=1$.  The unitary normalization used in the Gross--Zagier formula has
center $1/2$, with the shift stated in
\eqref{eq:unitary-arithmetic-shift}. 

 A superscript $0$ on a Hom-space
means tensoring with $\mathbf Q$, and
$A(F)_{\mathbf Q}:=A(F)\otimes_{\mathbf Z}\mathbf Q$ for an abelian variety $A$ over a field $F$.

\subsubsection{The CM field and the three primes}
Throughout,
$K=\mathbf Q(\omega)$ with $\omega^2+\omega+1=0$,
$\mathcal O_K=\mathbf Z[\omega]$, and $\lambda=1-\omega$.  
For $n\ge1$, $\mu_n$ denotes the group of
$n$th roots of unity in a fixed algebraic closure and $\zeta_n$ a chosen
primitive $n$th root of unity; in particular $\mu_3=\{1,\omega,\omega^2\}$.
A bar on an element of $K$ denotes the non-trivial automorphism of
$K/\mathbf Q$.  An element of $\mathcal O_K$ coprime to $3$ is called \emph{primary} when it
is congruent to $1$ modulo $3$.

The fixed and auxiliary primes satisfy
\[
 p\equiv8\pmod9,\qquad q\equiv4\pmod9,\qquad
 q=\varpi\bar\varpi,\qquad \varpi\equiv1\pmod3,
\]
and are always chosen subject to \eqref{eq:intro-ray-condition}.  Put
$\mathfrak q=(\varpi)$, choose $t^3=\varpi$, and write $F_q=K(t)$.
We write
$\omega_q=\omega\bmod\mathfrak q\in\mathbf F_q$ for the residue of
$\omega$ at the chosen prime above $q$.  The cubic residue symbol is normalized by
\[
 \left(\frac a\pi\right)_3\equiv a^{(N\pi-1)/3}\pmod\pi,
\]
with values in $\mu_3=\{1,\omega,\omega^2\}$.  Put 
$\alpha_q(p)=(p/\varpi)_3$, $h_q=(q-1)/3$, and
$D_{q,1}=qp^2$, $D_{q,2}=qp$.

\subsubsection{Curves, the Fermat cubic, and reduction}
For $a\in K^\times$,
\begin{equation}\label{eq:intrinsic-cm-convention}
 E_a:y^2=x^3+a^2/4,
 \qquad [\omega](x,y)=(\omega x,y).
\end{equation}
We write $[\alpha]P$ for the action of
$\alpha\in\mathcal O_K$, and $E[\alpha]$ for its kernel.  For a
field $F\supset K$, $[u]_F$ is the class of $u\in F^\times$ modulo
cubes.

The Fermat cubic used for the first division is
\[
 \mathscr C_a:\qquad V^3-U^3=aW^3,
 \qquad O_{\mathscr C}=[1:1:0].
\]
The degree-three isogeny $\ell_a:\mathscr C_a\to E_a$ is given in
\eqref{eq:fermat-isogeny}, the inverse of the identification
$\jmath_a:E_a\xrightarrow{\sim}\mathscr C_a$ in
\eqref{eq:fermat-identification}, and the associated descent map
$\mathfrak d_a$ in \Cref{lem:exact-descent-kernel}.  The two modular maps are
$\varphi:X_\Gamma\to E_{\bar\varpi}$ and its coefficient conjugate
$\varphi^c:X_\Gamma\to E_\varpi$, both normalized by
$\varphi(\infty)=\varphi^c(\infty)=O$.

At a fixed prime $\mathfrak P$, $\operatorname{sp}_{\mathfrak P}$
denotes specialization.  In a local calculation a bar on a point or
morphism is an abbreviated reduction symbol; a bar on a coefficient
continues to mean conjugation.  We use the full specialization symbol
whenever both operations occur in the same formula.

\subsubsection{Ring-class fields and Galois actions}
The order $\mathcal O_c=\mathbf Z+c\mathcal O_K$ has ring-class field
$H_c$.  For a rational prime $\ell$, put
$K_\ell=K\otimes_{\mathbf Q}\mathbf Q_\ell$ and
$\mathcal O_{K,\ell}=\mathcal O_K\otimes_{\mathbf Z}\mathbf Z_\ell$.

We write $K^{\rm ab}$ for the maximal abelian extension.  For a non-zero
integral modulus $\mathfrak m$, the notation
$K_{\mathfrak m}^{\mathrm{ray}}$ denotes the ray-class field of $K$ of
modulus $\mathfrak m$.  A hat denotes finite adelic points or profinite
completion, according to context.  The Artin maps use arithmetic Frobenius, and for an
unramified primary prime $\mathfrak l=(\pi)\nmid3d$, one has 
\begin{equation}\label{eq:reciprocity-convention}
 \frac{(\sqrt[3]{d})^{\sigma_{\mathfrak l}}}{\sqrt[3]{d}}
 =\left(\frac d\pi\right)_3,
 \qquad \sigma_u=\operatorname{Art}^{\rm ar}_K(u).
\end{equation}
In Shimura reciprocity the finite-idele action is represented by
$u\mapsto\bar u^{-1}=u/N_{K/\mathbf Q}(u)$ modulo rational ideles.

The cyclic groups indexing the local Hecke orbit are
\[
 C_p=\mathbf F_{p^2}^\times/\mathbf F_p^\times,
 \qquad
 C_q=(\mathcal O_K/q\mathcal O_K)^\times/
       (\mathbf Z/q\mathbf Z)^\times\simeq\mathbf F_q^\times.
\]
Thus $|C_p|=p+1$ and $|C_q|=q-1$.  The subgroup
$\Delta_p\subset C_p$ is the image of
$\mathcal O_K^\times/\{\pm1\}\simeq\mu_3$, and
$G_p=C_p/\Delta_p$ has order $(p+1)/3$; see
\eqref{eq:Gp-identification}.  We put
$C_q^{(3)}=(\mathbf F_q^\times)^3$.  When it appears, $C_{p,3}$ denotes the
$3$-primary Sylow subgroup of $C_p$.  The subscripts on
$C_p,C_q,C_{p,3}$ record arithmetic role; abstract cyclic groups are
written $\mathbf Z/n\mathbf Z$.  For a finite group $G$ acting on an additive group of points,
$N_G=\sum_{\gamma\in G}\gamma$ denotes the norm operator, while
$N_GP=\sum_{\gamma\in G}P^\gamma$ is the trace of $P$.

\subsubsection{Characters and Heegner points}
For an idele character $\theta$, the notation
$\theta^\sigma(u)=\theta(\bar u)$ means conjugation of the argument;
$\overline{\theta(u)}$ means conjugation of its value.
For the CM Hecke characters $\nu_m$, we put
$\psi_q=\nu_\varpi$, $\kappa_d=\nu_d\nu_1^{-1}$, and
$\chi_{q,j}=\nu_{p^j}\psi_q^{-1}$ for $j=1,2$.
Their Galois Kummer counterparts are
$\kappa_d^{\rm gal}(\sigma)=(\sqrt[3]{d})^\sigma/\sqrt[3]{d}$, and the
induced cubic characters $\vartheta_{p,j}$ of $G_p$ are defined in
\eqref{eq:ringclass-kummer-labels}.

The auxiliary Heegner points $\mathcal P_j(\Phi)$ are the weighted CM
sums in \eqref{eq:normalized-heegner-point}.  Their
$E_\varpi$-projections are compared with the cubic points
$P_{q,p,j}$ in \eqref{eq:heegner-point-comparison}; the latter are
defined in \eqref{eq:cubic-projector}.  The integral projectors
$\Pi_{j,n_0}$ and divided operators $\mathscr D_{j,n_0}$ are given in
\eqref{eq:integral-projector-operators}.

\subsubsection{Frobenius conventions}
Four Frobenius operations occur in the text:
\begin{center}
\begin{tabular}{@{}p{.27\textwidth}p{.62\textwidth}@{}}
\toprule
notation & meaning\\
\midrule
$\operatorname{Fr}_{p^2}$
& The $p^2$-power action on co-ordinates after reduction.\\
$[-p\alpha_q(p)]$
& The same operation on $E_{\bar\varpi}$, expressed as the CM
endomorphism $[-p\alpha_q(p)]$.\\
$\operatorname{Fr}^{\mathrm{rel}}_p$
& Relative $p$-Frobenius from the reduction of
$E_{\bar\varpi}$ to the reduction of $E_\varpi$.\\
$\operatorname{Fr}^{(\lambda)}_p$
& The $\lambda$-compatible normalization
$[\bar\lambda/\lambda]\operatorname{Fr}^{\mathrm{rel}}_p$ of relative
$p$-Frobenius, satisfying
$[\lambda]\operatorname{Fr}^{(\lambda)}_p
=\operatorname{Fr}^{\mathrm{rel}}_p[\lambda]$.\\
\bottomrule
\end{tabular}
\end{center}

\subsection*{Acknowledgements} The authors are grateful to James Sylvester, Ernst Selmer, Bryan Birch, and Nelson Stephens for the inspiring conjecture bearing Sylvester's name. They also thank Hongbo Yin for a valuable discussion and acknowledge
the influence of his beautiful work \cite{YinGZ}. Finally, they thank Samit Dasgupta and Christopher Skinner for helpful comments on the preprint. 

During the
preparation of this paper, A.B. was partially supported by NSF grant
DMS-2302064; Y.T. was partially supported by the National Natural Science
Foundation of China grant no.~12288201.

\subsection*{Use of artificial intelligence}
ChatGPT and Claude tools were used for exploratory 
computations, editing and reorganization. 

\part{The auxiliary Rankin--Selberg convolution}

\section{The Rankin--Selberg factorization}\label{sec:rankin-convolution}

This section constructs the auxiliary Rankin--Selberg convolutions used
in the proof.  We first apply Chebotarev to choose a prime $q$ satisfying
the cubic-residue condition \eqref{eq:intro-ray-condition}
(cf.~\Cref{lem:chebotarev}).  We then define the two ring-class characters
corresponding to $E_p$ and $E_{p^2}$.  The main result,
\Cref{prop:rankin-factorization}, factors each resulting self-dual
Rankin--Selberg $L$-function as $L(s,E_{p^j})$ times a
complementary $L$-function.

Fix $p\equiv8\pmod9$.

\subsection{Choosing the prime \texorpdfstring{$q$}{q}}
We choose $q$ using Chebotarev's theorem.

\begin{lemma}\label{lem:chebotarev}
The set of rational primes $q$ satisfying the following two conditions
has natural density $1/9$:
\begin{enumerate}[label=\textup{(\roman*)}]
\item $q\equiv4\pmod9$;
\item for either of the two conjugate primary factors $\varpi\mid q$,
\(\left(\frac p\varpi\right)_3\ne1\).
\end{enumerate}
\end{lemma}

\begin{proof}
Put $M_p=K(\sqrt[3]p)$ and $K_9=K(\zeta_9)$.  Since
$-p\in1+9\mathbf Z_3$ and
$(1+3\mathbf Z_3)^3=1+9\mathbf Z_3$, one has
$p\in K_\lambda^{\times3}$.  Thus $M_p/K$ is unramified at $\lambda$ and
ramified above $p$, whereas $K_9/K$ is ramified only at $\lambda$; hence
$M_p\cap K_9=K$ and
\[
 \operatorname{Gal}(M_pK_9/K)\simeq (\mathbf Z/3\mathbf Z)^2,
 \qquad |\operatorname{Gal}(M_pK_9/\mathbf Q)|=18.
\]
For $\mathfrak q=(\varpi)\nmid3p$, Kummer reciprocity gives
\begin{equation}\label{eq:chebotarev-kummer-character}
 \frac{\operatorname{Frob}_{\mathfrak q}(\sqrt[3]p)}{\sqrt[3]p}
 =\left(\frac p\varpi\right)_3=\alpha_q(p).
\end{equation}

Inside $\operatorname{Gal}(M_pK_9/K)\simeq (\mathbf Z/3\mathbf Z)^2$, consider the two
elements whose restriction to $K_9$ is
$\zeta_9\mapsto\zeta_9^4$ and whose restriction to $M_p/K$ is non-trivial.
Complex conjugation exchanges these two elements and fixes their $K_9$
component, so they form one conjugacy class in
$\operatorname{Gal}(M_pK_9/\mathbf Q)$.  Chebotarev therefore gives density
$2/18=1/9$.  The corresponding rational primes satisfy $q\equiv4\pmod9$.
The two primes above $q$ have inverse non-trivial Frobenius in $M_p/K$, so
both cubic residue symbols are non-trivial by
\eqref{eq:chebotarev-kummer-character}.
\end{proof}

Fix such a prime $q$ and an ordered factorization
$q=\varpi\bar\varpi$ with $\varpi$ primary.

\subsection{The ring-class characters \texorpdfstring{$\vartheta_{p,j}$}{theta p,j}}
We isolate the two order-three ring-class characters arising from the
Kummer characters of $p$ and $p^2$.  Since $p\equiv2\pmod3$, the prime
$p$ is inert in $K$; the ring-class group of conductor $p$ is cyclic.

Let
\[
 \mathcal O_p=\mathbf Z+p\mathcal O_K.
\]
Since $\mathcal O_K/(p)\simeq\mathbf F_{p^2}$ and $\operatorname{Pic}(\mathcal O_K)=1$, the ring-class exact
sequence gives
\begin{equation}
 1\longrightarrow \mathcal O_K^\times/\mathcal O_p^\times
 \longrightarrow \mathbf F_{p^2}^\times/\mathbf F_p^\times
 \longrightarrow \operatorname{Pic}(\mathcal O_p)\longrightarrow 1.
\end{equation}
Here $\mathcal O_p^\times=\{\pm1\}$, so
$\mathcal O_K^\times/\mathcal O_p^\times\simeq\mu_3$.  We henceforth use the Artin
isomorphism to identify
\begin{equation}\label{eq:Gp-identification}
 G_p:=\operatorname{Pic}(\mathcal O_p)\simeq\operatorname{Gal}(H_p/K)
 \simeq\bigl(\mathbf F_{p^2}^\times/\mathbf F_p^\times\bigr)/\mu_3,
 \qquad |G_p|=\frac{p+1}{3},
\end{equation}
where $H_p$ is the ring-class field of conductor $p$.  In particular
$G_p$ is cyclic.  The same local description shows that the unique prime
of $K$ above $p$ is totally ramified in $H_p/K$.

For $d\in K^\times$, write
$\kappa_d^{\rm gal}(\sigma)=(\sqrt[3]{d})^\sigma/\sqrt[3]{d}$.
The two characters of exact order three that we need arise from this
Kummer construction.  Since $p\equiv-1\pmod9$, one has
$-p\in1+9\mathbf Z_3=(1+3\mathbf Z_3)^3$, so $p$ is already a cube in $\mathbf Q_3^\times$
and hence in $K_\lambda^\times$.  Thus the cubic Kummer character
$\kappa_p^{\rm gal}$ is trivial at the prime above $3$.

At $p$ its tame local character is trivial on
$\mathbf F_p^\times$ and on $1+p\mathcal O_{K,p}$.  It is also trivial on the residual
$\mu_3$: indeed $9\mid p+1$ implies $3\mid(p^2-1)/3$, so every
$\zeta\in\mu_3$ satisfies $\zeta^{(p^2-1)/3}=1$.  Hence local class field
theory makes the character factor through the ring-class quotient
\eqref{eq:Gp-identification}.
Consequently $K(\sqrt[3]{p})\subset H_p$.  
\begin{definition}
We define the induced
characters on $G_p=\operatorname{Gal}(H_p/K)$ by
\begin{equation}\label{eq:ringclass-kummer-labels}
 \vartheta_{p,j}(\sigma)
 =\frac{(\sqrt[3]{p^j})^\sigma}{\sqrt[3]{p^j}},
 \qquad \sigma\in G_p,\quad j=1,2.
\end{equation}
\end{definition}
Both have exact order three, since $K(\sqrt[3]{p})/K$ is totally
ramified at the prime above $p$.  As $G_p$ is cyclic,
$K(\sqrt[3]{p})$ is its unique cubic subextension, and the kernels of
$\vartheta_{p,1}$ and $\vartheta_{p,2}$ are its unique index-three
subgroup.

\subsection{Factorization of the Rankin--Selberg convolution}
For $m\in K^\times$, let $\nu_m$ be the CM Hecke character of $E_m$
normalized at an unramified primary prime $\mathfrak l=(\pi)$ by
\[
 \nu_m(\mathfrak l)=\overline{\left(\frac m\pi\right)_3}\,\pi.
\]
The theta series $g_q=\theta(\psi_q)$, where $\psi_q=\nu_\varpi$,
is the weight-two eigenform used in the modular maps.  Its associated
abelian variety $A_q/\mathbf Q$ has dimension two and is of
$\mathrm{GL}_2$-type.

Write
\[
 \kappa_d:=\nu_d\nu_1^{-1}
\]
for the finite cubic Hecke character associated with $d\in K^\times$;
at arithmetic Frobenius it has the inverse value to the Galois
Kummer character $\kappa_d^{\rm gal}$.  For $j=1,2$ set
\begin{equation*}
 \chi_{q,j}=\nu_{p^j}\psi_q^{-1}
 =\kappa_{p^j}\kappa_\varpi^{-1},
\end{equation*}
and put $D_{q,1}=qp^2$, $D_{q,2}=qp$.

\begin{proposition}\label{prop:rankin-factorization}

For $j=1,2$ one has an equality of primitive $L$-functions
\begin{equation}\label{eq:rankin-factorization}
 L(s,g_q\times\chi_{q,j})
 =L(s,E_{p^j})L(s,E_{D_{q,j}}).
\end{equation}
\end{proposition}

\begin{proof}
Automorphic induction gives
\[
 L(s,g_q\times\chi_{q,j})
 =L(s,\psi_q\chi_{q,j})L(s,\psi_q^\sigma\chi_{q,j}).
\]
The first character is \(\nu_{p^j}\).  At an unramified primary prime
\(\mathfrak l=(\pi)\), the chosen normalization gives
\[
 \overline{\psi_q^\sigma\chi_{q,j}}(\mathfrak l)
 =\left(\frac q\pi\right)_3^{-1}
  \left(\frac{p^j}\pi\right)_3\pi
 =\left(\frac{qp^{2j}}\pi\right)_3^{-1}\pi.
\]
Thus $\overline{\psi_q^\sigma\chi_{q,j}}$ is $\nu_{qp^2}$ for $j=1$
and $\nu_{qp}$ for $j=2$, the factor $p^3$ being a cube in the latter
case.  Equality at the unramified primary primes identifies the global
Hecke characters, so their ramified Euler factors also agree.  Since
$D_{q,j}\in\mathbf Q^\times$, one has
$\overline{\nu_{D_{q,j}}}=\nu_{D_{q,j}}^\sigma$; reindexing ideals by
complex conjugation shows that the second constituent contributes the
$L$-function of $E_{D_{q,j}}$.  This proves
\eqref{eq:rankin-factorization}.
\end{proof}

The Rankin--Selberg pair is self-dual in the sense of \cite[Theorem~1.2]{YZZ} (cf.\
\cite[Theorem~1.5(1)]{CST14}).  More precisely, its
central-character identity is
\begin{equation}\label{eq:rankin-self-dual}
 \omega_{g_q}\,\chi_{q,j}|_{\mathbf A_{\mathbf Q}^{\times}}=1,
\end{equation}
where $\omega_{g_q}$ is the central character of $g_q$.  Indeed, after
unitary normalization, \eqref{eq:rankin-self-dual} follows by taking the
quotient of the central-character identities for the automorphic
inductions of $\psi_q$ and $\nu_{p^j}$.  

We continue to use the arithmetic normalization of the preceding Hecke
characters, so the central point is $s=1$.
The restriction of the local character $\chi_{q,j,p}$ to
$\mathcal O_{K,p}^{\times}$ factors through
$G_p=C_p/\Delta_p$ and induces $\vartheta_{p,j}^{-1}$ there.
Indeed, $\psi_q$ is unramified at $p$, so this restriction is the local
Kummer character of $p^j$ with the inverse-value convention above.  We
use $\chi_{q,j}|_{G_p}$ as shorthand for this induced
local character; the global Hecke character $\chi_{q,j}$ need not
factor through $\operatorname{Gal}(H_p/K)$.

\section{The complementary rank-zero factor}\label{sec:complementary-factor}
For elliptic curves $E_p$ and $E_{p^2}$, the factorization in
\Cref{prop:rankin-factorization} gives one complementary elliptic
$L$-function.  We prove that these complementary curves have
rank zero and non-zero central $L$-value (cf.~
\Cref{prop:complementary-rank-zero}).  Chan's explicit $3$-isogeny
descent \cite[\S\S2.2--2.3 and Theorem~1.1]{Chan} reduces the algebraic
claim to the matrices computed in \Cref{lem:chan-matrices}.
The authors'
CM rank-zero converse \cite[Theorem~3.1]{BTconverse} then implies the
non-vanishing of the central $L$-values.

Let $q=\varpi\bar\varpi$ be as in \Cref{sec:rankin-convolution}, so that
\eqref{eq:intro-ray-condition} holds, and let $E_{D_{q,j}}$ denote the
complementary twist in \Cref{prop:rankin-factorization}.

Put
\[
 d_1=2,\qquad d_2=1,
 \qquad D_{q,j}=qp^{d_j},
 \qquad m_j=4D_{q,j}.
\]
Write $\mathfrak{cf}(m)$ for the positive cube-free part of a positive
integer $m$.  The exponents of $2,q,p$ in $m_j$ are $2,1,d_j$, so
$m_j$ is cube-free, and
\[
 \mathfrak{cf}(2m_j)=D_{q,j}.
\]

\subsection{The descent matrix}
Following Chan \cite[\S\S2.2--2.3]{Chan}, let $R_j$ be the
cubic-residue matrix $R$ attached to $n=m_j$, and let $R_j^T$ denote its
transpose.  The rows of $R_{j}^T$ are indexed by the primes $q,p$, and its columns
by the unit $\omega$ and the chosen prime $\varpi$ above $q$.  Chan
identifies $\ker R_j^T$ with
$\operatorname{Sel}_{\phi_{\rm Ch}}(E_{m_j}^{\rm Ch})$.

\begin{lemma}\label{lem:chan-matrices}
Let
\[
 m_j=4qp^{d_j},\qquad d_1=2,\quad d_2=1,
\]
and write
\[
 \alpha_q(p)=\left(\frac p\varpi\right)_3=\omega^\beta,
 \qquad b_j=d_j\beta\in\mathbf F_3.
\]
For the rational $3$-isogeny
\[
 \phi_{\rm Ch}:E_{m_j}^{\rm Ch}\longrightarrow
 \widehat E_{m_j}^{\rm Ch},
 \qquad E_{m_j}^{\rm Ch}:y^2=x^3+m_j^2,
\]
one has
\begin{equation}\label{eq:chan-matrix-derived}
R_j^T=\begin{pmatrix}2&2b_j\\0&b_j\end{pmatrix}
\quad\text{over }\mathbf F_3.
\end{equation}
The rows are indexed by $(q,p)$ and the columns by
$(\pi_0,\pi_1)=(\omega,\varpi)$.  Since
$\alpha_q(p)\ne1$, one has $b_j\ne0$ for $j=1,2$, and so
\[
 \operatorname{Sel}_{\phi_{\rm Ch}}(E_{m_j}^{\rm Ch})=0,
 \qquad
 \dim_{\mathbf F_3}
 \operatorname{Sel}_{\widehat\phi_{\rm Ch}}
 (\widehat E_{m_j}^{\rm Ch})=1.
\]

\end{lemma}

\begin{proof}
In Chan's notation, her $\rho$ is our $\omega$.  Since $\operatorname{ord}_3(m_j)=0$ and
$\mathfrak{cf}(2m_j)=qp^{d_j}$, the matrix $R_j^T$ has the two rows
$q,p$ and the two columns $\pi_0=\omega$, $\pi_1=\varpi$.
The local condition at $3$ is already encoded in Chan's matrix.

For the unit column, multiplicativity and the supplementary law for the
cubic residue symbol give
\[
 \left(\frac{\omega}{q}\right)_3
 =\left(\frac{\omega}{\varpi}\right)_3
  \left(\frac{\omega}{\bar\varpi}\right)_3=\omega^2
\]
since $q\equiv4\pmod9$.  On the other hand
\[
 \left(\frac{\omega}{p^{d_j}}\right)_3=1,
\]
since $p\equiv8\pmod9$.  Thus the first column is $(2,0)^T$.
For the $\varpi$ column, the diagonal $q$-entry and the off-diagonal
$p$-entry are
\[
 \left(\frac{2m_j/q}{\varpi}\right)_3^2
 =\left(\frac{8p^{d_j}}\varpi\right)_3^2=\omega^{2b_j},
 \qquad
 \left(\frac{p^{d_j}}\varpi\right)_3=\omega^{b_j};
\]
here $8=2^3$ contributes trivially.  This gives
\eqref{eq:chan-matrix-derived}.  Since $d_j\not\equiv0\pmod3$ and
$\beta\ne0$, the matrix is invertible, so
$\operatorname{Sel}_{\phi_{\rm Ch}}(E_{m_j}^{\rm Ch})=0$.

Let $\omega_2(d)$ count the distinct prime divisors of $d$ that are
congruent to $2$ modulo $3$.  By \cite[Theorem~1.1]{Chan},
\[
 \dim\operatorname{Sel}_{\widehat\phi_{\rm Ch}}
 =\dim\operatorname{Sel}_{\phi_{\rm Ch}}
  +\omega_2(\mathfrak{cf}(2m_j))+\delta_{m_j}.
\]
Her correction term is
\[
\delta_m=
\begin{cases}
1,&m\equiv\pm3\pmod9,\\
-1,&m\equiv\pm4\pmod9,\\
0,&\text{otherwise}.
\end{cases}
\]
Here $m_1\equiv-2$ and $m_2\equiv2\pmod9$, so
$\delta_{m_j}=0$.  Since $\mathfrak{cf}(2m_j)=D_{q,j}$, and among
$q,p$ only $p\equiv2\pmod3$, one has
$\omega_2(D_{q,j})=1$.  The dual Selmer dimension is therefore one.
\end{proof}

\subsection{Non-vanishing of central \texorpdfstring{$L$}{L}-values}
The two isogeny Selmer groups control the full $3$-Selmer group, and the
Kummer class of the rational point $(0,m_j)$ generates it.

\begin{proposition}\label{prop:complementary-rank-zero}

For $j=1,2$,
\[
 \operatorname{rank}E_{D_{q,j}}(\mathbf Q)=0,
 \qquad
 \Sha(E_{D_{q,j}}/\mathbf Q)[3^\infty]=0,
 \qquad
 L(1,E_{D_{q,j}})\ne0.
\]

\end{proposition}

\begin{proof}
The standard comparison sequence for
$[3]=\widehat\phi_{\rm Ch}\circ\phi_{\rm Ch}$ contains
\[
 \operatorname{Sel}_{\phi_{\rm Ch}}(E_{m_j}^{\rm Ch})
 \longrightarrow \operatorname{Sel}_3(E_{m_j}^{\rm Ch})
 \longrightarrow
 \operatorname{Sel}_{\widehat\phi_{\rm Ch}}(\widehat E_{m_j}^{\rm Ch}).
\]
By the matrix calculation in \Cref{lem:chan-matrices},
\[
 \dim_{\mathbf F_3}\operatorname{Sel}_3(E_{m_j}^{\rm Ch})\le1.
\]
The point $P_j=(0,m_j)$ has order three.  To see that it is not divisible
by $3$, reduce at the prime $5$.  The congruence
conditions on $p$ and $q$ ensure $5\nmid m_j$, so the curve has good
reduction at $5$, and a direct count gives
$\#E_{m_j}^{\rm Ch}(\mathbf F_5)=6$.  Since $5\ne3$, reduction is
injective on rational $3$-torsion.  Thus the reduction of $P_j$ still
has order three, whereas $[3]E_{m_j}^{\rm Ch}(\mathbf F_5)$ has order two.
Hence $P_j\notin3E_{m_j}^{\rm Ch}(\mathbf Q)$, so its Kummer class is a
non-zero element of the $3$-Selmer group.  Therefore
\[
 \dim_{\mathbf F_3}\operatorname{Sel}_3(E_{m_j}^{\rm Ch})=1.
\]
The Kummer exact sequence then gives
\[
 \operatorname{rank}E_{m_j}^{\rm Ch}(\mathbf Q)=0,
 \qquad
 \Sha(E_{m_j}^{\rm Ch}/\mathbf Q)[3]=0.
\]
Any non-zero $3$-primary torsion group has non-zero $3$-torsion, so in fact
$\Sha(E_{m_j}^{\rm Ch}/\mathbf Q)[3^\infty]=0$.  Since $m_j=4D_{q,j}$,
the change of variables $x=4X$, $y=8Y$ identifies $E_{m_j}^{\rm Ch}$ with
$E_{D_{q,j}}$.  Let $T_3(E)$ be the $3$-adic Tate module and put
$V_3(E)=T_3(E)\otimes_{\mathbf Z_3}\mathbf Q_3$.  The preceding
equalities imply the vanishing of the Bloch--Kato Selmer group:
\[
 H_{\rm f}^1(\mathbf Q,V_3(E_{D_{q,j}}))=0.
\]
The authors' rank-zero converse for CM newforms
\cite[Theorem~3.1]{BTconverse} applies at an arbitrary prime, in particular
at $3$, and therefore gives
\[
 \operatorname{ord}_{s=1}L(s,E_{D_{q,j}})=0.
\]
Thus $L(1,E_{D_{q,j}})\ne0$.
\end{proof}

\part{CM points and the division boundary}

\section{CM points and modular parametrizations}\label{sec:cm-data}

This section introduces the modular curves, CM points, and modular parametrizations used below. 
 We first identify the cubic field over which the two base CM
points are defined and determine their Galois actions
(cf.~\Cref{lem:cubic-field,lem:fine-base-field}).  We then describe some basic properties of 
the normalized modular parametrizations
(cf.~\Cref{lem:modular-functions}).  These results will be used in the next
section to choose a CM point whose $E_\varpi$-image is non-zero
$\lambda$-torsion and to compute the cubic descent class of its
$E_{\bar\varpi}$-image.

Fix $q=\varpi\bar\varpi$ and $\mathfrak q=(\varpi)$ as in
\Cref{sec:rankin-convolution}, and put $N=9q$.  Let
$\Gamma\subset\Gamma_0(N)$ be the index-three subgroup defined by the
condition that the lower-right entry modulo $q$ be a cube, and let
$X_\Gamma/\mathbf Q$ be the associated connected modular curve 
(cf.~\Cref{subsubsec:level-structure}).
We use the normalized modular parametrizations 
\[
 \varphi:X_\Gamma\longrightarrow E_{\bar\varpi},\qquad
 \varphi^c:X_\Gamma\longrightarrow E_\varpi
\]
of
\cite[(2.0.8)--(2.0.9)]{Yin47}
specialized to exponent one.
Here $\varphi$ is attached to $g_q=\theta(\nu_\varpi)$ and
$\varphi^c$ is its coefficient conjugate; both send the cusp $\infty$
to $O$.

\subsection{The cubic field \texorpdfstring{$F_q$}{Fq}}
Recall that $K_{\mathfrak m}^{\mathrm{ray}}$ denotes the ray-class
field of modulus $\mathfrak m$.
Put
\[
 F_q=K(t),\qquad t^3=\varpi.
\]
Let $u_q\in\mathbf A_{K,f}^{\times}$ be equal to $\omega$ at
$\mathfrak q$ and to $1$ at every other finite place.

\begin{lemma}\label{lem:cubic-field}
\begin{enumerate}[label=\textup{(\roman*)}]
\item The extension $F_q/K$ is ramified at $\mathfrak q$ and at
$\lambda=1-\omega$, unramified away from $3\mathfrak q$, and
\[
 F_q\subset K_{3\varpi}^{\mathrm{ray}}
       \subset K_{9q}^{\mathrm{ray}}.
\]
\item With the Artin normalization of \Cref{sec:reference-conventions},
\[
 t^{\sigma_{u_q}}=\omega^2t,\qquad
 t^{\sigma_{u_q^{-1}}}=\omega t.
\]
In particular, $\sigma_{u_q}$ generates $\operatorname{Gal}(F_q/K)$.
\end{enumerate}
\end{lemma}

\begin{proof}
Kummer theory gives ramification at $\mathfrak q$ and unramifiedness at
all primes away from $3\mathfrak q$.  Since the degree-three ramification
at $\mathfrak q$ is tame, its conductor exponent is one.  If $F_q/K$ were
unramified at $\lambda$, its conductor would divide $\mathfrak q$ and
$F_q$ would lie in the $\mathfrak q$-ray class field.  As $h_K=1$, that
field has degree
\[
 \frac{\#(\mathcal O_K/\mathfrak q)^\times}{\#\mathcal O_K^\times}
 =\frac{q-1}{6},
\]
which is prime to $3$ for $q\equiv4\pmod9$.  This proves the ramification
at $\lambda$.  The displayed ray-class containment is
\cite[Proposition~2.1(2)]{Yin47}.

For (ii), the tame unit--uniformizer formula for the local cubic norm
residue symbol gives
\[
 (u,\varpi)_{\mathfrak q,3}
 =\operatorname{red}_{\mathfrak q}(u)^{-(q-1)/3}.
\]
Since $(q-1)/3\equiv1\pmod3$, evaluation at $u=\omega$ gives the
first formula, and in turn the second follows.
\end{proof}

\subsection{The modular curve and its CM points}
\subsubsection{The level structure}\label{subsubsec:level-structure}
Let
\[
\mathcal U=\left\{g\in\operatorname{GL}_2(\widehat{\mathbf Z}):
 c(g)\in N\widehat{\mathbf Z},\quad
 d(g)\bmod q\in(\mathbf F_q^\times)^3\right\}.
\]
The corresponding arithmetic subgroup is
\[
 \Gamma=\left\{\gamma\in\Gamma_0(N):
 d(\gamma)\bmod q\in(\mathbf F_q^\times)^3\right\}.
\]
This is the level structure of \cite[(2.0.7) and \S3]{Yin47}.  Since
$\det(\mathcal U)=\widehat{\mathbf Z}^{\times}$, the associated
connected modular curve has a canonical model over $\mathbf Q$ and we let $X_\Gamma$ denote this model
\cite[\S\S5, 12--13]{MilneSV}.
The map
\[
 \Gamma_0(N)\longrightarrow \mathbf F_q^\times/(\mathbf F_q^\times)^3,
 \qquad
 \begin{pmatrix}a&b\\ c&d\end{pmatrix}\longmapsto d
\]
is surjective with kernel $\Gamma$ and so 
$[\Gamma_0(N):\Gamma]=3$.

The non-cuspidal complex points have the adelic description
\[
 \operatorname{GL}_2(\mathbf Q)\backslash
 \bigl(\mathfrak H^{\pm}\times\operatorname{GL}_2(\mathbf A_f)\bigr)
 /\mathcal U,
 \qquad \mathfrak H^{\pm}=\mathbf C\setminus\mathbf R.
\]
We write $[z,g]$ for the equivalence class of the pair $(z,g)$; thus
$[z,g]=[\gamma z,\gamma gu]$ for
$\gamma\in\operatorname{GL}_2(\mathbf Q)$ and $u\in\mathcal U$.
Right translation on the modular-curve tower is
$[z,g]\mapsto[z,gu]$ for $u\in\operatorname{GL}_2(\mathbf A_f)$.

\subsubsection{Two CM embeddings}\label{subsec:two-cm-embeddings}
We introduce two CM embeddings $\iota_a:K\hookrightarrow M_2(\mathbf Q)$,
$a=0,1$, whose fixed points give the two CM points used below; the second
is obtained from the dual cyclic isogeny by the Atkin--Lehner involution.

Choose $r\in\mathbf Z$ such that
\[
 r^2-r+1\equiv0\pmod{3q},\qquad -r\equiv\omega\pmod\varpi,
\]
and put
\begin{equation}\label{eq:section2-r-choice}
\begin{aligned}
 s_r&=r^2-r+1,&
 \tau_0&=-\frac1{3(\omega+r)},\\
 \tau_1&=W_N\left(-\frac1{3(\omega+1-r)}\right)
       =\frac{\omega+1-r}{3q},&
 W_N&=\begin{pmatrix}0&-1\\N&0\end{pmatrix}.
\end{aligned}
\end{equation}
Define $\iota_a$ by $\iota_a(\omega)=I_a$, where
\[
I_0=
\begin{pmatrix}-r&-1/3\\3s_r&r-1\end{pmatrix},\qquad
I_1=
\begin{pmatrix}-r&-s_r/(3q)\\3q&r-1\end{pmatrix}.
\]
Also put
\[
\mathcal A_3=\begin{pmatrix}1&1/3\\0&1\end{pmatrix},\qquad
B_3=\begin{pmatrix}1&0\\3q&1\end{pmatrix},\qquad
S_0=\mathcal A_3,\quad S_1=B_3^2.
\]

The fixed point $\tau_a$ determines the CM type of $\iota_a$:
equivalently, the line spanned by $(\tau_a,1)^{\mathsf T}$ is its
$\omega$-eigenline, as recorded in part~\textup{(i)} below.

\begin{lemma}\label{lem:two-cm-orders}
For $a=0,1$, the following hold.
\begin{enumerate}[label=\textup{(\roman*)}]
\item One has $I_a^2+I_a+1=0$ and
\[
 I_a\binom{\tau_a}{1}=\omega\binom{\tau_a}{1}.
\]
\item For the lattice $\mathbf Z\tau_a+\mathbf Z$,
\[
\operatorname{End}(\mathbf C/(\mathbf Z\tau_0+\mathbf Z))=\mathcal O_3,
\qquad
\operatorname{End}(\mathbf C/(\mathbf Z\tau_1+\mathbf Z))=\mathcal O_K.
\]
If
$\mathcal R_0(N)=\{g\in M_2(\mathbf Z):N\mid c(g)\}$, then
\[
 \iota_a(K)\cap\mathcal R_0(N)=\iota_a(\mathcal O_3).
\]
\item Let $p\equiv2\pmod3$, $p\nmid3q$.  After conjugating the local
embedding at $p$ by $\operatorname{diag}(p,1)$, the endomorphism orders\footnote{In both cases, the suborder of
endomorphisms preserving the chosen cyclic subgroup of order $N$ is
$\mathcal O_{3p}$.}
 of the underlying elliptic curves are $\mathcal O_{3p}$ for $a=0$ and
$\mathcal O_p$ for $a=1$.\end{enumerate}
\end{lemma}

\begin{proof}
The identities in (i) are immediate from the displayed matrices.  For
(ii), an integral element of the torus is $a+b\omega$ with
$a,b\in\mathbf Z$.  For $I_0$ its upper-right entry is $-b/3$, so
$3\mid b$; once this holds, the level condition is automatic because
$3q\mid s_r$.  For $I_1$ all entries are integral, while the condition
$N\mid c(a+bI_1)=3qb$ is equivalent to $3\mid b$.  This gives the
asserted orders.

After conjugation by $\operatorname{diag}(p,1)$, the lower-left entry is
$3s_rb/p$ or $3qb/p$.  Since $p\nmid3q s_r$---the last assertion follows
because $X^2-X+1$ has no root modulo $p\equiv2\pmod3$---integrality
forces $p\mid b$.  For $a=0$, the upper-right entry remains $-bp/3$, so
integrality also requires $3\mid b$ and gives $\mathcal O_{3p}$; the
level condition is then automatic.  For $a=1$, all other entries are
integral once $p\mid b$, which gives $\mathcal O_p$.  The level
condition $N\mid 3qb/p$ imposes the additional requirement
$3\mid b/p$, and hence cuts out $\mathcal O_{3p}$.
\end{proof}

\begin{remark}\label{rem:yin-conventions}
Our labels differ from those of \cite{Yin47}.  Hongbo Yin denotes by $\varpi$
the primary factor satisfying $-r\equiv\omega^2\pmod\varpi$, whereas
we denote its conjugate by $\varpi$, so that
$-r\equiv\omega\pmod\varpi$.  Thus our pair $(\varpi,\varphi)$
corresponds to $(\bar\varpi,\varphi^c)$ in that paper.  With this
choice, the $\varphi^c$-branch supplies the non-zero \(\lambda\)-torsion value used to choose the base point 
and the $\varphi$-branch supplies the point to be divided; see
\Cref{lem:base-branches} and \Cref{prop:primary-non-vanishing}.  Once the
exchange of labels and our arithmetic Artin convention are taken into
account, the factor $\omega$ in
\cite[Proposition~2.1(3)]{Yin47} becomes the factor $\omega^2$ in
\Cref{lem:cubic-field}(ii).  Reversing the Artin convention inverts both
the action on $t$ and the action on $Q^{(a)}$.  Thus the conclusion
$t^2x(Q^{(a)}),y(Q^{(a)})\in K$ is unchanged; in particular, the two
$K$-lines in \eqref{eq:normalized-lift-lines} are
convention-independent.
\end{remark}

\subsubsection{Fields of definition and Galois actions}\label{subsec:cm-reciprocity}
Let $\delta_3\in\mathbf A_{K,f}^\times$ be equal to $\omega$ at the
prime above $3$ and to $1$ elsewhere.

\begin{lemma}\label{lem:fine-base-field}
For $a=0,1$, $g\in\operatorname{GL}_2(\mathbf A_f)$, and
$u\in\mathbf A_{K,f}^\times$, one has
\begin{equation}\label{eq:reflex-reciprocity}
 [\tau_a,g]^{\sigma_u}
 =[\tau_a,\iota_a(\bar u^{-1})g].
\end{equation}
Moreover,
\begin{enumerate}[label=\textup{(\roman*)}]
\item $[\tau_a,1]\in X_\Gamma(F_q)$;
\item $\sigma_{\delta_3}$ acts on $[\tau_a,1]$ through $S_a$, and
$\sigma_{u_q}$ acts through $S_a^{-1}$.
\end{enumerate}
\end{lemma}

\begin{proof}
For the special point $\tau_a$, the reflex map is
$u\mapsto\iota_a(\bar u^{-1})$.  Formula
\eqref{eq:reflex-reciprocity} is the canonical-model reciprocity law 
(see \cite[Example~12.7 and Definition~12.8, especially (54)]{MilneSV}).

Under this reciprocity action, the stabilizer consists of the ideles
$u$ for which
$[\tau_a,\iota_a(\bar u^{-1})]=[\tau_a,1]$.  It is the open subgroup
\[
K^\times(\mathbf Z_3+3\mathcal O_{K,3})^\times
(\mathcal O_{K,\mathfrak q}^{\times})^3
\prod_{\mathfrak l\nmid3\mathfrak q}\mathcal O_{K,\mathfrak l}^\times.
\]
Indeed, at $3$ the condition is $3\mid b$ for both embeddings; at
$\mathfrak q$ the lower-right cube condition is the cube condition on
the first factor of
$K\otimes\mathbf Q_q=K_{\mathfrak q}\times K_{\bar{\mathfrak q}}$;
all other local stabilizers are maximal.  The local quotients at $3$ and
$\mathfrak q$ each have order three, while the diagonal image of
$\mathcal O_K^\times/\{\pm1\}\simeq\mu_3$ identifies them.  Since
$h_K=1$, the corresponding field has degree three over $K$.

By \Cref{lem:cubic-field}, the Kummer character of $F_q/K$ is trivial on
the displayed stabilizer.  Hence the field of definition is $F_q$.
Finally, direct multiplication gives $S_aI_a\in\Gamma_0(N)$ and
$S_a\in\mathcal U_v$ for $v\ne3$, which yields the $\delta_3$-action.
The principal idele $\omega$ fixes the CM point, so the action supported at $q$
is the inverse one, giving the action of $u_q$.
\end{proof}

\subsection{Modular maps and cube-root functions}
Let $g_q$ be the normalized weight-two CM eigenform attached to
$\varphi$, and let $g_q^c$ be its coefficient conjugate.  Let
$L_g,L_g^c$ be the corresponding period lattices on $\Gamma_1(N)$, and
write $Y,Y^c$ for the $y$-co-ordinate functions on
$E_{\bar\varpi},E_\varpi$.

\begin{lemma}\label{lem:modular-functions}
\begin{enumerate}[label=\textup{(\roman*)}]
\item For every $\gamma\in\Gamma_0(N)$, one has
\[
2\pi i\int_\infty^{\gamma\infty}g_q(z)\,dz\in L_g,
\qquad
2\pi i\int_\infty^{\gamma\infty}g_q^c(z)\,dz\in L_g^c.
\]
In particular,
$\varphi(\gamma\infty)=\varphi^c(\gamma\infty)=O$.
\item Put $\xi_q(x)=(x/\varpi)_3^{-1}$ for
$x\in(\mathbf Z/q\mathbf Z)^\times$.  If $\gamma\in\Gamma_0(N)$, then
\[
\begin{aligned}
\varphi(\mathcal A_3z)&=[\omega]\varphi(z),&
\varphi^c(\mathcal A_3z)&=[\omega]\varphi^c(z),\\
\varphi(\gamma z)&=[\xi_q(d(\gamma))]\varphi(z),&
\varphi^c(\gamma z)&=[\overline{\xi_q(d(\gamma))}]\varphi^c(z),\\
\varphi(B_3z)&=[\omega^2]\varphi(z),&
\varphi^c(B_3z)&=[\omega^2]\varphi^c(z).
\end{aligned}
\]
\item For one sign $\epsilon_0\in\{1,-1\}$,
\[
 \varphi(0)=(0,\epsilon_0\bar\varpi/2),\qquad
 \varphi^c(0)=(0,\epsilon_0\varpi/2).
\]
\item For each $\epsilon\in\{1,-1\}$ there is
$\mathcal F_\epsilon\in K(X_\Gamma)$, with non-zero constant term at
$\infty$, such that
\begin{equation}\label{eq:modular-cube-functions}
 \mathcal F_\epsilon^3
 =\frac{Y\circ\varphi+\epsilon\bar\varpi/2}
        {Y^c\circ\varphi^c+\epsilon\varpi/2},
 \qquad
 \mathcal F_\epsilon(S_az)=\mathcal F_\epsilon(z)\quad(a=0,1).
\end{equation}
\end{enumerate}
\end{lemma}

\begin{proof}
Part (i) follows from \cite[Lemma~5.3]{Yin47}: the lattice
in {\it loc. cit.} generated by the corresponding cusp integrals on
$\Gamma_0(N/3)$ equals $L_g$ and the coefficient-conjugate assertion appears 
just after \cite[Corollary~5.4]{Yin47}.

The $\mathcal A_3$- and $\Gamma_0(N)$-actions in (ii), together with the
cusp values in (iii), are the transformation formulas of
\cite[Proposition~3.2 and Proposition~5.8]{Yin47}; part (i) removes the
additive cusp terms.  For the $B_3$-action, use
$B_3=W_N\mathcal A_3^{-1}W_N^{-1}$.  The Atkin--Lehner relation
\cite[(6.3.14)]{Yin47} gives group isomorphisms
$\mathcal J,\mathcal J^c$ with
\begin{equation}\label{eq:based-atkin-lehner}
 \varphi(W_Nz)=\mathcal J\varphi^c(z)+\varphi(0),\qquad
 \varphi^c(W_Nz)=\mathcal J^c\varphi(z)+\varphi^c(0).
\end{equation}
At $z=0$, the first identity in \eqref{eq:based-atkin-lehner} gives
\[
 \mathcal J\varphi^c(0)+\varphi(0)=\varphi(W_N0)=\varphi(\infty)=O.
\]
Applying \eqref{eq:based-atkin-lehner} twice also gives
$\mathcal J\mathcal J^c=1$.  These two identities, together with the
$\mathcal A_3^{-1}$-action and the fact that $[\omega-1]$ annihilates
$E[\lambda]$, show that the translation term vanishes and
$\varphi(B_3z)=[\omega^2]\varphi(z)$.  The coefficient-conjugate
identity follows in the same way.

Finally, \cite[Proposition~7.4, (8.0.1)--(8.0.2), and
Proposition~8.1]{Yin47} gives the functions in (iv) with
Fourier expansions at $\infty$ with coefficients in $K$.  The matrices
$\mathcal A_3$ and $B_3$
normalize $\Gamma$: for
$\gamma=\left(\begin{smallmatrix}a&b\\c&d\end{smallmatrix}\right)$,
\[
\mathcal A_3^{-1}\gamma\mathcal A_3=
\begin{pmatrix}
 a-c/3&b+(a-d)/3-c/9\\ c&d+c/3
\end{pmatrix},
\]
and $W_N$ exchanges the diagonal residues modulo $q$.
Thus the quotients
$\mathcal F_\epsilon(\mathcal A_3z)/\mathcal F_\epsilon(z)$ and
$\mathcal F_\epsilon(B_3z)/\mathcal F_\epsilon(z)$ are constant cube
roots of unity.  The non-zero constant term at $\infty$ makes the first
constant one.  At the cusp $0$, the common sign in (iii) gives the same conclusion
for $\mathcal F_{\epsilon_0}$ under $B_3$.  The other quotient is
$0/0$ there, so it cannot be read off from its value at that cusp.
Instead, multiplying the two cube identities gives
\[
 \mathcal F_{+}\mathcal F_{-}
 =\zeta\,\frac{X\circ\varphi}{X^c\circ\varphi^c},
 \qquad \zeta\in\mu_3,
\]
where $X$ and $X^c$ are the $x$-co-ordinates on the two elliptic
curves.  Both scale by $\omega^2$ under $B_3$, so their ratio is
invariant.  The product identity therefore gives the $B_3$-invariance
of $\mathcal F_{-\epsilon_0}$.  Since
$S_0=\mathcal A_3$ and $S_1=B_3^2$, \eqref{eq:modular-cube-functions}
follows.
\end{proof}

\begin{remark}\label{rem:modular-normalization}
The integral normalizations of the modular parametrizations underlie
parts~\textup{(iii)} and~\textup{(iv)}.  Indeed, the cusp values in
part~\textup{(iii)} generate $E_{\bar\varpi}[\lambda]$ and
$E_\varpi[\lambda]$, respectively.
Moreover, the
Manin--Stevens constant comparing the pullback of a N\'eron differential
on the chosen minimal elliptic model with the normalized eigenform is a
unit by \cite[Theorem~6.1]{Yin47}. This gives the integral normalization of
the $q$-expansions entering the cube-root identities in
part~\textup{(iv)}.
\end{remark}

\section{A cubic identity for the base CM point}\label{sec:cm-kummer}
This section chooses the base CM point and proves the cubic identity
used to construct the first \(\lambda\)-division. 
By \Cref{lem:base-branches}, the $\varphi^c$-values are
$K$-rational $\lambda$-torsion, while
\Cref{prop:primary-non-vanishing} shows that neither $\varphi$-value is
$\lambda$-torsion.  The Atkin--Lehner relation then allows us to choose
a CM point whose $\varphi^c$-value is non-zero.  Evaluating the modular
cube-root identity at this point proves \Cref{prop:uniform-kummer}: the
cubic descent class of its $E_{\bar\varpi}$-image is represented by
$\varpi$.  This Kummer class will be used for the Fermat
$\lambda$-division in \Cref{sec:fermat-cover}.

\subsection{Choosing the base CM point}\label{subsec:base-cm-point}

\begin{lemma}\label{lem:base-branches}
For $a=0,1$, put
$Q^{(a)}=\varphi([\tau_a,1])$ and
$Q^{c,(a)}=\varphi^c([\tau_a,1])$.
Then $Q^{(a)}\in E_{\bar\varpi}(F_q)$ and
$Q^{c,(a)}\in E_\varpi[\lambda](K)$.  Every $Q^{(a)}\ne O$
has co-ordinates satisfying $t^2x(Q^{(a)}),y(Q^{(a)})\in K$.
\end{lemma}

\begin{proof}
By \Cref{lem:fine-base-field}(i), $[\tau_a,1]\in X_\Gamma(F_q)$, so
$Q^{(a)}\in E_{\bar\varpi}(F_q)$.  Put
$\gamma_a=S_aI_a\in\Gamma_0(N)$.  Its lower-right residue is
$\omega_q^2$, hence $\xi_q(d(\gamma_a))=\omega$.  By
\Cref{lem:modular-functions}(ii),
\[
 \varphi(S_az)=[\omega]\varphi(z),\qquad
 \varphi^c(S_az)=[\omega]\varphi^c(z),
\]
and
\[
 \varphi(\gamma_az)=[\omega]\varphi(z),\qquad
 \varphi^c(\gamma_az)=[\omega^2]\varphi^c(z).
\]
Since $\gamma_a=S_aI_a$, it follows that
\[
 \varphi(I_az)=\varphi(z),\qquad
 \varphi^c(I_az)=[\omega]\varphi^c(z).
\]
Now $I_a\tau_a=\tau_a$ by \Cref{lem:two-cm-orders}(i), so
$Q^{c,(a)}\in E_\varpi[\lambda]$.  Since
\[
 E_\varpi[\lambda]=\{O,(0,\varpi/2),(0,-\varpi/2)\},
\]
this kernel is $K$-rational.

By \Cref{lem:fine-base-field}(ii), $\sigma_{u_q}$ acts on
$[\tau_a,1]$ through $S_a^{-1}$.  Another application of
\Cref{lem:modular-functions}(ii) gives
\[
 \varphi(S_a^{-1}z)=[\omega^2]\varphi(z),
\]
whence
\[
 (Q^{(a)})^{\sigma_{u_q}}=[\omega^2]Q^{(a)},\qquad
 (Q^{(a)})^{\sigma_{u_q^{-1}}}=[\omega]Q^{(a)}.
\]
Together with $t^{\sigma_{u_q}}=\omega^2t$ from
\Cref{lem:cubic-field}(ii), this gives
$t^2x(Q^{(a)}),y(Q^{(a)})\in K$.
\end{proof}

\begin{lemma}\label{lem:cubic-specialization}
Let $F$ be a field of characteristic zero containing $K$, let
$a,b\in F^\times$, and let $X/F$ be a smooth geometrically integral
curve.  Suppose that non-constant $F$-morphisms
$f:X\to E_a$ and $g:X\to E_b$ satisfy, for some
$\epsilon\in\{1,-1\}$,
\[
\frac{y\circ f+\epsilon a/2}{y\circ g+\epsilon b/2}=H^3,
\qquad H\in F(X)^\times.
\]
For $c\in F^\times$ define on $E_c[\lambda]$ the three classes
\[
d_c(O)=1,\qquad d_c((0,c/2))=[c]_F,\qquad
 d_c((0,-c/2))=[c]_F^{-1}.
\]
If $P\in X(F)$ and both $f(P)$ and $g(P)$ belong to their
$\lambda$-kernels, then
\[
d_{\epsilon a}(f(P))=d_{\epsilon b}(g(P))
\quad\text{in }F^\times/F^{\times3},
\]
where $E_{\epsilon a}=E_a$ and $E_{\epsilon b}=E_b$.
\end{lemma}

\begin{proof}
Replacing $(a,b)$ by $(\epsilon a,\epsilon b)$, it suffices to treat
$\epsilon=1$.  Put $h_c=y+c/2$.  Since
\[
 \operatorname{div}(h_c)=3(0,-c/2)-3(O),
\]
the order of $h_c$ at each point of $E_c[\lambda]$ is divisible by
$3$.  Its leading coefficient therefore defines a class in
$F^\times/F^{\times3}$ independent of the local parameter.  At the
three points of
$E_c[\lambda]$, the cubic classes of the leading coefficients of $h_c$
are, respectively,
$1,[c]_F,[c]_F^{-1}$.  Indeed, at $O$ the parameter $u=-x/y$ gives
$h_c=-u^{-3}+\cdots$; at $(0,c/2)$ one has $h_c=c$; and at
$(0,-c/2)$,
\[
 (y+c/2)(y-c/2)=x^3,
 \qquad h_c=-x^3/c+\cdots.
\]
If $z$ is a local parameter at $f(P)$ and $w$ one at $P$, then
$z\circ f=c_1w^e+\cdots$ with $c_1\in F^\times$.  Pullback therefore
multiplies the leading coefficient of $h_a$ by
$c_1^{\operatorname{ord}_{f(P)}(h_a)}$, which is a cube; the same
applies to $g$.  Thus the leading-coefficient classes of
$h_a\circ f$ and $h_b\circ g$ are
$d_a(f(P))$ and $d_b(g(P))$.  Their quotient is the leading coefficient
of $H^3$, hence a cube.
\end{proof}

\begin{proposition}\label{prop:primary-non-vanishing}
Let $q\equiv4\pmod9$ be prime, let $q=\varpi\bar\varpi$ with
$\varpi\equiv1\pmod3$, and let $r\in\mathbf Z$ satisfy
$r^2-r+1\equiv0\pmod{3q}$ and $-r\equiv\omega\pmod\varpi$.
Then both CM points satisfy
\[
Q^{(a)}=\varphi([\tau_a,1])\notin E_{\bar\varpi}[\lambda]
\qquad(a=0,1).
\]
\end{proposition}

\begin{proof}
Using coefficient conjugation and the identity
\[
 -\frac1{3(\omega+1-r)}
 =-\overline{\left(-\frac1{3(\omega+r)}\right)},
\]
as well as the
Atkin--Lehner identities \eqref{eq:based-atkin-lehner} from the
proof of \Cref{lem:modular-functions}(ii), we obtain
\begin{equation}\label{eq:two-occurrence-values}
Q^{(1)}=\mathcal J\overline{Q^{(0)}}+\varphi(0),\qquad
Q^{c,(1)}=\mathcal J^c\overline{Q^{c,(0)}}+\varphi^c(0).
\end{equation}
Here the bars denote complex conjugation on the co-ordinates. 
By \Cref{lem:modular-functions}(iii), the two cusp values are
non-zero $\lambda$-torsion points. 
Complex conjugation carries the
$\lambda$-kernel on one curve to that on its conjugate because
$\bar\lambda$ is associated to $\lambda$; the isomorphisms
$\mathcal J,\mathcal J^c$ likewise carry the respective kernels into
one another.  Hence $Q^{(0)}$ lies in $E_{\bar\varpi}[\lambda]$ if and
only if $Q^{(1)}$ does.

Suppose this occurs.  Then both $\varphi$-values lie in the kernel,
while \Cref{lem:base-branches} gives
$Q^{c,(a)}\in E_\varpi[\lambda]$ for $a=0,1$.  Apply
\Cref{lem:cubic-specialization} over
$F=F_q=K(\sqrt[3]{\varpi})$ to the cube identity of
\Cref{lem:modular-functions}(iv), at the two $F_q$-rational CM points
from \Cref{lem:fine-base-field}(i), with
$a=\bar\varpi$, $b=\varpi$, and $\epsilon=1$.  Since $\varpi$ is a cube
in $F_q$,
$d_b$ is trivial on $E_\varpi[\lambda]$.  By
\Cref{lem:cubic-field}(i), $F_q/K$ is unramified at
$\bar{\mathfrak q}$.  Thus the normalized valuation of $\bar\varpi$ is
$1$ at every prime of $F_q$ above $\bar{\mathfrak q}$, and
$[\bar\varpi]_{F_q}\ne1$.  Hence the two
non-zero points of $E_{\bar\varpi}[\lambda]$ have non-trivial
$d_a$-class, forcing
\[
 Q^{(0)}=Q^{(1)}=O.
\]
Then \eqref{eq:two-occurrence-values} gives $\varphi(0)=O$, contrary to
\Cref{lem:modular-functions}(iii).  Hence neither $\varphi$-value lies
in the $\lambda$-kernel.
\end{proof}

By \Cref{prop:primary-non-vanishing}, both $\varphi$-values lie outside
$E_{\bar\varpi}[\lambda]$.  At least one conjugate value is non-zero:
if $Q^{c,(0)}=O$, then \eqref{eq:two-occurrence-values} gives
$Q^{c,(1)}=\varphi^c(0)\ne O$.  Choose $a_*\in\{0,1\}$ so that
$Q^{c,(a_*)}\ne O$, and put
\[
 \tau_*=\tau_{a_*},\qquad
 \iota=\iota_{a_*},\qquad
 I=I_{a_*},\qquad S_*=S_{a_*}.
\]
From now on set
\[
P_0=[\tau_*,1],\qquad
Q_q=\varphi(P_0)=(x_q,y_q),\qquad Q_q^c=\varphi^c(P_0).
\]
Accordingly $Q_q^c=(0,\eta_q\varpi/2)$ for a sign $\eta_q$.

\subsection{The cubic Kummer class}\label{subsec:cubic-kummer-class}
We now compare the two $\lambda$-descent functions at the chosen CM
point.  Since
$Q_q^c=(0,\eta_q\varpi/2)$ is non-zero $\lambda$-torsion, the function
$Y^c+\eta_q\varpi/2$ takes the explicit value $\eta_q\varpi$.
On $E_{\bar\varpi}$ put
\[
 f_q=Y+\eta_q\bar\varpi/2,
 \qquad
 \operatorname{div}(f_q)=3(0,-\eta_q\bar\varpi/2)-3(O).
\]
\Cref{lem:modular-functions}(iv) supplies a modular
cube root of the ratio of these two descent functions.  Thus it remains
only to show that its value at $P_0$ descends from $F_q$ to $K$.

\begin{proposition}\label{prop:uniform-kummer}
Set
\[
 z_q=\mathcal F_{\eta_q}(P_0).
\]
Then $z_q\in K^\times$ and
\begin{equation}\label{eq:intro-kummer}
 y_q+\eta_q\bar\varpi/2=\eta_q\varpi z_q^3.
\end{equation}
Consequently
\begin{equation}\label{eq:uniform-kummer}
 [f_q(Q_q)]_K=[\varpi]_K.
\end{equation}
\end{proposition}

\begin{proof}
The numerator in \eqref{eq:modular-cube-functions} is finite and non-zero
because $Q_q$ is not $\lambda$-torsion, while the denominator at
$Q_q^c$ is $\eta_q\varpi$.  Hence $z_q\in F_q^\times$.
By \Cref{lem:fine-base-field}(ii), the generator $\sigma_{u_q}$ acts on
this value through $S_*^{-1}$, which fixes it by
\Cref{lem:modular-functions}(iv).  Thus $z_q\in K^\times$.
Evaluating the modular cube identity at $P_0$ now gives
\eqref{eq:intro-kummer}.  Since $\eta_q$ is a cube in $K$, the same
identity gives \eqref{eq:uniform-kummer}.
\end{proof}

The coefficient-conjugate branch makes the resulting Kummer class
explicit.  Indeed, $Q_q^c$ is non-zero $\lambda$-torsion and its
descent value is $\eta_q\varpi$, whereas the modular cube root and
Shimura reciprocity show that the two descent values differ by a cube
in $K^\times$.  Thus, although $Q_q$ is in general only
$F_q$-rational, $f_q(Q_q)\in K^\times$ and
\eqref{eq:uniform-kummer} gives
\[
 K\bigl(\sqrt[3]{f_q(Q_q)}\bigr)=F_q.
\]

\section{The Fermat isogeny and Frobenius}\label{sec:fermat-cover}

This section uses the cubic identity of
\Cref{prop:uniform-kummer} to construct a first $\lambda$-division of $Q_q$ and
computes its reduction at $p$.  We first identify the Fermat isogeny
with multiplication by $\lambda$ and record its descent map
(cf.~\Cref{lem:Fermat-lambda-cover,lem:exact-descent-kernel}).  The Kummer
class of \Cref{prop:uniform-kummer} then gives a normalized lift whose
two affine co-ordinates lie respectively in $K$ and in $tK$, where
$t^3=\varpi$ (cf.~\Cref{prop:normalized-lift}).  Finally, comparing the
$p^2$-power action on these co-ordinates with the CM Frobenius proves
\Cref{prop:frobenius-obstruction}: when $\alpha_q(p)\ne1$, the point
$[p+1]\bar R_q$ is a specified non-zero element of the
$\lambda$-kernel.  This local calculation underlies the global
division boundary constructed in the next section.

Put
\[
 a=\bar\varpi,
 \qquad t^3=\varpi,
 \qquad F_q=K(t),
\]
and consider the Fermat cubic
\begin{equation}
 \mathscr C_a:\qquad V^3-U^3=aW^3.
\end{equation}
We regard $\mathscr C_a$ as an elliptic curve with origin
$O_{\mathscr C}=[1:1:0]$ and normalize its CM action by
\[
[\omega]_{\mathscr C_a}[U:V:W]=[\omega^2U:\omega^2V:W].
\]
The three points at infinity form the kernel of $\ell_a$.
We choose $\mathcal T=[1:\omega:0]$ as a generator.  Translation by
$\ker\ell_a$ acts simply transitively on each geometric fibre of
$\ell_a$.

\subsection{The isogeny and cubic descent}
The two affine co-ordinates on the Fermat model realize the two descent
functions for the $\lambda$-isogeny.

\begin{lemma}\label{lem:Fermat-lambda-cover}
The formula
\begin{equation}\label{eq:fermat-isogeny}
 \ell_a(U,V,W)
 =\left(\frac{UV}{W^2},
 \frac{U^3+V^3}{2W^3}\right)
\end{equation}
on $W\ne0$ extends to a degree-three $K$-isogeny, where
$E_a:y^2=x^3+a^2/4$.  Its kernel is
\begin{equation}
 \ker\ell_a
 =\{[1:\zeta:0]:\zeta\in\mu_3\}
 =\mathscr C_a[\lambda],
 \qquad \lambda=1-\omega,
\end{equation}
and
\[
 \ell_a\circ[\omega]_{\mathscr C_a}=[\omega]_{E_a}\circ\ell_a.
\]
There is a unique $K$-isomorphism
\[
 \jmath_a:E_a\xrightarrow{\sim}\mathscr C_a
 \qquad\text{such that}\qquad
 \jmath_a\circ\ell_a=[\lambda]_{\mathscr C_a};
\]
moreover $\jmath_a$ is $\mathcal O_K$-linear.  Under this identification,
$\ell_a$ is multiplication by $\lambda$.  Finally,
\begin{equation}\label{eq:descent-pullback}
 \ell_a^*\left(y+\frac a2\right)=\left(\frac VW\right)^3,
 \qquad
 \ell_a^*\left(y-\frac a2\right)=\left(\frac UW\right)^3.
\end{equation}
\end{lemma}

\begin{proof}
Using $V^3-U^3=aW^3$, direct substitution gives
\[
 \left(\frac{U^3+V^3}{2W^3}\right)^2-\frac{a^2}{4}
 =\frac{U^3V^3}{W^6}
 =\left(\frac{UV}{W^2}\right)^3,
\]
so \eqref{eq:fermat-isogeny} maps the affine Fermat cubic to $E_a$ and
gives \eqref{eq:descent-pullback}.  Since $\mathscr C_a$ is smooth and
projective, this rational map extends uniquely to a morphism
$\mathscr C_a\to E_a$.  At $W=0$ one has $V/U\in\mu_3$, and the
co-ordinate functions in
\eqref{eq:fermat-isogeny} have poles there, so each of the three points at
infinity maps to $O$.  The resulting morphism sends $O_{\mathscr C}$
to $O$ and is non-constant, hence is an isogeny.  Its kernel is exactly
these three points, so its degree is $3$.  They are fixed by
$[\omega]_{\mathscr C_a}$; since $[1-\omega]$ also has degree three,
they form $\mathscr C_a[1-\omega]$.
The identity
$\ell_a\circ[\omega]_{\mathscr C_a}=[\omega]_{E_a}\circ\ell_a$
follows directly from the two displayed $\omega$-actions and
\eqref{eq:fermat-isogeny}.

Both $[\lambda]_{\mathscr C_a}:\mathscr C_a\to \mathscr C_a$ and $\ell_a:\mathscr C_a\to E_a$ are
separable degree-three isogenies with kernel $\mathscr C_a[\lambda]$.  The universal
property of the quotient by this kernel therefore gives a unique
$K$-isomorphism $\jmath_a:E_a\to \mathscr C_a$ satisfying
$\jmath_a\circ\ell_a=[\lambda]_{\mathscr C_a}$.  Since $\ell_a$ is
$\mathcal O_K$-linear and surjective, this identity also shows that
$\jmath_a$ is $\mathcal O_K$-linear.

The inverse isomorphism is given by
\begin{equation}\label{eq:fermat-identification}
\jmath_a^{-1}[U:V:W]
=\left(-\frac{\omega aW}{V-U},
       -\frac{a(1+2\omega)(V+U)}{2(V-U)}\right).
\end{equation}
Let $i:\mathscr C_a\dashrightarrow E_a$ denote the rational map given
by this formula.  At $O_{\mathscr C}$, take $w=W/U$ as a local
parameter and put $v=V/U$.  Since $v^3-1=aw^3$, one has
$v-1=(a/3)w^3+O(w^6)$; hence the two co-ordinates of $i$ have poles of
orders two and three, respectively.  Thus $i$ extends across
$O_{\mathscr C}$ and sends it to $O$.  Properness gives the extension
at the remaining omitted points.  A morphism of elliptic curves that
sends origin to origin is a group homomorphism.  The formula also shows
directly that
$i\circ[\omega]_{\mathscr C_a}=[\omega]_{E_a}\circ i$, so $i$ is
$\mathcal O_K$-linear.

To verify its compatibility with $[\lambda]$, take $W=1$ and write
$P=i[U:V:1]=(X,Y)$.  Substitution gives $Y^2=X^3+a^2/4$.
Since $1+2\omega=\omega\lambda$, the slope used to add
$P$ to $-[\omega]P$ is
\[
\frac{2Y}{\lambda X}=U+V.
\]
The addition law then gives
\[
[\lambda]P=\left(UV,\frac{U^3+V^3}{2}\right)=\ell_a[U:V:1].
\]
Thus $[\lambda]_{E_a}\circ i=\ell_a$.  The inverse of the quotient
isomorphism constructed above satisfies the same identity, since
$\jmath_a$ is $\mathcal O_K$-linear.  Their difference is a morphism
from the connected curve $\mathscr C_a$ to the finite group
$E_a[\lambda]$, hence is constant; it is zero at $O_{\mathscr C}$.  Therefore
$i=\jmath_a^{-1}$, proving that the displayed formula is the
isomorphism supplied by the quotient construction.
\end{proof}

When $a=m\in\mathbf Q^\times$, the formulas for $\ell_m$ are defined
over $\mathbf Q$ and give the usual rational $3$-isogeny from
$\mathscr C_m$ to $E_m$.  Over $K$, the isomorphism $\jmath_m$
identifies this isogeny with multiplication by $\lambda$.

\begin{lemma}\label{lem:exact-descent-kernel}
Let $M/K$ be a field extension and $a\in K^\times$.  The
$\lambda$-descent homomorphism
\[
 \mathfrak d_a:E_a(M)\longrightarrow M^\times/M^{\times3},
 \qquad
 \mathfrak d_a(P)=\left[y(P)+\frac a2\right],
\]
has value $1$ at $O$ and value $[a]^{-1}$ at $(0,-a/2)$.  Its kernel is
\begin{equation}\label{eq:exact-descent-kernel}
 \ker(\mathfrak d_a)
 =\ell_a(\mathscr C_a(M))
 =[\lambda]E_a(M).
\end{equation}
If $P\in E_a[3]$, $P\ne O$, then over any extension of $K$ over which
$P$ is defined,
\[
 \mathfrak d_a(P)=
 \begin{cases}
 [a]^{\pm1},&P\in E_a[\lambda]\setminus\{O\},\\[2pt]
 [a\omega]\ \text{or}\ [a\omega^2],
   &P\in E_a[3]\setminus E_a[\lambda].
 \end{cases}
\]
\end{lemma}

\begin{proof}
Translation by the generator $\mathcal T$ of $\ker\ell_a$ multiplies
$V/W$ by $\omega^2$.  With this identification
$\ker\ell_a\simeq\mu_3$, the pullback identity
\eqref{eq:descent-pullback} represents the connecting homomorphism for
$\ell_a$ by $\mathfrak d_a$.  In particular,
\[
 \ker(\mathfrak d_a)=\ell_a(\mathscr C_a(M)).
\]
At the points where $y+a/2$ vanishes or has a pole, the same map is
obtained from the leading coefficient.  Indeed
$(y+a/2)(y-a/2)=x^3$, so at $(0,-a/2)$ the second factor has value $-a$ and
\[
 \mathfrak d_a(0,-a/2)=[-a]^{-1}=[a]^{-1};
\]
at $O$ the value is $1$.

It remains to identify the image of $\ell_a$ with $[\lambda]E_a(M)$.
By \Cref{lem:Fermat-lambda-cover},
$\jmath_a\circ\ell_a=[\lambda]_{\mathscr C_a}$ and $\jmath_a$ is
$\mathcal O_K$-linear.  Hence
\[
 \ell_a\circ\jmath_a
 =\jmath_a^{-1}\circ[\lambda]_{\mathscr C_a}\circ\jmath_a
 =[\lambda]_{E_a},
\]
and therefore
$\ell_a(\mathscr C_a(M))=[\lambda]E_a(M)$.

The $3$-division polynomial is
$\psi_3=3x(x^3+a^2)$.  For $x=0$ one has $y=\pm a/2$, giving the two
non-zero $\lambda$-torsion classes $[a]$ and $[a]^{-1}$.  For the remaining
six points, $x^3=-a^2$ and $y=\pm a\sqrt{-3}/2$.  Since
$\sqrt{-3}=1+2\omega$,
\[
 y+\frac a2\in\{-a\omega,-a\omega^2\}.
\]
As $-1$ is a cube in $K$, these classes are $[a\omega]$ and
$[a\omega^2]$.

\end{proof}

\begin{remark}
For $P\in E_a(M)$, the fibre $\ell_a^{-1}(P)$ is a torsor under
$\ker\ell_a\simeq\mu_3$.  Under the Kummer identification
$H^1(M,\mu_3)\simeq M^\times/M^{\times3}$, its class is
$\mathfrak d_a(P)$.  Thus the descent factors above encode
the obstruction to a first $\lambda$-division over $M$.
\end{remark}

\subsection{An explicit division of the CM point}
The cubic identity produces a first division of $Q_q$ whose two affine
co-ordinates lie on different $K$-lines.

\begin{proposition}\label{prop:normalized-lift}
There is a point
\[
 R_q=[U:V:1]\in \mathscr C_{\bar\varpi}(F_q),
 \qquad \ell_{\bar\varpi}(R_q)=Q_q,
\]
whose co-ordinates satisfy
\begin{equation}\label{eq:normalized-lift-lines}
\begin{cases}
 U\in K^\times,\quad V\in tK^\times,&\eta_q=+1,\\
 U\in tK^\times,\quad V\in K^\times,&\eta_q=-1.
\end{cases}
\end{equation}
Exactly one affine Fermat co-ordinate is therefore $K$-rational, and
the other lies in the one-dimensional $K$-space $tK$.
\end{proposition}

\begin{proof}
The cubic twisting isomorphism over $F_q$
\[
 S_t:E_{\bar\varpi}\otimes_KF_q
     \xrightarrow{\sim} E_q\otimes_KF_q,
 \qquad (x,y)\mapsto(t^2x,\varpi y)
\]
sends $Q_q$ to a $K$-rational point: for the base point $\tau_*$
chosen in \Cref{subsec:base-cm-point}, one has
$t^2x(Q_q),y(Q_q)\in K$ by \Cref{lem:base-branches}.
Hence $b_q:=t^2x_q\in K$, while
$y_q\pm\bar\varpi/2\in K$.  Moreover,
\begin{equation}
 \frac{x_q^3}{\varpi}=\left(\frac{b_q}{\varpi}\right)^3.
\end{equation}
Using
$(y_q+\eta_q\bar\varpi/2)(y_q-\eta_q\bar\varpi/2)=x_q^3$
and the cubic identity \eqref{eq:intro-kummer},
\[
 y_q-\eta_q\frac{\bar\varpi}{2}
 =\eta_q\left(\frac{b_q}{\varpi z_q}\right)^3.
\]
Since $-1$ is a cube in $K$, the two descent factors therefore have
classes $[1]$ and $[\varpi]$
in $K^\times/K^{\times3}$.  The sign $\eta_q$ determines which factor
carries the non-trivial class:
\begin{equation}\label{eq:descent-class-table}
\begin{array}{c|cc}
 & y_q-\bar\varpi/2 & y_q+\bar\varpi/2 \\ \hline
 \eta_q=+1 & [1] & [\varpi] \\
 \eta_q=-1 & [\varpi] & [1].
\end{array}
\end{equation}
By \eqref{eq:descent-pullback}, an $\ell_{\bar\varpi}$-preimage of $Q_q$
with $W=1$ is obtained by choosing cube roots
\[
 U^3=y_q-\frac{\bar\varpi}{2},
 \qquad
 V^3=y_q+\frac{\bar\varpi}{2}.
\]
Neither factor vanishes, since otherwise $Q_q$ would be a non-zero
$\lambda$-torsion point, contrary to the choice of $\tau_*$.  The table \eqref{eq:descent-class-table}
therefore permits the roots to be chosen with the field memberships in
\eqref{eq:normalized-lift-lines}.  Their product has cube $x_q^3$, so
$UV/x_q\in\mu_3\subset K$.  Multiplying either chosen root by the
inverse of this element preserves its $K$- or $tK$-line and gives
$UV=x_q$.
The identities \eqref{eq:descent-pullback} then give
$\ell_{\bar\varpi}(R_q)=Q_q$.
\end{proof}

We record the following elementary properties of this division point used in the
Frobenius calculation (cf.~\Cref{prop:frobenius-obstruction}).

\begin{lemma}\label{lem:fermat-kernel-action}
Put $a=\bar\varpi$ and let $R_q$ be as in
\Cref{prop:normalized-lift}.
\begin{enumerate}[label=\textup{(\roman*)}]
\item The other two $\ell_a$-preimages of $Q_q$ are obtained from
$R_q$ by translation by the two non-zero points of $\ker\ell_a$; they
satisfy the same $K$- and $tK$-line conditions as $R_q$.

\item If
\begin{equation}\label{eq:elliptic-normalized-lift}
 R_q^E:=\jmath_a^{-1}(R_q)\in E_a(F_q),
 \qquad T_E:=\jmath_a^{-1}(\mathcal T)\in E_a[\lambda],
\end{equation}
then
\[
 [\lambda]R_q^E=Q_q,
 \qquad T_E\ne O.
\]

\item Define
\[
 d[U:V:W]=[\omega U:\omega^2V:W].
\]
Then $d$ is translation by $\mathcal T$, $d^2$ is translation by
$-\mathcal T$, and for $R=[U:V:W]$,
\begin{equation}\label{eq:affine-coordinates}
\begin{aligned}
 [U:\omega V:W]&=[\omega]R+\mathcal T,&
 [\omega U:V:W]&=[\omega]R-\mathcal T,\\
 [U:\omega^2V:W]&=[\omega^2]R-\mathcal T,&
 [\omega^2U:V:W]&=[\omega^2]R+\mathcal T.
\end{aligned}
\end{equation}

\item Under \eqref{eq:fermat-identification},
\begin{equation}\label{eq:TE-explicit}
 T_E=(0,-\bar\varpi/2),
 \qquad T_E^\sharp:=-\overline{T_E}=(0,\varpi/2).
\end{equation}
In particular, $T_E^\sharp\in E_\varpi[\lambda]\setminus\{O\}$.
\end{enumerate}
\end{lemma}

\begin{proof}
The formula for $\ell_a$ gives $\ell_a\circ d=\ell_a$.  Hence
$P\mapsto d(P)-P$ maps the connected curve $\mathscr C_a$ to the finite
\'{e}tale group $\ker\ell_a$ and is constant.  At
$O_{\mathscr C}=[1:1:0]$ its value is $\mathcal T=[1:\omega:0]$; thus $d$ is
translation by $\mathcal T$ and $d^2$ by $-\mathcal T$.  Since
multiplication of $U$ and $V$ by powers of $\omega\in K$ preserves the
$K$- and $tK$-lines, part~\textup{(i)} follows.  Combining $d$ and $d^2$ with
$[\omega]_{\mathscr C_a}[U:V:W]=[\omega^2U:\omega^2V:W]$ gives
\eqref{eq:affine-coordinates}.

For (ii), use
$\jmath_a\circ\ell_a=[\lambda]_{\mathscr C_a}$ and the
$\mathcal O_K$-linearity of $\jmath_a$; equivalently,
$\ell_a\circ\jmath_a=[\lambda]_{E_a}$.  Applying this identity to
$R_q^E$ gives $[\lambda]R_q^E=Q_q$.  Finally,
\eqref{eq:fermat-identification} gives
$T_E=(0,-\bar\varpi/2)$ and hence
$-\overline{T_E}=(0,\varpi/2)$, proving (iv).
\end{proof}

\subsection{Frobenius at \texorpdfstring{$p\equiv8\pmod9$}{p = 8 mod 9}}\label{sec:affine-frobenius}
We compare the $p^2$-power action on the reduced Fermat
co-ordinates with the CM Frobenius endomorphism.  We first describe the
cubic reciprocity relation at $p$ and the $p^2$-Frobenius endomorphism.

Fix $p\equiv8\pmod9$, set $a=\bar\varpi$, and let $R_q$ be the point of
\Cref{prop:normalized-lift}.  Since $p\ne q,3$, the curves
$\mathscr C_a,E_a$ and the isogeny $\ell_a$ have good reduction at the
unique prime $\mathfrak p=(p)$ of $K$.  Choose a prime $\mathfrak P$ of
$F_q$ above $\mathfrak p$ and write a bar for reduction at
$\mathfrak P$.

\begin{lemma}\label{lem:fermat-frobenius}
Let
\[
 \operatorname{Fr}_{p^2}:z\longmapsto z^{p^2}
\]
denote the $p^2$-power automorphism of the residue field, applied to
the co-ordinates of the reduced Fermat model.
\begin{enumerate}[label=\textup{(\roman*)}]
\item One has
\begin{equation}\label{eq:cubic-reciprocity-at-p}
 \left(\frac{\varpi}{(p)}\right)_3
 =\left(\frac{-p}{\varpi}\right)_3
 =\left(\frac p\varpi\right)_3
 =\alpha_q(p),
\end{equation}
and
\begin{equation}\label{eq:t-frobenius}
 \overline t^{\,p^2}=\alpha_q(p)\,\overline t.
\end{equation}
If $\alpha_q(p)\ne1$, then $\mathfrak P/\mathfrak p$ has residue
degree $3$, so the residue field at $\mathfrak P$ is
$\mathbf F_{p^6}$.

\item On the reduction of $E_{\bar\varpi}$,
\begin{equation}\label{eq:intrinsic-frob}
 \operatorname{Fr}_{p^2}=[-p\alpha_q(p)].
\end{equation}
Via the $\mathcal O_K$-linear identification $\jmath_a$, the same
endomorphism acts on the reduction of $\mathscr C_a$.
\end{enumerate}
\end{lemma}

\begin{proof}
Since $-p\equiv1\pmod3$ is primary and $-1$ is a cube in $K$, cubic
reciprocity gives \eqref{eq:cubic-reciprocity-at-p}.  The extension
$F_q/K$ is unramified at $\mathfrak p$, and its Frobenius sends $t$ to
$\alpha_q(p)t$, which proves \eqref{eq:t-frobenius}.  If
$\alpha_q(p)\ne1$, this automorphism has order three, proving the
residue-degree assertion.

For $E_1:y^2=x^3+1/4$ over $\mathbf F_p$, cubing is bijective because
$p\equiv2\pmod3$; the quadratic-character sum gives
$\#E_1(\mathbf F_p)=p+1$.  Its Frobenius polynomial is therefore
$T^2+p$, so its $p^2$-Frobenius is $[-p]$.
Choose $c^3=\bar\varpi$ in an algebraic closure of the residue field
and, on the reductions, put
$\operatorname{Tw}_c(x,y)=(c^2x,c^3y):
E_1\xrightarrow{\sim}E_{\bar\varpi}$.  With
\[
 \beta_c=c^{p^2-1}=\left(\frac{\bar\varpi}{(p)}\right)_3,
\]
one has
\[
 \operatorname{Fr}_{p^2}\circ\operatorname{Tw}_c
 =[\beta_c^2]\operatorname{Tw}_c\circ[-p].
\]
Since cubing is bijective on $\mathbf F_p^\times$,
$(q/(p))_3=1$; multiplicativity and \eqref{eq:cubic-reciprocity-at-p}
give $1=(q/(p))_3=\alpha_q(p)\beta_c$.  Thus
$\beta_c^2=\alpha_q(p)$, proving \eqref{eq:intrinsic-frob}.
The isomorphism $\jmath_a$ is defined over $K$, hence its reduction is
defined over $\mathbf F_{p^2}$ and commutes with
$\operatorname{Fr}_{p^2}$.  Transporting \eqref{eq:intrinsic-frob}
through this $\mathcal O_K$-linear isomorphism gives the same formula on
$\mathscr C_a$.
\end{proof}

Assume now that $\alpha_q(p)\ne1$.  By
\Cref{lem:fermat-frobenius}(i), the $p^2$-power map fixes a reduced
$K$-co-ordinate and multiplies a reduced $tK$-co-ordinate by
$\alpha_q(p)$; by \Cref{lem:fermat-frobenius}(ii), the same Frobenius
is the CM endomorphism $[-p\alpha_q(p)]$.  Comparing these actions on
$R_q$ yields the non-zero kernel term computed in
\Cref{prop:frobenius-obstruction}.

\begin{proposition}\label{prop:frobenius-obstruction}
Assume $\alpha_q(p)\ne1$.  Define
$\varepsilon_{q,p}\in\{\pm1\}$ by
\begin{equation}\label{eq:affine-sign-table}
\begin{array}{c|cc}
 &\alpha_q(p)=\omega&\alpha_q(p)=\omega^2\\ \hline
 \eta_q=+1&-1&+1\\
 \eta_q=-1&+1&-1
\end{array}
\qquad(\text{entries }=\varepsilon_{q,p}).
\end{equation}
For the normalized Fermat lift of \Cref{prop:normalized-lift}, we use
$\mathcal T$ also for its reduction at $p$.  Then
\begin{equation}\label{eq:affine-obstruction}
 [p+1]\bar R_q=\varepsilon_{q,p}\mathcal T\ne O.
\end{equation}
Under the canonical CM identification \eqref{eq:elliptic-normalized-lift},
this identity becomes
\begin{equation}\label{eq:affine-obstruction-elliptic}
 [p+1]\bar R_q^E=\varepsilon_{q,p}T_E\ne O
 \qquad\text{on }E_{\bar\varpi}.
\end{equation}
Consequently
\begin{equation}\label{eq:NQ-non-zero}
 \left[\frac{p+1}{3}\right]\bar Q_q\ne O.
\end{equation}
\end{proposition}

\begin{proof}
Because $p\ne3$, specialization is injective on prime-to-$p$ torsion, so
$\mathcal T$ remains non-zero after reduction.

Choose a projective representative of $R_q$ whose co-ordinates are
$\mathfrak P$-integral and not all in the maximal ideal.  Scaling by an
element of $K^\times$ achieves this: the extension $F_q/K$ is unramified
at $p$, and $t^3=\varpi$ is a $p$-adic unit, so the $K$-line and
$tK$-line in \eqref{eq:normalized-lift-lines} have the same valuation
group.  The third co-ordinate remains in $K$.  Consequently
$\operatorname{Fr}_{p^2}$ fixes the reduced $K$-co-ordinates and, by
\eqref{eq:t-frobenius}, multiplies the reduced $tK$-co-ordinate by
$\alpha_q(p)$.

Since $\alpha_q(p)\ne1$, one has $\alpha_q(p)\in\{\omega,\omega^2\}$.  If
$\eta_q=+1$, the $tK$-co-ordinate is $V$; if $\eta_q=-1$, it is $U$.
Applying \eqref{eq:affine-coordinates} gives
\begin{equation}\label{eq:coordinate-affine-frob}
 \operatorname{Fr}_{p^2}(\bar R_q)
 =[\alpha_q(p)]\bar R_q+\delta(\eta_q,\alpha_q(p))\mathcal T,
\end{equation}
where
\[
 \delta(+1,\omega)=\delta(-1,\omega^2)=+1,
 \qquad
 \delta(+1,\omega^2)=\delta(-1,\omega)=-1.
\]
This table depends only on the globally fixed sign $\eta_q$ and on the
residue symbol $\alpha_q(p)\in\mu_3\subset\mathcal O_K$, and is therefore
independent of the chosen prime of $F_q$ above $p$.  By
\eqref{eq:intrinsic-frob}, the
left side of \eqref{eq:coordinate-affine-frob} is
$[-p\alpha_q(p)]\bar R_q$.  Hence
\[
 [-(p+1)\alpha_q(p)]\bar R_q
 =\delta(\eta_q,\alpha_q(p))\mathcal T.
\]
The units $\omega$ and $\omega^2$ act trivially on
$\mathscr C_a[\lambda]$ because $\omega\equiv1\pmod\lambda$.  Multiplying
by $-\alpha_q(p)^{-1}$ gives
\[
 [p+1]\bar R_q=-\delta(\eta_q,\alpha_q(p))\mathcal T.
\]
The right side is $\varepsilon_{q,p}\mathcal T$ by
\eqref{eq:affine-sign-table}, proving \eqref{eq:affine-obstruction}.

Finally, if $[(p+1)/3]\bar Q_q=O$, then
\[
 \left[\frac{p+1}{3}\right]\bar R_q\in\ker\ell_a.
\]
The kernel has order $3$, so $[p+1]\bar R_q=O$, contradicting
\eqref{eq:affine-obstruction}.  The $K_{\mathfrak p}$-isomorphism
$\jmath_a^{-1}$ extends to the good-reduction models.  Applying its
reduction to \eqref{eq:affine-obstruction} gives
\eqref{eq:affine-obstruction-elliptic}; it also preserves the
non-vanishing of $T_E$.  The contradiction proves
\eqref{eq:NQ-non-zero}.
\end{proof}
\begin{remark}
If $\alpha_q(p)=1$, the same calculation gives
$[p+1]\bar R_q=O$.  Thus the non-triviality of
$\alpha_q(p)=(p/\varpi)_3$ is precisely what gives the non-zero
kernel term.  By \Cref{prop:uniform-kummer}, this is the action of
Frobenius at $p$ on the cubic class $[\varpi]_K$.  The same cubic
residue symbol occurs in the $\varpi$-column of the auxiliary
$3$-isogeny descent matrix in \Cref{lem:chan-matrices}.
\end{remark}

\section{The Hecke orbit and its division boundary}\label{sec:hecke-projector}
This section passes from the single-point Frobenius calculation of
\Cref{prop:frobenius-obstruction} to the complete conductor-$p$ Hecke
orbit.  The orbit points have a common reduction modulo $p$
(cf.~\Cref{prop:fine-hecke-orbit}), whereas their unweighted trace in
characteristic zero vanishes as $a_p(g_q^c)=0$
(cf.~\Cref{lem:vanishing-hecke-trace}).  To retain the information lost in
this trace, we first choose $\lambda$-divisions of the orbit points and
then trace them.  The main result, \Cref{prop:division-boundary}, proves
that the resulting $\lambda$-torsion boundary is well defined and
computes it as $\varepsilon_{q,p}T_E^\sharp\ne O$.

Fix an algebraic closure $\overline K_p$ of $K_p$ and, once and for all,
compatible embeddings of the global fields occurring below into
$\overline K_p$, extending the place chosen in
\Cref{sec:affine-frobenius}.  All
completions, specialization maps, and primes above $p$ are taken with respect to 
these embeddings.

Let $P_0=[\tau_*,1]$ be the base CM point chosen in
\Cref{subsec:base-cm-point}, so that
\[
 Q_q=\varphi(P_0),\qquad Q_q^c=\varphi^c(P_0).
\]
We continue to assume the cubic-residue condition
\eqref{eq:intro-ray-condition}.  Put
\[
 C_p=\mathbf F_{p^2}^\times/\mathbf F_p^\times,
 \qquad \Delta_p\simeq\mu_3\subset C_p,
 \qquad G_p=C_p/\Delta_p.
\]
Here $\Delta_p$ is the image of
$\mathcal O_K^\times/\{\pm1\}\simeq\mu_3$.  It acts trivially
on the $\varphi^c$-images of the conductor-$p$ orbit
(cf.~\Cref{lem:exact-residual-action}), and the specialization argument
preceding \Cref{def:complete-division-boundary} shows that it also fixes
their first $\lambda$-divisions.  We nevertheless retain $C_p$
as  it indexes the complete Hecke fibre with its
multiplicities.

\subsection{The ring-class fields and ramification at \texorpdfstring{$p$}{p}}
We first record the class-field facts used to control the conductor-$p$
orbit and its reduction.

\begin{lemma}\label{lem:ring-class-p-ramification}
Let $H_{3p}$ be the ring-class field of
$\mathcal O_{3p}=\mathbf Z+3p\mathcal O_K$, and let $H_p$ be the
ring-class field of conductor $p$.
\begin{enumerate}[label=\textup{(\roman*)}]
\item One has
\[
 H_p\subset H_{3p},\qquad
 \operatorname{Gal}(H_{3p}/K)\simeq C_p,\qquad
 \operatorname{Gal}(H_p/K)\simeq G_p.
\]
\item The extension $H_{3p}/K$ is totally ramified at the unique prime
$\mathfrak p=(p)$, with ramification degree $p+1$.
\item Put $\mathcal K_{9q}=K_{9q}^{\mathrm{ray}}$ and
$\mathscr W=\mathcal K_{9q}H_{3p}$.  Then
\[
 \mathcal K_{9q}\cap H_{3p}=K,
 \qquad
 \operatorname{Gal}(\mathscr W/\mathcal K_{9q})\simeq C_p.
\]
If $B_p=(\mathcal K_{9q})_{\mathfrak P}$ and
$L_p=\mathscr W_{\widetilde{\mathfrak P}}$ at primes above $p$, then
$L_p/B_p$ is totally ramified with group $C_p$.  This remains true
after any finite unramified extension of $B_p$.
\end{enumerate}
\end{lemma}

\begin{proof}
Since $h_K=1$, the ring-class exact sequence identifies the class
group of conductor $f$ with the corresponding local unit quotient
modulo the image of $\mathcal O_K^\times$.  For $f=p$ this gives
$C_p/\Delta_p=G_p$.  For $f=3p$, the conductor-$3$ quotient has order
three, and the image of
$\mathcal O_K^\times/\{\pm1\}\simeq\mu_3$ projects isomorphically onto
that factor; the remaining quotient is therefore $C_p$.  This proves
(i).  Locally at $p$ the latter quotient is
\[
 \mathcal O_{K,p}^\times/
 \bigl(\mathbf Z_p^\times(1+p\mathcal O_{K,p})\bigr),
\]
so it is the full inertia group and has order $p+1$, proving (ii).
The ray extension $\mathcal K_{9q}/K$ is unramified at $p$, whereas
$H_{3p}/K$ is totally ramified there; hence their intersection is $K$.
This proves the first two assertions of (iii); completing at $p$ gives
the local assertions.  Unramified base change does not change the
inertia group.
\end{proof}

\subsection{The CM orbit and its reduction}
Using the ordered factorization $q=\varpi\bar\varpi$, identify
\[
 C_q=(\mathcal O_K/q\mathcal O_K)^\times/(\mathbf Z/q\mathbf Z)^\times
 \simeq \mathbf F_q^\times.
\]
Put
\begin{equation}\label{eq:hq-definition}
 C_q^{(3)}=(\mathbf F_q^\times)^3,
 \qquad h_q=|C_q^{(3)}|=\frac{q-1}{3}\equiv1\pmod3.
\end{equation}
Let $\mathbf h_p=\operatorname{diag}(p,1)$ at $p$ and $1$ elsewhere, and put
\[
 x=[\tau_*,\mathbf h_p^{-1}].
\]
Thus $x$ is obtained from $P_0$ by the conductor-$p$ modification of
the local lattice.

Since $p\nmid9q$, this modification changes only the CM point and its
prime-to-level Hecke fibre.  The modular curve $X_\Gamma$, the eigenform
$g_q$, and the normalized maps $\varphi,\varphi^c$ remain those fixed in
\Cref{sec:cm-data}; in particular, no new modular parametrization is
introduced at conductor $p$.

\begin{lemma}\label{lem:conductor-p-field}
The point $x$ and its complete $p$-Hecke fibre are defined over
$F_qH_{3p}\subset\mathscr W$.
\end{lemma}

\begin{proof}
At $p$ the stabilizer of $x$ consists of
$a+b\omega\in(\mathbf Z_p+p\mathcal O_{K,p})^\times$.  Indeed, in
$\mathbf h_p(a+bI)\mathbf h_p^{-1}$ the lower-left entry is $3s_rb/p$ or $3qb/p$ for
the two possible CM embeddings.  Here $p\nmid s_r$, since
$X^2-X+1$ has no root modulo $p\equiv2\pmod3$, and $p\ne q$.
At every other place the stabilizers are those of
\Cref{lem:fine-base-field}.

Replacing the maximal units at $p$ in the subgroup of
$\mathbf A_{K,f}^\times$ fixing the base CM point by
$(\mathbf Z_p+p\mathcal O_{K,p})^\times$ gives local quotients of
orders $3$, $3$, and $p+1$ at $3,q,p$, respectively.  The diagonal image of
$\mathcal O_K^\times/\{\pm1\}$ has order $3$, and both the Kummer character of $F_q$ and
the ring-class characters of $H_{3p}$ are trivial on this subgroup.  The
corresponding field therefore contains $F_qH_{3p}$ and has degree
$3(p+1)$ over $K$.  Since $F_q/K$ is unramified at $p$ and
$H_{3p}/K$ is totally ramified there, the compositum $F_qH_{3p}$ has
this same degree.  Hence, the complete $p$-Hecke
fibre of $P_0$\footnote{By the \emph{complete $p$-Hecke fibre} of $P_{0}$ we mean the $p+1$
points in the prime-to-level Hecke correspondence $T_p(P_{0})$, with
their full level structures. }, including $x$, is defined over $F_qH_{3p}$.
\end{proof}

 For $C\in C_q$, choose an idele
$u_q(C)\in\mathbf A_{K,f}^\times$ that is $1$ away from $q$ and whose
$q$-component has residue $(C,1)$ under the ordered splitting
$(\mathfrak q,\bar{\mathfrak q})$.  Its action below depends only on
this residue class; for $C=\omega_q$ we take $u_q(C)$ to be the idele
$u_q$ of \Cref{lem:cubic-field}.

The local torus both labels the complete Hecke fibre and describes its
Galois action.  \Cref{prop:fine-hecke-orbit} also records the
$q$-stabilizer used later and the reduction of the $\varphi^c$-images.

\begin{proposition}\label{prop:fine-hecke-orbit}
The complete $p$-Hecke fibre of $P_0$ is
\[
 \{x^u:u\in C_p\},
\]
with each element of $C_p$ occurring once.  If $C\in C_q^{(3)}$, the
idele $u_q(C)$ fixes $x$.  Moreover, for every $u\in C_p$ and every
prime above $p$,
\begin{equation}\label{eq:hecke-orbit-reduction}
 \overline{\varphi^c(x^u)}
 =F_{\bar E_{\bar\varpi}/\mathbf F_{p^2}}(\bar Q_q).
\end{equation}
\end{proposition}

\begin{proof}
The local torus modulo rational units acts simply transitively on the
$p+1$ lines in $\mathbf F_p^2$: identify $\mathbf F_p^2$ with
$\mathbf F_{p^2}$ and let $\mathbf F_{p^2}^\times$ act by
multiplication.  By the canonical-model reciprocity law
\eqref{eq:reflex-reciprocity}, the Galois action of a local class $u$
is represented by $\bar u^{-1}=u/N(u)$, which is the same class modulo
$\mathbf Q_p^\times$.  Hence the $C_p$-indexing agrees with the
complete Hecke fibre and with its Galois action.  The quotient
$C_p\to G_p$ identifies labels differing by $\Delta_p$ only after
forgetting part of the level structure; the full level on $x$ keeps the
$p+1$ labels distinct (cf.~\Cref{lem:two-cm-orders}).

For the assertion at $q$, canonical-model reciprocity represents the
action of $u_q(C)$ at $q$ by $(1,C^{-1})$.  When $C\in C_q^{(3)}$, this belongs to the level
stabilizer at $q$, while the other local components are trivial.  Thus
it fixes $x$.

To prove \eqref{eq:hecke-orbit-reduction}, work on a prime-to-$p$ cover carrying the
universal elliptic curve and level structure.  The $\mathbf Q$-model of
$X_\Gamma$ fixed above has a good integral model at $p$ defined over
$\mathbf Z_p$; thus its special fibre satisfies the canonical
identification $\bar X_\Gamma^{(p)}=\bar X_\Gamma$.  The schematic
closure of a cyclic subgroup of order $p$ is finite flat of rank $p$.
Since the reduction is supersingular, this closure has special fibre
equal to the unique rank-$p$ subgroup scheme
$\ker F_E\simeq\alpha_p$.  Hence the Hecke
correspondence reduces to relative Frobenius, with level carried along:
\[
 (\bar E,\bar\eta)\longmapsto
 (\bar E^{(p)},F_{\bar E}(\bar\eta)).
\]
Consequently
\[
\begin{CD}
\bar X_\Gamma @>{F_X}>> \bar X_\Gamma^{(p)}\\
@V{\bar\varphi}VV @VV{\bar\varphi^{(p)}=\bar\varphi^{\,c}}V\\
\bar E_{\bar\varpi} @>{F_{E_{\bar\varpi}}}>> \bar E_\varpi.
\end{CD}
\]

Write $T$ for the Fourier parameter at $\infty$.  For either affine
co-ordinate, let $\sum_m a_mT^m$ be its Laurent expansion along
$\varphi$, with $a_m\in\mathcal O_{K,(p)}$.  By our
coefficient-conjugate normalization, the corresponding expansion
along $\varphi^c$ is $\sum_m\overline{a_m}T^m$.  Since $p$ is inert
in $K$, the non-trivial automorphism of $K/\mathbf Q$ induces the
$p$-power map on $\mathcal O_K/(p)=\mathbf F_{p^2}$.  Thus
\[
 a_m^p\equiv\overline{a_m}\pmod p
 \qquad\text{for every }m.
\]
The expansion of $\bar\varphi^{(p)}$ is obtained by applying this
$p$-power map to the reduced coefficients, with $T$ unchanged.
Hence its two co-ordinate expansions agree with those of
$\bar\varphi^c$.  The $q$-expansion principle therefore gives
\[
 \bar\varphi^{(p)}=\bar\varphi^c,
\]
which proves \eqref{eq:hecke-orbit-reduction}.

\end{proof}

The conductor-$p$ modification requires a representative of the
$q$-level action whose conjugate by $\mathbf h_p$ remains integral.

\begin{lemma}\label{lem:p-compatible-q-action}
Let $B_0\in\Gamma_0(N)$ and $h=\operatorname{diag}(p,1)$, where
$p\nmid N$.  There is a matrix $B\in B_0\Gamma(N)$ such that
$B\equiv I_2\pmod p$ and
\[
 B'=hBh^{-1}\in\Gamma_0(N),\qquad d(B')\equiv d(B_0)\pmod q.
\]
The induced actions are
\[
 \varphi((B')^{-1}z)=[\xi_q(d(B_0))^{-1}]\varphi(z),\qquad
 \varphi^c((B')^{-1}z)=[\overline{\xi_q(d(B_0))}^{\,-1}]\varphi^c(z).
\]
\end{lemma}

\begin{proof}
Choose $\gamma\in\operatorname{SL}_2(\mathbf Z)$ with
$\gamma\equiv I_2\pmod N$ and
$\gamma\equiv B_0^{-1}\pmod p$, and put $B=B_0\gamma$.
The Chinese remainder theorem gives the required class modulo $Np$,
and the reduction map from $\operatorname{SL}_2(\mathbf Z)$ is
surjective.  Thus $B\equiv I_2\pmod p$, and for
$B=\left(\begin{smallmatrix}a&b\\c&d\end{smallmatrix}\right)$,
\[
 B'=\begin{pmatrix}a&pb\\ c/p&d\end{pmatrix}\in\Gamma_0(N),
\]
because $p\nmid N$.  Its lower-right residue modulo $q$ is unchanged.
The two displayed action formulas follow from
\Cref{lem:modular-functions}(i),(ii).
\end{proof}

We next determine the action of the cubic subgroup $\Delta_p$ on the
two elliptic images.

\begin{lemma}\label{lem:exact-residual-action}
On the $\varphi^c$-images of the complete Hecke orbit, $\Delta_p$ acts
trivially.  On the $\varphi$-images, the generator $\delta_p$
represented by the idele whose $p$-component is $\omega$ and whose
other components are $1$ acts by $[\omega^2]$.
\end{lemma}

\begin{proof}
The principal idele defined by the global unit $\omega$ relates
$\delta_p$ to the corresponding components at $3$ and $q$.  Under
\eqref{eq:reflex-reciprocity}, these components can be represented after
the conductor-$p$ modification by matrices $S'$ and $B'$ as follows.
If $\tau_*=\tau_0$, take $S=\mathcal A_3$; if
$\tau_*=\tau_1$, take $S=B_3^{4p}$.  In the latter case,
$3\mid4p-2$ gives $B_3^{4p}\in B_3^2\Gamma(N)$, so $S$ represents the
required $3$-component; moreover, $p\mid4p$, so its conjugate below is
integral.  The conductor-$p$ modification is on the right by
$h^{-1}$, and the identity
$Sh^{-1}=h^{-1}(hSh^{-1})$ shows that the relevant conjugation is
$S'=hSh^{-1}$.  Thus, with $h=\operatorname{diag}(p,1)$,
\[
 S'=hSh^{-1}=\mathcal A_3^p\quad\text{or}\quad B_3^4.
\]
In either case $(S')^{-1}$ acts by $[\omega]$ on both $\varphi$ and
$\varphi^c$ (cf.~\Cref{lem:modular-functions}(ii)).  Thus the inverse
$\delta_3$-contribution occurring in the principal-unit relation changes
under the modification: the unmodified representative $S_a^{-1}$ acts
by $[\omega^2]$, whereas $(S')^{-1}$ acts by $[\omega]$.  Equivalently,
$\delta_3$ itself acts by $[\omega^2]$ after the modification, as used
in \Cref{prop:heegner-local-weights}.

Direct multiplication gives $IS\in\Gamma_0(N)$ and
$d(IS)\equiv\omega_q^2\pmod q$.  Take $B_0=(IS)^{-1}$ and apply
\Cref{lem:p-compatible-q-action} to choose
$B\in B_0\Gamma(N)$ with $B\equiv I_2\pmod p$.  Then
\[
 ISB\in\Gamma(N),\qquad
 B'=hBh^{-1}\in\Gamma_0(N),\qquad
 d(B')\equiv d(B_0)\equiv\omega_q\pmod q.
\]
Put $\gamma=B'^{-1}(S')^{-1}$.  In the reciprocity formula
\eqref{eq:reflex-reciprocity} for $x^{\delta_p}$, apply the rational
matrix $\gamma hI^{-1}$.  Since $I^{-1}\tau_*=\tau_*$, the
archimedean component becomes $\gamma h\tau_*$.  The finite component
becomes $\gamma$ at $p$ and $h(ISB)^{-1}$ at every other prime.
These components lie in $\mathcal U$: at $p$ both $B'$ and $S'$
are integral and invertible, while away from $p$ one has
$h\in\mathcal U_\ell$ and $ISB\in\Gamma(N)$.  Consequently
\[
 x^{\delta_p}=[\gamma h\tau_*,1].
\]
As $x=[h\tau_*,1]$, the action on the classical variable is therefore
$\gamma=B'^{-1}(S')^{-1}$.

The first column below records the action of $(S')^{-1}$, which is
$[\omega]$ on both factors.  For the second column,
$\xi_q(d(B_0))=\omega^2$, so
\Cref{lem:p-compatible-q-action} gives $[\omega]$ on the
$\varphi$-factor and $[\omega^2]$ on the $\varphi^c$-factor.
The last column records their composite, which the preceding
adelic calculation identifies with the action of $\delta_p$:

\[
\begin{array}{c|cc|c}
 & (S')^{-1}&B'^{-1}&B'^{-1}(S')^{-1}\\ \hline
\varphi&[\omega]&[\omega]&[\omega^2]\\
\varphi^c&[\omega]&[\omega^2]&1.
\end{array}
\]
Thus $\delta_p$ acts as claimed.  Since $C_p$ is abelian, the same
formulas hold on every translate of $x$.
\end{proof}

Set
\[
 Q^c_{q,p}:=\varphi^c(x).
\]
For $C\in C_q^{(3)}$, put
\[
 Q_C:=\varphi^c\bigl(x^{\sigma_{u_q(C)}}\bigr).
\]
By \Cref{lem:conductor-p-field}, the points before projection are
defined over $F_qH_{3p}$; by \Cref{lem:exact-residual-action}, their
$\varphi^c$-images descend through $\Delta_p$.  Hence
\begin{equation}\label{eq:hecke-field}
 Q_C=Q^c_{q,p}\in E_\varpi(F_qH_p)
 \qquad(C\in C_q^{(3)}).
\end{equation}
Although the $\varphi^c$-image is $\Delta_p$-invariant, the complete
Hecke fibre itself has $p+1$ points.  Accordingly, the traces below are
taken over $C_p$, with the corresponding multiplicities.

\subsection{The vanishing Hecke trace}
\begin{lemma}\label{lem:vanishing-hecke-trace}
One has 
\begin{equation}\label{eq:complete-hecke-trace}
 N_{C_p}Q^c_{q,p}=O.
\end{equation}
\end{lemma}

\begin{proof}
Let $\varphi_J^c:J(X_\Gamma)\to E_\varpi$ be the homomorphism induced by
the map normalized at $\infty$.  The induced Hecke action is through
the eigenform $g_q^c$.  Since $p\equiv2\pmod3$ and $p\nmid N$, CM gives
$a_p(g_q^c)=0$, and hence $\varphi_J^cT_p=0$.  The divisor identity yields
\[
\sum_{u\in C_p}\varphi^c(x^u)
=\varphi_J^cT_p([P_0]-[\infty])
 +\varphi_J^c\bigl(T_p[\infty]-(p+1)[\infty]\bigr).
\]
The first term vanishes by the eigenvalue calculation.  On $X_0(N)$,
representatives for the prime-to-$N$ correspondence are
\[
\begin{pmatrix}1&r\\0&p\end{pmatrix}\ (0\le r<p),
\qquad
\begin{pmatrix}p&0\\0&1\end{pmatrix},
\]
and all send $\infty$ to $\infty$.  Each cusp upstairs, counted with
multiplicity, therefore has the form $\gamma_i\infty$ with
$\gamma_i\in\Gamma_0(N)$.  The correction is
\[
 \sum_{i=1}^{p+1}\varphi^c(\gamma_i\infty)=O
\]
(cf.~\Cref{lem:modular-functions}(i)).  Finally
\Cref{prop:fine-hecke-orbit} identifies this fibre with the
$C_p$-orbit, proving \eqref{eq:complete-hecke-trace}.
\end{proof}

To form the boundary from \eqref{eq:complete-hecke-trace}, we choose a
first $\lambda$-division before taking the $C_p$-trace.  The result lies in
$E_\varpi[\lambda]$, and reduction identifies it by the Frobenius
calculation of \Cref{sec:fermat-cover}.

\subsection{The division boundary}
We first compare first $\lambda$-divisions on the two reduced CM
factors.

\begin{lemma}\label{lem:boundary-frobenius-transfer}
Let $F_{\bar E_{\bar\varpi}/\mathbf F_{p^2}}$ denote relative
$p$-Frobenius, with target the $p$-twist of the source, and put
\[
 \operatorname{Fr}^{\mathrm{rel}}_p
 =F_{\bar E_{\bar\varpi}/\mathbf F_{p^2}}:
 \bar E_{\bar\varpi}\longrightarrow\bar E_\varpi.
\]
\begin{enumerate}[label=\textup{(\roman*)}]
\item For every $\alpha\in\mathcal O_K$,
\[
 \operatorname{Fr}^{\mathrm{rel}}_p[\alpha]
 =[\bar\alpha]\operatorname{Fr}^{\mathrm{rel}}_p;
\]
\item One has
\[
 \overline{Q^c_{q,p}}=\operatorname{Fr}^{\mathrm{rel}}_p(\bar Q_q);
\]
\item If
\[
 u_\lambda=\frac{\bar\lambda}{\lambda}=-\omega^2,
 \qquad
 \operatorname{Fr}^{(\lambda)}_p
 =[u_\lambda]\circ\operatorname{Fr}^{\mathrm{rel}}_p,
\]
then
\begin{equation}\label{eq:switching-sharp-lambda}
 [\lambda]\operatorname{Fr}^{(\lambda)}_p
 =\operatorname{Fr}^{\mathrm{rel}}_p[\lambda].
\end{equation}
Consequently $\operatorname{Fr}^{(\lambda)}_p$ maps
$[\lambda]^{-1}(\bar Q_q)$ into
$[\lambda]^{-1}(\overline{Q^c_{q,p}})$, and its restriction to the
$\lambda$-kernel is injective.
\end{enumerate}
\end{lemma}

\begin{proof}
Since $p\equiv2\pmod3$, relative $p$-Frobenius conjugates the CM action,
which gives (i).  Part (ii) is exactly the reduction formula
\eqref{eq:hecke-orbit-reduction}.  For (iii), the definition of
$u_\lambda$ and (i) give
\[
 [\lambda][u_\lambda]\operatorname{Fr}^{\mathrm{rel}}_p
 =[u_\lambda\lambda]\operatorname{Fr}^{\mathrm{rel}}_p
 =[\bar\lambda]\operatorname{Fr}^{\mathrm{rel}}_p
 =\operatorname{Fr}^{\mathrm{rel}}_p[\lambda].
\]
The final assertion follows because
$\deg\operatorname{Fr}^{(\lambda)}_p=p$
and $p\ne3$.
\end{proof}

Multiplication by $\lambda$ is finite \'{e}tale on the good integral
model at $p$.  Choose a finite unramified
extension $B'_p/B_p$ such that
\[
 \mathscr H_p=L_pB'_p
\]
contains the first $\lambda$-divisions under consideration.  By
\Cref{lem:ring-class-p-ramification}(iii),
\[
\operatorname{Gal}(\mathscr H_p/B'_p)=C_p,
\]
and this group acts trivially on the residue field.
If $R$ is a first $\lambda$-division of $Q^c_{q,p}$ and
$\delta\in\Delta_p$, then $R^\delta-R\in E_\varpi[\lambda]$ and
$\overline{R^\delta-R}=O$.  Since $p\ne3$, specialization is injective
on $E_\varpi[\lambda]$, so $R^\delta=R$.  The full $C_p$-trace remains
essential: it records the multiplicity in the complete Hecke fibre,
and it is the full unweighted trace $N_{C_p}Q^c_{q,p}$ that vanishes.

\begin{definition}\label{def:complete-division-boundary}
Fix a prime $\mathfrak P$ above $p$.  For
$R\in E_\varpi(\mathscr H_p)$ satisfying
$[\lambda]R=Q^c_{q,p}$, define the \emph{first-division boundary at
$\mathfrak P$} by
\begin{equation}\label{eq:complete-boundary-definition}
 \partial^c_{q,p,\mathfrak P}:=N_{C_p}R.
\end{equation}
\end{definition}

\begin{lemma}\label{lem:boundary-well-defined}
\begin{enumerate}[label=\textup{(\roman*)}]
\item One has
$\partial^c_{q,p,\mathfrak P}\in E_\varpi[\lambda](K)$.
\item It is independent of the chosen first division $R$, of the
chosen member of the $C_p$-orbit, and of the unramified field used to
split the first divisions.
\end{enumerate}
\end{lemma}

\begin{proof}
By \Cref{lem:vanishing-hecke-trace},
\[
 [\lambda]\partial^c_{q,p,\mathfrak P}
 =N_{C_p}Q^c_{q,p}=O.
\]
The $\lambda$-kernel is already $K$-rational, proving (i).  If $R$ is
replaced by $R+d$ with $d\in E_\varpi[\lambda](K)$, the boundary changes
by $(p+1)d=0$.  Replacing $Q^c_{q,p}$ by a conjugate merely permutes the
$C_p$-trace.  Finally, two unramified splitting fields may be compared
inside their common unramified compositum.  This proves (ii).
\end{proof}

We can now identify the boundary explicitly.

\begin{proposition}\label{prop:division-boundary}
Under \eqref{eq:intro-ray-condition}, for every prime
$\mathfrak P$ above $p$,
\begin{equation}\label{eq:explicit-global-boundary}
 \partial^c_{q,p,\mathfrak P}=\varepsilon_{q,p}T_E^\sharp\ne O.
\end{equation}
In particular, the left side is independent of $\mathfrak P$.
\end{proposition}

\begin{proof}
By finite \'{e}tale lifting and
\Cref{lem:boundary-frobenius-transfer}(iii), choose a division $R$ whose
reduction is
\[
 \operatorname{sp}_{\mathfrak P}(R)
 =\operatorname{Fr}^{(\lambda)}_p
   (\operatorname{sp}_{\mathfrak P}R_q^E).
\]
Every element of $C_p$ acts trivially on the residue field of
$\mathscr H_p$, so all $R^u$ have the same reduction.  Hence
\[
\begin{aligned}
 \operatorname{sp}_{\mathfrak P}(\partial^c_{q,p,\mathfrak P})
 &=[p+1]\operatorname{Fr}^{(\lambda)}_p
       (\operatorname{sp}_{\mathfrak P}R_q^E)\\
 &=\operatorname{Fr}^{(\lambda)}_p
       ([p+1]\operatorname{sp}_{\mathfrak P}R_q^E)\\
 &=\varepsilon_{q,p}\operatorname{Fr}^{(\lambda)}_p
       (\operatorname{sp}_{\mathfrak P}T_E)\\
 &=\varepsilon_{q,p}\operatorname{sp}_{\mathfrak P}(T_E^\sharp).
\end{aligned}
\]
The third equality is \eqref{eq:affine-obstruction-elliptic}.  For the last,
$\operatorname{Fr}^{\mathrm{rel}}_p(T_E)=(0,-\varpi/2)$ after reduction and
$[u_\lambda]=[-1]$ on the $\lambda$-kernel, since $[\omega]$ acts
trivially there.  Hence
$\operatorname{Fr}^{(\lambda)}_p(T_E)=T_E^\sharp$.

Since $p\ne3$, specialization is injective on
$E_\varpi[\lambda](K)$, so the displayed equality lifts to
\eqref{eq:explicit-global-boundary}.  The right-hand side depends only
on $\eta_q$ and $\alpha_q(p)$ through
\eqref{eq:affine-sign-table}, and is therefore independent of
$\mathfrak P$.
\end{proof}

\begin{definition}\label{def:global-division-boundary}
Under the condition \eqref{eq:intro-ray-condition}, define
\[
 \partial^c_{q,p}:=\partial^c_{q,p,\mathfrak P}
 =\varepsilon_{q,p}T_E^\sharp
\]
for any prime $\mathfrak P$ above $p$.  We call this the
\emph{division boundary} of the conductor-$p$ orbit.
\end{definition}

\begin{remark}\label{rem:tate-boundary}
The boundary has a Tate-cohomological interpretation.  Inside
$E_\varpi(\mathscr H_p)$, let $M$ be the $\mathcal O_K[C_p]$-module
generated by $Q^c_{q,p}$ and put $\widetilde M=[\lambda]^{-1}M$.
The unramified splitting field $\mathscr H_p$ is chosen so that the
first $\lambda$-divisions of the elements of $M$ lie in a common
$C_p$-module.  Thus multiplication by $\lambda$ is surjective in the
exact sequence
\[
 0\longrightarrow E_\varpi[\lambda]\longrightarrow\widetilde M
 \xrightarrow{[\lambda]}M\longrightarrow0.
\]
The norm-zero point $Q^c_{q,p}$ defines a class in
$\widehat H^{-1}(C_p,M)$.  Its image under the connecting map to
$\widehat H^0(C_p,E_\varpi[\lambda])=E_\varpi[\lambda]$ is
$N_{C_p}R=\partial^c_{q,p}$.  The last equality of groups uses the
trivial action on the kernel and the fact $3\mid|C_p|$.  By
\Cref{lem:boundary-well-defined}, this representative is unchanged if
the unramified splitting field is enlarged or replaced.
\end{remark}

\section{From the division boundary to cubic points}\label{sec:cubic-components}
Using the non-zero division boundary of
\Cref{prop:division-boundary}, this section proves that both non-trivial
cubic character components of the conductor-$p$ orbit are non-torsion
(cf.~\Cref{thm:cubic-points}).  We first trace away the prime-to-$3$ part of
the orbit while retaining its boundary.  On the resulting cyclic
$3$-power quotient, \Cref{thm:cubic-projection-criterion} compares the
norm integrally with each cubic projector.  If either cubic component
vanished, this comparison would give a further
$\lambda$-division of a non-zero $3$-torsion point over $F_q$, which the
Kummer obstruction of \Cref{lem:three-torsion-obstruction} excludes.
A final local argument at $p$ rules out torsion.

\subsection{The three-primary quotient and its boundary}
Throughout this section we retain \eqref{eq:intro-ray-condition}, so
\(\alpha_q(p)\ne1\).  Put
\[
 v=\operatorname{ord}_3(p+1)\ge2,\qquad n=3^{v-1},
\]
and write
\[
 C_p=C_p^{(3')}\times C_{p,3},
 \qquad |C_p^{(3')}|=\frac{p+1}{3^v},\qquad |C_{p,3}|=3n.
\]
We first isolate the field and group through which the cubic characters
factor.

\begin{lemma}\label{lem:three-primary-quotient}
Let $H_p^{(3)}$ be the degree-$n$ subfield of $H_p/K$, and put
\[
 \mathscr L=F_qH_p^{(3)}.
\]
\begin{enumerate}[label=\textup{(\roman*)}]
\item One has $F_q\cap H_p^{(3)}=K$, and so 
\[
 \operatorname{Gal}(\mathscr L/F_q)
 \simeq\operatorname{Gal}(H_p^{(3)}/K)
 \simeq\mathbf Z/n\mathbf Z.
\]
\item If $G=\operatorname{Gal}(\mathscr L/F_q)$, restriction of the
$3$-primary part of the conductor-$p$ class group gives an exact
sequence
\[
 1\longrightarrow\Delta_p\longrightarrow C_{p,3}
 \longrightarrow G\longrightarrow1.
\]
\end{enumerate}
\end{lemma}

\begin{proof}
The extension $F_q/K$ is ramified at
$\mathfrak q$, whereas $H_p^{(3)}/K$, being contained in the ring-class
field of conductor $p$, is unramified there.  Since $[F_q:K]=3$, their
intersection is $K$, proving (i).  The group $C_p$ is cyclic of order
$p+1$, while $G_p=C_p/\Delta_p$ and $|\Delta_p|=3$ by
\Cref{lem:ring-class-p-ramification}(i).  The $3$-primary quotient of
$G_p$ therefore has order $n=3^{v-1}$ and cuts out $H_p^{(3)}$.
Restriction from $C_{p,3}$ to this quotient has kernel $\Delta_p$,
which proves (ii).
\end{proof}

\begin{definition}\label{def:three-primary-trace}
With $C_q^{(3)}$ and $h_q=(q-1)/3$ as in
\eqref{eq:hq-definition}, define the \emph{three-primary trace}
\begin{equation}\label{eq:three-primary-trace}
 \mathcal Z^c
 :=\sum_{C\in C_q^{(3)}}N_{C_p^{(3')}}Q_C
 =[h_q]N_{C_p^{(3')}}Q^c_{q,p}.
\end{equation}
Since the trace by $C_p^{(3')}$ is fixed by that subgroup and
$Q_C\in E_\varpi(F_qH_p)$ by \eqref{eq:hecke-field}, one has
$\mathcal Z^c\in E_\varpi(F_qH_p^{(3)})=E_\varpi(\mathscr L)$.
\end{definition}

The essential point is that tracing away the prime-to-$3$ part does not
lose the boundary constructed in \Cref{sec:hecke-projector}.

\begin{lemma}\label{lem:three-primary-boundary}
Fix the local embedding at $p$ used in the definition of the division
boundary.  Choose first $\lambda$-divisions $R_C$ of the points $Q_C$
in a common unramified splitting field as in
\Cref{def:complete-division-boundary}, and put
\[
 \mathcal R^c
 :=\sum_{C\in C_q^{(3)}}N_{C_p^{(3')}}R_C.
\]
\begin{enumerate}[label=\textup{(\roman*)}]
\item One has $[\lambda]\mathcal R^c=\mathcal Z^c$.
\item $\Delta_p$ acts trivially on $\mathcal Z^c$, and the induced action
of $C_{p,3}/\Delta_p$ is the $G$-action on $\mathcal Z^c$. 
\item One has
\begin{equation}\label{eq:local-boundary}
 N_{C_{p,3}}\mathcal R^c
 =h_q\partial^c_{q,p}\ne O.
\end{equation}
\item The norm $S=N_G\mathcal Z^c$ satisfies
\begin{equation}\label{eq:norm-boundary-identity}
 S\in E_\varpi[3](F_q),\qquad
 [\lambda]S=[-\omega h_q]\partial^c_{q,p}\ne O.
\end{equation}
\end{enumerate}
\end{lemma}

\begin{proof}
Part (i) follows from the definition.  For (ii), the subgroup
$\Delta_p$ acts trivially on the $\varphi^c$-images by
\Cref{lem:exact-residual-action}, while
\Cref{lem:ring-class-p-ramification}(iii) identifies the local
$C_p$-action with the original Artin action after unramified base
change.  Thus the quotient action is the one in
\Cref{lem:three-primary-quotient}(ii).  Since
$C_p=C_p^{(3')}\times C_{p,3}$, transitivity of the trace and
\Cref{lem:boundary-well-defined} give
\[
 N_{C_{p,3}}\mathcal R^c
 =\sum_{C\in C_q^{(3)}}N_{C_p}R_C
 =h_q\partial^c_{q,p}.
\]
Now $h_q\equiv1\pmod3$ by \eqref{eq:hq-definition}, while
$\partial^c_{q,p}\in E_\varpi[\lambda]\setminus\{O\}$ by
\Cref{prop:division-boundary}.  This proves (iii).

For (iv), choose a section $s:G\to C_{p,3}$ and put
$T'=\sum_{g\in G}s(g)\mathcal R^c$.  Then $[\lambda]T'=S$.
If $\delta$ generates $\Delta_p$, the $\Delta_p$-invariance of
$\mathcal Z^c$ gives $\delta\mathcal R^c=\mathcal R^c+d$ with
$d\in E_\varpi[\lambda]$.  Since $\Delta_p$ fixes this kernel,
$\delta^r\mathcal R^c=\mathcal R^c+[r]d$; summing over $r=0,1,2$
gives $N_{\Delta_p}\mathcal R^c=[3]\mathcal R^c$.  Hence
$[3]T'=h_q\partial^c_{q,p}$ by (iii).  Using
$\lambda^2=-3\omega$, we obtain
\[
 [\lambda]S=[\lambda^2]T'=[-\omega h_q]\partial^c_{q,p}.
\]
Since $h_q\equiv1\pmod3$ and $[\omega]$ acts trivially on the
$\lambda$-kernel, this is also $-\partial^c_{q,p}$.
Also $[3]S=[\lambda][3]T'=O$, and $S$ is $G$-fixed by definition.
This proves (iv).
\end{proof}

\subsection{The obstruction to a further \texorpdfstring{$\lambda$}{lambda}-division}
The norm in \eqref{eq:norm-boundary-identity} lies in
$E_\varpi[3]\setminus E_\varpi[\lambda]$.  The divided projector
argument will show that a vanishing cubic component would make this
norm $\lambda$-divisible over $F_q$.  We first record the obstruction
to such a division.

\begin{lemma}\label{lem:three-torsion-obstruction}
In $F_q^\times/F_q^{\times3}$ one has
\begin{equation}\label{eq:three-torsion-non-cube}
 [\varpi\omega]\ne1,\qquad [\varpi\omega^2]\ne1.
\end{equation}
Equivalently, no point of
$E_\varpi[3]\setminus E_\varpi[\lambda]$ is
$\lambda$-divisible in $E_\varpi(F_q)$.
\end{lemma}

\begin{proof}
Since $F_q=K(\sqrt[3]{\varpi})$, one has $[\varpi]=1$ in
$F_q^\times/F_q^{\times3}$.  If either class in
\eqref{eq:three-torsion-non-cube} were trivial, then $\omega$ would be a
cube in $F_q$, so $K(\zeta_9)\subseteq F_q$.  Both extensions have
degree three over $K$, but $F_q/K$ is ramified at $\mathfrak q$ whereas
$K(\zeta_9)/K$ is unramified there, a contradiction.

Writing $t^3=\varpi$, the points of
$E_\varpi[3]\setminus E_\varpi[\lambda]$ have co-ordinates
\[
 x=-\zeta t^2,\qquad
 y=\pm\varpi\sqrt{-3}/2,
 \qquad \zeta\in\mu_3.
\]
Their descent classes are $[\varpi\omega]$ and
$[\varpi\omega^2]$ by \Cref{lem:exact-descent-kernel}; hence neither is
$\lambda$-divisible over $F_q$.
\end{proof}

\subsection{An integral norm--projector comparison}
Let $G=\langle X\rangle$ be cyclic of order $n_0$ with $3\mid n_0$.
\begin{definition}
For $j=1,2$ define
\begin{equation}\label{eq:integral-projector-operators}
 N_G=\sum_{k=0}^{n_0-1}X^k,
 \qquad
 \Pi_{j,n_0}=\sum_{k=0}^{n_0-1}\omega^{-jk}X^k,
 \qquad
 \mathscr D_{j,n_0}=\frac{N_G-\Pi_{j,n_0}}{\lambda}.
\end{equation}
\end{definition}
The point is that the last operator is integral: the trivial and cubic
characters become congruent modulo the prime $(\lambda)$ above $3$.

\begin{lemma}\label{lem:integral-projector-identities}
For $j=1,2$ one has $\mathscr D_{j,n_0}\in\mathcal O_K[G]$ and
\begin{equation}\label{eq:torsion-obstruction-group-ring}
 \lambda\mathscr D_{j,n_0}=N_G-\Pi_{j,n_0},\qquad
 (X-1)\mathscr D_{j,n_0}
 =-\frac{\omega^j-1}{\lambda}\Pi_{j,n_0}.
\end{equation}
\end{lemma}

\begin{proof}
Since $\omega\equiv1\pmod\lambda$, every coefficient
$1-\omega^{-jk}$ is divisible by $\lambda$, proving integrality.  The
first identity is the definition.  For the second, multiply by $X-1$
and use $\omega^{-jn_0}=1$, which follows from $3\mid n_0$.
\end{proof}

We now formulate an abstract criterion which will be used for the conductor-$p$
orbit: the non-zero unweighted boundary implies that the cubic projections
are non-zero.

\begin{theorem}[Integral boundary criterion]
\label{thm:cubic-projection-criterion}
Let $F_0\supset K$, let $a\in K^\times$, and let $L/F_0$ be cyclic with
$G=\langle X\rangle$ of order $n_0$, where $3\mid n_0$.  Suppose that a
finite cyclic group $\widetilde G$ fits into
\[
1\longrightarrow\Delta\simeq \mathbf Z/3\mathbf Z
\longrightarrow\widetilde G\longrightarrow G\longrightarrow1
\]
and acts $\mathcal O_K$-linearly on $E_a(\mathscr H)$ for a field
$\mathscr H\supset L$.  Assume that its action on $E_a(L)$ induces
the natural Galois action of $G=\operatorname{Gal}(L/F_0)$ and that
it fixes $E_a[\lambda]$ pointwise.  Let
\[
 R\in E_a(\mathscr H),\qquad Z=[\lambda]R\in E_a(L),\qquad
 \partial=N_{\widetilde G}R\in E_a[\lambda]\setminus\{O\}.
\]
If $[a\omega]$ and $[a\omega^2]$ are non-trivial in
$F_0^\times/F_0^{\times3}$, then
\[
 \Pi_{j,n_0}Z\ne O\qquad(j=1,2).
\]
\end{theorem}

\begin{proof}
Fix $j$ and suppose that $\Pi_{j,n_0}Z=O$.  Set
$T_j=\mathscr D_{j,n_0}Z$.  By
\Cref{lem:integral-projector-identities},
\[
 (X-1)T_j=O,
 \qquad
 [\lambda]T_j=N_GZ=:S.
\]
Thus $T_j\in E_a(F_0)$.

Choose a set-theoretic section $s:G\to\widetilde G$ and put
$T'=\sum_{g\in G}s(g)R$.  Then $[\lambda]T'=S$, so
$T_j-T'\in E_a[\lambda]$ and $[3]T_j=[3]T'$.  If $\delta$ generates
$\Delta$, then $\delta R=R+t$ for some $t\in E_a[\lambda]$, because
$\Delta$ acts trivially on $Z$.  Since $\Delta$ fixes
$E_a[\lambda]$ pointwise, one has
$\delta^rR=R+[r]t$ for $r=0,1,2$, and therefore
\[
 N_\Delta R=[3]R.
\]
The elements $s(g)\delta^r$ run through $\widetilde G$, and therefore
\begin{equation}\label{eq:torsion-obstruction-boundary}
 [3]T_j=N_{\widetilde G}R=\partial\ne O.
\end{equation}

It follows that $[3]S=[\lambda][3]T_j=O$.  Moreover
$S\notin E_a[\lambda]$: otherwise $[\lambda]^2T_j=O$, and the identity
$3=-\omega^2\lambda^2$ would contradict
\eqref{eq:torsion-obstruction-boundary}.  Hence
\[
 S\in E_a[3]\setminus E_a[\lambda].
\]
By \Cref{lem:exact-descent-kernel}, the descent class of $S$ is
$[a\omega]$ or $[a\omega^2]$, and is non-trivial by hypothesis.  This
contradicts $S=[\lambda]T_j$ with $T_j\in E_a(F_0)$.
\end{proof}
\begin{remark}
For the smallest quotient, $n_0=3$, one has
\[
 \mathscr D_{1,3}=-\omega^2X+X^2,
 \qquad
 \mathscr D_{2,3}=X-\omega^2X^2.
\]
For general $n_0$, the coefficient of $X^k$ in
$\mathscr D_{j,n_0}$ is $(1-\omega^{-jk})/\lambda$, which depends only
on $k$ modulo $3$.  Thus the same three coefficients recur as $k$
ranges from $0$ to $n_0-1$.  The criterion therefore depends only on
the congruence modulo $\lambda$ of the norm and the two cubic character
sums, not on the depth of the $3$-primary quotient.  The same boundary
supplies the non-vanishing for both components, and the operators
$\mathscr D_{j,n_0}$ separate the two characters.
\end{remark}
\subsection{The two cubic components}
We now apply the preceding criterion to the conductor-$p$ orbit.
Regard $\vartheta_{p,j}$ as a character of the $3$-primary quotient
$G=\operatorname{Gal}(\mathscr L/F_q)$, and choose a generator $X$ of
$G$ so that $\vartheta_{p,1}(X)=\omega$; see
\eqref{eq:ringclass-kummer-labels}.

\begin{definition}\label{def:cubic-points}
For $j=1,2$, define the \emph{cubic component}
\begin{equation}\label{eq:cubic-projector}
 P_{q,p,j}
 :=\sum_{C\in C_q^{(3)}}\sum_{\sigma\in G_p}
 [\vartheta_{p,j}(\sigma)^{-1}]Q_C^\sigma.
\end{equation}
\end{definition}

Since $\vartheta_{p,j}$ is trivial on the prime-to-$3$ part of $G_p$
and its descent to $G$ sends $X$ to $\omega^j$,
we use the operators of \eqref{eq:integral-projector-operators} with
$n_0=n$.  Then
\Cref{def:three-primary-trace} gives
\begin{equation}\label{eq:cubic-projector-expression}
 P_{q,p,j}=\Pi_{j,n}\mathcal Z^c.
\end{equation}

We also need a standard local consequence of supersingular reduction.

\begin{lemma}\label{lem:torsion-reduction-kernel}
Let $\mathfrak P$ be a prime of $\mathscr L$ above $p$.  A torsion point
of $E_\varpi(\mathscr L)$ which specializes to $O$ at $\mathfrak P$ is
$O$.
\end{lemma}

\begin{proof}
Prime-to-$p$ torsion injects under good reduction.  It therefore
suffices to exclude non-zero $p$-torsion.  The completion of
$\mathscr L$ at $\mathfrak P$ has ramification degree $n$ over the
unramified quadratic field $K_p$: $F_q/K$ is unramified at $p$, while
$H_p^{(3)}/K$ is totally ramified there of degree $n$ by
\Cref{lem:ring-class-p-ramification}. Since $K_p/\mathbf Q_p$ is unramified, multiplicativity of
ramification indices gives
\[
 e(\mathscr L_{\mathfrak P}/\mathbf Q_p)
 =e(\mathscr L_{\mathfrak P}/K_p)\,e(K_p/\mathbf Q_p)
 =n.
\]

The formal group of $E_\varpi/K_p$ has height two.  Since $K_p$ is
unramified over $\mathbf Q_p$, the coefficients of $[p](T)$ below
degree $p^2$ are divisible by $p$, and the coefficient of $T^{p^2}$
is a unit.  The Newton polygon of $[p](T)/T$ therefore gives valuation
$1/(p^2-1)$ for a non-zero $p$-torsion parameter, with $v_p(p)=1$.
Its field of definition must consequently have ramification degree
divisible by $p^2-1$.  This is the standard height-two formal
group calculation; see, for example,
\cite[Chap.~IV, \S\S6--7]{SilvermanAEC}.  But
\[
 n=3^{v-1}\le\frac{p+1}{3}<p^2-1,
\]
so such a point cannot be defined over $\mathscr L$.  Injectivity of specialization
on prime-to-$p$ torsion is standard as well; cf.
\cite[Chap.~VII, \S3]{SilvermanAEC}.
\end{proof}

We can now prove the main result of the section.

\begin{theorem}[Non-torsion of the cubic components]\label{thm:cubic-points}
Assume \eqref{eq:intro-ray-condition}.  For $j=1,2$, the point
$P_{q,p,j}$ has infinite order.
\end{theorem}

\begin{proof}
We first verify the hypotheses of
\Cref{thm:cubic-projection-criterion}.  Take
\[
 (a,F_0,L,\mathscr H)
 =(\varpi,F_q,\mathscr L,\mathscr H_p),
 \qquad
 (\widetilde G,G,\Delta)
 =(C_{p,3},\operatorname{Gal}(\mathscr L/F_q),\Delta_p),
\]
and $R=\mathcal R^c$, $Z=\mathcal Z^c$.  By
the compatible $p$-adic embeddings fixed in
\Cref{sec:hecke-projector}, the global field $\mathscr L$ is identified
with a subfield of the local field $\mathscr H_p$.  By
\Cref{lem:three-primary-quotient}, $G$ is cyclic of order
$n=3^{v-1}$ and the displayed exact sequence in the criterion is
satisfied.  Since $p\equiv8\pmod9$, one has $v\ge2$, hence $3\mid n$.
The group $C_{p,3}$ fixes $K$, so its Galois action commutes with the
$\mathcal O_K$-action; its action on $E_\varpi(\mathscr L)$ factors
through $G$ by \Cref{lem:three-primary-quotient}(ii), and it fixes
$E_\varpi[\lambda]$ pointwise because this kernel is $K$-rational.
Finally, \Cref{lem:three-primary-boundary} gives
\[
 N_{C_{p,3}}\mathcal R^c
 =h_q\partial^c_{q,p}\ne O,
\]
and \Cref{lem:three-torsion-obstruction} gives the two required Kummer
obstructions over $F_q$.  Therefore
\Cref{thm:cubic-projection-criterion} and
\eqref{eq:cubic-projector-expression} imply
\[
 P_{q,p,j}\ne O\qquad(j=1,2).
\]

It remains to exclude torsion.  For each fixed $C$, all conductor-$p$
conjugates of $Q_C$ have the same reduction by
\Cref{prop:fine-hecke-orbit}.  Since $\vartheta_{p,j}$ is non-trivial
on $G_p$,
\[
 \sum_{\sigma\in G_p}\vartheta_{p,j}(\sigma)^{-1}=0,
\]
so \eqref{eq:cubic-projector} specializes to $O$.  Since
$P_{q,p,j}\ne O$, \Cref{lem:torsion-reduction-kernel} shows that it
has infinite order.
\end{proof}

\part{The auxiliary Heegner point and Gross--Zagier}

\section{The auxiliary Heegner point}\label{sec:rankin-gz}
This section identifies the two geometric cubic components constructed
in \Cref{thm:cubic-points} with the corresponding components of the
auxiliary Heegner points in the theory of Yuan--Zhang--Zhang.  We first
realize the two normalized modular
parametrizations as projections of an $A_q$-valued morphism in the
isogeny category (cf.~\Cref{lem:automorphic-morphism}).  We then express the
Heegner functional as a finite sum and compute its local weights at
$p$.  The resulting comparison, \Cref{prop:heegner-point-comparison},
identifies each projected Heegner point with a non-zero rational
multiple of the corresponding geometric point.  Together with
\Cref{thm:cubic-points}, this proves \Cref{thm:auxiliary-heegner} and
the analogous non-vanishing for $\mathcal P_2(\Phi)$.

We retain the notation of Parts~I--II.  Thus $p\equiv8\pmod9$, the
prime $q=\varpi\bar\varpi$ satisfies
\eqref{eq:intro-ray-condition}, $P_0=[\tau_*,1]$ is the CM point chosen
in \Cref{sec:cm-kummer}, and
$\chi_{q,j}=\nu_{p^j}\psi_q^{-1}$ for $j=1,2$.

Write $J_\Gamma=J(X_\Gamma)$.  All abelian varieties and morphisms in
this section are considered up
to $\mathbf Q$-isogeny.  Thus $\operatorname{Hom}^0$ means
$\operatorname{Hom}\otimes_{\mathbf Z}\mathbf Q$ and
$A(F)_{\mathbf Q}=A(F)\otimes_{\mathbf Z}\mathbf Q$.  We choose a
representative $A_q$ of the $\mathbf Q$-isogeny class cut out by the
rational Hecke orbit $\{g_q,g_q^c\}$.  The map
$\operatorname{ev}_\varpi$, defined below in \Cref{lem:coefficient-action},
is the projection to the $E_\varpi$-factor after extending the
coefficient action through the embedding corresponding to $g_q^c$.
This rationalization concerns only the automorphic comparison in this
part.  The maps $\varphi$ and $\varphi^c$ remain the fixed, integrally
normalized maps of \Cref{sec:cm-data}.
 We choose the rational
projections from $A_q$ so that their composites are precisely those
maps in $\operatorname{Hom}^0$.

\subsection{The modular morphism and coefficient action}
Let
\[
 M_q=\operatorname{End}_{\mathbf Q}^0(A_q).
\]
Let $L_q$ be the Hecke field generated by the Fourier coefficients of
$g_q$.  Since $g_q$ is the theta series of the $K$-valued CM character
$\psi_q$, one has $L_q\subset K$.  The coefficient conjugation maps $g_q$
to $g_q^c$, and these forms are distinct: by
\Cref{lem:modular-functions}(ii), the corresponding differentials have
the distinct nebentypus characters $\xi_q$ and $\bar\xi_q$.  Hence
$L_q=K$, and $A_q$ has dimension $[L_q:\mathbf Q]=2$.

The Hecke action gives $L_q\hookrightarrow M_q$.  Eichler--Shimura
theory makes $A_q$ a simple abelian variety of strict
$\mathrm{GL}_2$-type: $M_q$ is a field and $\operatorname{Lie}(A_q)$ is
free of rank one as an $M_q$-module
(cf.~\cite[\S1B]{CST14}).  Thus $[M_q:\mathbf Q]=\dim A_q=2$, and the
preceding embedding gives $M_q=L_q\simeq K$.  We fix the two embeddings
\[
 \iota_g,\iota_c:M_q\hookrightarrow K,
 \qquad \iota_c=\bar\iota_g,
\]
corresponding to $g_q$ and $g_q^c$.
Let $\xi^{\rm Hdg}\in\operatorname{Pic}(X_\Gamma)_{\mathbf Q}$ denote
the normalized degree-one Hodge class of \cite[\S1B]{CST14}, and let
$\pi_{g_q}^{\,\infty}$ denote the finite part of the cuspidal
automorphic representation generated by $g_q$.  Right translation
on the modular-curve tower acts on morphisms by
\[
 (\mathsf R(h)f)([z,g])=f([z,gh]).
\]

The two modular parametrizations of \Cref{sec:cm-data} fit into the
same rational Hecke quotient.

\begin{lemma}\label{lem:automorphic-morphism}
There exists 
\[
 \Phi_0\in
 \operatorname{Hom}_{\mathbf Q,\xi^{\rm Hdg}}^0(X_\Gamma,A_q)
\]
and maps
\[
 \operatorname{pr}_{\bar\varpi}\in
 \operatorname{Hom}_K^0(A_q,E_{\bar\varpi}),\qquad
 \operatorname{pr}_{\varpi}\in
 \operatorname{Hom}_K^0(A_q,E_\varpi)
\]
such that, in rationalized morphism spaces,
\[
 \operatorname{pr}_{\bar\varpi}\Phi_0=\varphi,
 \qquad
 \operatorname{pr}_{\varpi}\Phi_0=\varphi^c.
\]
After clearing denominators,
$(\operatorname{pr}_{\bar\varpi},\operatorname{pr}_{\varpi})$
is a $K$-isogeny from $A_{q,K}$ to
$E_{\bar\varpi}\times E_\varpi$; the two factors carry the
$M_q$-action through $\iota_g$ and $\iota_c$, respectively.

At varying compact open levels $\mathcal V$, put
\[
\pi_{A_q}
=\varinjlim_{\mathcal V}\operatorname{Hom}_{\mathbf Q}^0(J_{\mathcal V},A_q)
=\varinjlim_{\mathcal V}
 \operatorname{Hom}_{\mathbf Q,\xi_{\mathcal V}^{\rm Hdg}}^0
 (X_{\mathcal V},A_q).
\]
Then
\begin{equation}\label{eq:automorphic-morphism-space}
 \pi_{A_q}\otimes_{M_q,\iota_g}\mathbf C
 \simeq \pi_{g_q}^{\,\infty}.
\end{equation}
In particular $\Phi=\mathsf R(\mathbf h_p^{-1})\Phi_0$ belongs to this
representation.
\end{lemma}

\begin{proof}
Choose a Hecke quotient of $J_\Gamma$ in the isogeny class
attached to $\{g_q,g_q^c\}$.  The maps $\varphi$ and $\varphi^c$
induce homomorphisms from $J_{\Gamma,K}$ whose differentials lie in the
two conjugate Hecke eigenspaces.  Hence both factor through $A_{q,K}$,
and the two differentials span its cotangent space; the product of the
resulting maps is therefore an isogeny after clearing denominators.
The $M_q$-actions on the two factors are the two coefficient embeddings.

It remains to match the base-point normalization used in
\cite[\S1B]{CST14}.  Since $\Gamma\subset\Gamma_0(9q)$, there are no
elliptic fixed points: the standard congruences
$x^2+1\equiv0\pmod{9q}$ and $x^2-x+1\equiv0\pmod{9q}$ for elliptic
points of orders two and three already have no solutions modulo $3$
and modulo $9$, respectively.  Thus the normalized Hodge class is
represented by a rational weighted sum of cusps.  By Manin--Drinfeld,
$[\infty]-\xi^{\rm Hdg}$ is torsion after clearing denominators, so the
maps normalized at $\infty$ agree with the Hodge-normalized maps in
$\operatorname{Hom}^0$.  Finally, the tower realization of $\pi_{A_q}$
and its right-translation action are those of \cite[\S1B]{CST14}, and 
tensoring through $\iota_g$ gives
\eqref{eq:automorphic-morphism-space} by multiplicity one.
\end{proof}

\begin{remark}\label{rem:jacobian-optimality}
The construction in \Cref{lem:automorphic-morphism} takes place in the
isogeny category.  This is harmless for our applications: 
a non-zero rational rescaling does not affect the non-vanishing of Heegner points.
\end{remark}

\begin{remark}\label{rem:endomorphism-algebras}
The notation $M_q=\operatorname{End}_{\mathbf Q}^0(A_q)$ refers throughout to
endomorphisms defined over $\mathbf Q$.  After base change to $K$ the
endomorphism algebra is larger.  Under the $K$-isogeny of
\Cref{lem:automorphic-morphism}, the subalgebra $M_q\simeq K$
acts on $E_{\bar\varpi}\times E_\varpi$ via 
\[
 a\longmapsto
 \bigl([\iota_g(a)],[\iota_c(a)]\bigr).
\]
Thus the assertion $M_q\simeq K$ concerns the coefficient-field action
over $\mathbf Q$, not the full endomorphism algebra after base change.
Only this distinguished copy of $K$ inside
$\operatorname{End}_K^0(A_{q,K})$ is used below.
\end{remark}

\begin{lemma}\label{lem:coefficient-action}
Put $\omega_M=\iota_g^{-1}\circ\omega_{g_q}$,
$\chi=\chi_{q,j}$, and $\chi_M=\iota_g^{-1}\circ\chi$.
\begin{enumerate}[label=\textup{(\roman*)}]
\item For $m\in M_q$, one has 
\[
\operatorname{pr}_{\bar\varpi}(mP)
=[\iota_g(m)]\operatorname{pr}_{\bar\varpi}(P),\qquad
\operatorname{pr}_\varpi(mP)
=[\iota_c(m)]\operatorname{pr}_\varpi(P).
\]
\item If
\begin{equation}\label{eq:typed-weight}
 w_j(t):=\iota_c(\chi_M(t)^{-1}),
\end{equation}
then $w_j(t)=\chi(t)$ and
$w_j|_{G_p}=\vartheta_{p,j}^{-1}$.
\item The formula
\[
 \operatorname{ev}_\varpi(P\otimes a)
 =[a]\operatorname{pr}_\varpi(P)
\]
defines a map
\[
 A_q(K^{\rm ab})_{\mathbf Q}\otimes_{M_q,\iota_c}K
 \longrightarrow E_\varpi(K^{\rm ab})_{\mathbf Q}.
\]
\item With $\Phi=\mathsf R(\mathbf h_p^{-1})\Phi_0$,
\[
 \operatorname{pr}_\varpi\Phi(P_0)=Q^c_{q,p},
\]
and for every finite idele $t$,
\begin{equation}\label{eq:heegner-reciprocity}
 \operatorname{pr}_\varpi\bigl(\Phi(P_0)^{\sigma_t}\bigr)
 =\varphi^c([\tau_*,\iota(\bar t^{-1})\mathbf h_p^{-1}]).
\end{equation}
\end{enumerate}
\end{lemma}

\begin{proof}
Part~(i) follows from the two Hecke eigendifferentials in
\Cref{lem:automorphic-morphism}.  Since $\iota_c=\bar\iota_g$ and
$\chi$ is cubic,
$w_j=\overline{\chi^{-1}}=\chi$, giving (ii); the restriction to $G_p$
is \eqref{eq:ringclass-kummer-labels}.  Part~(iii) is well defined by
the balanced tensor relation
$(mP)\otimes a=P\otimes\iota_c(m)a$.  Finally, (iv) follows from the
right-translation convention and the canonical-model reciprocity law
\eqref{eq:reflex-reciprocity} (cf.~\Cref{lem:fine-base-field}).
\end{proof}

\subsection{The Heegner point}
Put
\[
 \mathcal X
 =K^\times\widehat{\mathbf Q}^{\,\times}\backslash\widehat K^\times.
\]
The finite quotient needed below is the ring-class quotient of
conductor $pq$.

\begin{lemma}\label{lem:finite-torus-quotient}
There is a natural identification
\[
 \mathfrak T_{q,p}
 =K^\times\backslash\widehat K^\times/
 (\widehat{\mathbf Q}^{\times}\widehat{\mathcal O}_{pq}^{\times})
 \simeq (C_p\times C_q)/\Delta\mu_3,
\]
where a unit $\zeta\in\mu_3$ maps to
$([\zeta]_p,(\zeta\bmod\varpi)^{-1})$.  Moreover
\[
 \#\mathfrak T_{q,p}=\frac{(p+1)(q-1)}3.
\]
\end{lemma}

\begin{proof}
Since $h_K=1$, every finite idele is represented, modulo
$K^\times\widehat{\mathbf Q}^{\times}$, by local units at $p$ and $q$.
Reduction gives $C_p\times C_q$, and the remaining global units give
exactly the displayed diagonal image of $\mu_3$.  The order follows
from $|C_p|=p+1$ and $|C_q|=q-1$.
\end{proof}

\begin{definition}\label{def:auxiliary-heegner-point}
For $j=1,2$, define
\begin{equation}\label{eq:normalized-heegner-point}
 \mathcal P_j(\Phi)
 =\frac{\#\mathfrak T_{q,p}}{\operatorname{Vol}(\mathcal X)}
 \int_{\mathcal X}
 \chi_M(t)^{-1}\Phi(P_0)^{\sigma_t}\,dt.
\end{equation}
\end{definition}
The point denoted by $\mathcal P_q$ in the introduction is
$\mathcal P_1(\Phi)$.  The point $\mathcal P_2(\Phi)$ is defined in the
same way using the character $\chi_{q,2}$ and corresponds to the
$E_{p^2}$-factor in \eqref{eq:rankin-factorization}.
For $z\in\widehat{\mathbf Q}^{\times}$, the arithmetic action on $\Phi$
contributes $\omega_M(z)^{-1}$, while \eqref{eq:rankin-self-dual} gives
$\chi_M(z)^{-1}=\omega_M(z)$.  Hence the integrand is well defined on
$\mathcal X$.

To expand \eqref{eq:normalized-heegner-point} as a finite sum, let
$\delta_3$ be the idele with component $\omega$ at $3$ and $1$
elsewhere, and put
\[
J_3=1+3\mathcal O_{K,3},\quad
J_p=(\mathbf Z_p+p\mathcal O_{K,p})^\times,\quad
J_q=(\mathbf Z_q+q\mathcal O_{K,q})^\times,
\]
with $J_\ell=\mathcal O_{K,\ell}^\times$ for $\ell\nmid3pq$.
Let $J=\prod_\ell J_\ell$.

\begin{lemma}\label{lem:finite-heegner-average}
The integrand in \eqref{eq:normalized-heegner-point} is $J$-invariant,
and
\[
 |\mathcal X/J|=(p+1)(q-1)=3\#\mathfrak T_{q,p}.
\]
Choose a section $s_p:G_p\to C_p$.  For $d\in C_q$, let $t_{\tau,d}$
have $p$-part $s_p(\tau)$, ordered $q$-part $(d,1)$, and $3$-part $1$.
Then for every $J$-invariant function $\mathcal I$ on $\mathcal X$,
\begin{equation}\label{eq:finite-average}
\frac{\#\mathfrak T_{q,p}}{\operatorname{Vol}(\mathcal X)}
\int_{\mathcal X}\mathcal I(t)\,dt
=\frac13\sum_{\tau\in G_p}\sum_{d\in C_q}\sum_{e=0}^2
\mathcal I(t_{\tau,d}\delta_3^e).
\end{equation}
\end{lemma}

\begin{proof}
At $3$, $J_3$ lies in the CM stabilizer of
\Cref{lem:fine-base-field}; the $p$-modification does not alter that
local condition.  The character is trivial there because the
$p^j$-Kummer character is locally trivial at $3$ and $F_q/K$ has
conductor dividing $3\varpi$ (cf.~\Cref{lem:cubic-field}).  At $p$ and
$q$, after removing rational local units, $J_p$ and $J_q$ are principal
units; they fix the modified level structure and the relevant cubic
characters are tame.  The remaining maximal unit groups fix both
factors.  Hence the integrand is $J$-invariant.

Since $h_K=1$, reduction to local units gives directly
\[
 \mathcal X/J\simeq
 \widehat{\mathcal O}_K^{\times}/
 \bigl(\mathcal O_K^{\times}\widehat{\mathbf Z}^{\times}J\bigr).
\]
Modulo rational local units, the factors at $3,p,q$ have orders
$3,p+1,q-1$, respectively, and all other local factors are trivial.
The diagonal group $\mathcal O_K^\times/\{\pm1\}\simeq\mu_3$ has order
three, so
\[
 |\mathcal X/J|=\frac{3(p+1)(q-1)}3=(p+1)(q-1).
\]
Every $p$-component can be written uniquely as a chosen representative
$s_p(\tau)$ times an element of $\Delta_p$.  Multiplying by a global
unit absorbs that element into the $q$-component and the power of
$\delta_3$.  Hence the elements $t_{\tau,d}\delta_3^e$
form a complete set of representatives.  Equal Haar measure of the
cosets gives \eqref{eq:finite-average}.
\end{proof}

\subsection{Comparison with the geometric points}
Let $u_q$ be the idele of \Cref{lem:cubic-field}, put
$a_q=u_q^{-1}$, and let $a_{q,q}\in C_q$ be its class.  Then
$C_q=C_q^{(3)}\times\langle a_{q,q}\rangle$.  The local identities
needed to compare \eqref{eq:finite-average} with the points
$P_{q,p,j}$ are as follows.

\begin{proposition}\label{prop:heegner-local-weights}
For $\tau\in G_p$ and $d\in C_q$, put
\[
Q^c_{\tau,d}
=\varphi^c(x^{\sigma_{t_{\tau,d}}}).
\]
Then one has
\[
 \operatorname{ev}_\varpi\!\left(
 \chi_M(t_{\tau,d})^{-1}\Phi(P_0)^{\sigma_{t_{\tau,d}}}\right)
 =[w_j(t_{\tau,d})]Q^c_{\tau,d}.
\]
For $C\in C_q^{(3)}$ and $m=0,1,2$, one has
\[
 Q^c_{\tau,Ca_{q,q}^m}=[\omega^{2m}]Q_C^{s_p(\tau)},
 \qquad
 w_j(a_q)=\omega,
 \qquad
 w_j(u_q(C))=1,
\]
and
\[
 [w_j(\delta_3)^e](Q^c_{\tau,d})^{\sigma_{\delta_3^e}}
 =Q^c_{\tau,d}
 \qquad(e=0,1,2).
\]
Finally, the geometric action of $\Delta_p$ on the $\varphi^c$-images
is trivial and $w_j|_{G_p}=\vartheta_{p,j}^{-1}$.
\end{proposition}

\begin{proof}
The first formula follows from \Cref{lem:coefficient-action}(iii),(iv).
At $q$, \Cref{lem:p-compatible-q-action} and
\Cref{lem:modular-functions}(ii) show that $a_q$ acts on the
$E_\varpi$-factor by $[\omega^2]$, while cube residues fix the CM point.
By \Cref{lem:cubic-field}(ii) and Kummer reciprocity,
$\chi(u_q)=\omega^2$; the Kummer factor attached to $p^j$ is
unramified at $q$.
Thus $w_j(a_q)=\omega$ and $w_j(u_q(C))=1$ for
$C\in C_q^{(3)}$.

At $3$, canonical-model reciprocity
\eqref{eq:reflex-reciprocity} and the representatives used in
\Cref{lem:exact-residual-action} give the action $[\omega^2]$ on the
$\varphi^c$-factor after the $p$-modification.  The Kummer factor
attached to $p^j$ is
trivial on the local class of $\omega$ at $p$: this class lies in
$\Delta_p$, whereas the character factors through $G_p=C_p/\Delta_p$
(cf.~\eqref{eq:ringclass-kummer-labels}).  The principal idele $\omega$,
together with the $q$-calculation above, therefore gives
$\chi(\delta_3)=\omega$; hence the displayed $\delta_3$-identity.
At $p$, the last assertion is exactly
\Cref{lem:exact-residual-action} together with
\Cref{lem:coefficient-action}(ii).
\end{proof}

We can now identify the auxiliary Heegner point with the geometric
points constructed in Part~II.

\begin{proposition}\label{prop:heegner-point-comparison}
For $j=1,2$, one has
\begin{equation}\label{eq:heegner-point-comparison}
 \operatorname{ev}_\varpi\mathcal P_j(\Phi)
 =[3]P_{q,p,j}
 \qquad\text{in }E_\varpi(K^{\rm ab})_{\mathbf Q}.
\end{equation}
\end{proposition}

\begin{proof}
Apply \eqref{eq:finite-average} to the integrand in
\eqref{eq:normalized-heegner-point}.  By
\Cref{prop:heegner-local-weights}, the three terms indexed by
$\delta_3^e$ have the same $E_\varpi$-evaluation, cancelling the factor
$1/3$.  Write uniquely $d=Ca_{q,q}^m$ with
$C\in C_q^{(3)}$ and $0\le m\le2$.  Again by
\Cref{prop:heegner-local-weights},
$[w_j(a_q)^m][\omega^{2m}]=1$, so the sum over $m$ contributes $[3]$.
The remaining $p$-weight is $\vartheta_{p,j}^{-1}$, and therefore the
remaining sum is exactly \eqref{eq:cubic-projector}.  This proves
\eqref{eq:heegner-point-comparison}.
\end{proof}

\begin{proof}[Proof of \Cref{thm:auxiliary-heegner} and its $j=2$ analogue]
By \Cref{thm:cubic-points}, $P_{q,p,j}$ has infinite order for
$j=1,2$.  Hence \eqref{eq:heegner-point-comparison} shows that the
$E_\varpi$-projection of $\mathcal P_j(\Phi)$ is non-torsion, and so is
$\mathcal P_j(\Phi)$.
\end{proof}

Thus \Cref{thm:auxiliary-heegner} and its $j=2$ analogue are proved.  The proof of \Cref{thm:main} still
requires \Cref{sec:rankin-derivative}: there we match these points with
the Heegner functionals in the Yuan--Zhang--Zhang formula, verify that
the relevant Shimura curve is the present modular curve, and combine
the resulting derivative non-vanishing with the Rankin--Selberg
factorization and \Cref{prop:complementary-rank-zero}.

\section{The Gross--Zagier formula}
\label{sec:rankin-derivative}

By \Cref{prop:heegner-point-comparison,thm:cubic-points}, the two
auxiliary Heegner points are non-torsion.  This section completes the
analytic argument.  The Yuan--Zhang--Zhang generalization of the
Gross--Zagier formula converts their non-torsion into the non-vanishing
of the central derivatives of the associated 
 Rankin--Selberg $L$-functions 
(cf.~\Cref{prop:gross-zagier-nonvanishing}).  The factorizations of
\Cref{prop:rankin-factorization}, together with \Cref{prop:complementary-rank-zero}, then imply \Cref{thm:main}.

Fix $j\in\{1,2\}$ and put $\chi=\chi_{q,j}$ and
$\chi^\sigma(u)=\chi(\bar u)$.  For a finite-order $M_q$-valued
character $\xi_M$, write, up to a fixed non-zero Haar normalization,
\[
 P_\xi(f)=\int_{\mathcal X}
 \xi_M(u)f([\tau_*,\iota(u)])\,du.
\]

\subsection{The Gross--Zagier formula}
The Heegner point of \Cref{def:auxiliary-heegner-point} is defined
using Galois conjugation, while the Gross--Zagier formula is stated in
terms of the point-valued Heegner functional above.  We begin with the following comparison. 
\begin{lemma}\label{lem:heegner-character-comparison}
Up to the fixed non-zero Haar normalization\footnote{This is the
non-zero scalar determined by the choice of Haar measure on
$\mathcal X$.  It is fixed throughout.}, one has
\[
 \mathcal P_j(\Phi)=P_{\chi^\sigma}(\Phi).
\]
Moreover
\[
 L(s,g_q\times\chi^\sigma)=L(s,g_q\times\chi),
\]
and
\[
 P_\xi(\mathsf R(\iota(a))f)=\xi_M(a)^{-1}P_\xi(f)
 \qquad(a\in\widehat K^\times).
\]
\end{lemma}

\begin{proof}
For $F_f(u)=f([\tau_*,\iota(u)])$, canonical-model reciprocity
\eqref{eq:reflex-reciprocity} and central equivariance give
\[
 f(P_0)^{\sigma_u}=\omega_M(u\bar u)^{-1}F_f(u).
\]
Since $\chi_M|_{\widehat{\mathbf Q}^{\times}}=\omega_M^{-1}$, the
integrand in \eqref{eq:normalized-heegner-point} becomes
$\chi_M(\bar u)F_f(u)$, proving the first assertion.  Automorphic
induction gives
\[
 L(s,g_q\times\chi^\sigma)
 =L(s,\psi_q\chi^\sigma)L(s,\psi_q^\sigma\chi^\sigma).
\]
Conjugating the two Hecke characters interchanges these with
$L(s,\psi_q^\sigma\chi)$ and $L(s,\psi_q\chi)$, proving the second
assertion.  The last formula follows by changing variables in the
Heegner-point integral.
\end{proof}

The Gross--Zagier formula is realized on the same modular curve as in \Cref{sec:rankin-gz}: 

\begin{lemma}\label{lem:yzz-split-quaternion}
Let $\xi=\chi^\sigma$.  In the setting of the Yuan--Zhang--Zhang
generalization of the Gross--Zagier formula for $(\pi_{g_q},\xi)$, the
incoherent quaternion
algebra $\mathbb B$ satisfies
\[
 \mathbb B_v\simeq M_2(\mathbf Q_v)\quad(v<\infty),
 \qquad
 \mathbb B_\infty\simeq\mathbf H.
\]
The nearby coherent quaternion algebra obtained by switching the
archimedean invariant is therefore
\[
 B\simeq M_2(\mathbf Q).
\]
In particular, the Shimura curve occurring in the Gross--Zagier formula
is the modular-curve tower used in \Cref{sec:rankin-gz}.  The
Rankin--Selberg root number is $-1$.
\end{lemma}

\begin{proof}
By
\Cref{prop:heegner-point-comparison,thm:cubic-points,lem:heegner-character-comparison},
the Heegner point $P_\xi(\Phi)$ is non-zero.  Write $\Phi$ as a finite
sum of pure tensors over $M_q$, and choose a summand
$f=\otimes_v'f_v$ for which $Z=P_\xi(f)\ne0$.  Extend scalars in the rationalized point
space through $\iota_g:M_q\hookrightarrow\mathbf C$, and choose a
$\mathbf C$-linear functional non-zero on $Z\otimes1$.  Scalar
extension is faithful, so such a functional exists.  Its composition
with the scalar extension of $P_\xi$ is a non-zero
$\widehat K^\times$-equivariant functional of character $\xi^{-1}$.
So one has 
\[
 \operatorname{Hom}_{K_v^\times}
 (\pi_{g_q,v},\xi_v^{-1})\ne0
 \qquad(v<\infty).
\]
This is the local functional realized on
$\operatorname{GL}_2(\mathbf Q_v)$ by the modular-curve tower
\eqref{eq:automorphic-morphism-space}.  The Tunnell--Saito criterion
therefore selects the split local algebra $M_2(\mathbf Q_v)$ at every
finite $v$; see
\cite[\S3A, especially Lemma~3.1]{CST14}.  In the Gross--Zagier setup
the incoherent algebra is totally definite at infinity
\cite[\S1B]{CST14}, so $\mathbb B_\infty\simeq\mathbf H$.  Switching
the invariant at the unique real place gives a coherent quaternion
algebra split at every place, hence $B\simeq M_2(\mathbf Q)$.  The
defining local sign relation for this incoherent datum then gives global
Rankin--Selberg root number $-1$ (cf.~\cite[\S1B, \S3A]{CST14}).
\end{proof}

\begin{remark}\label{rem:independent-rankin-sign}
The Rankin--Selberg root number can be determined directly.
For $p\equiv8\pmod9$, one has
$\varepsilon(E_{p^j})=-1$ by \cite{BS66,Liverance}, as recalled in the
table of \Cref{subsec:descent-bsd}.  On the other hand,
one has 
$L(1,E_{D_{q,j}})\ne0$ by \Cref{prop:complementary-rank-zero}, and so 
$\varepsilon(E_{D_{q,j}})=+1$. (This also follows from a direct calculation.) The factorization
\eqref{eq:rankin-factorization} therefore gives Rankin--Selberg root
number $-1$.  
\end{remark}

We use the Yuan--Zhang--Zhang generalization of the Gross--Zagier
formula \cite{YZZ}, in the explicit form
\cite[Proposition~2.5]{CST14}.  Its unitary normalization
has central point $1/2$; with our arithmetic normalization,
\begin{equation}\label{eq:unitary-arithmetic-shift}
 L(s,\pi_{g_q},\theta)=L(s+\tfrac12,g_q\times\theta).
\end{equation} 
The central-character hypothesis is \eqref{eq:rankin-self-dual}.

\begin{proposition}\label{prop:gross-zagier-nonvanishing}
For $j=1,2$, one has 
\[
 L'(1,g_q\times\chi_{q,j})\ne0.
\]
\end{proposition}

\begin{proof}
The archimedean factors are finite and non-zero at the center, so
passing between completed and primitive $L$-functions does not affect
the first-derivative non-vanishing.

Put $\xi=\chi^\sigma$.  Write $\Phi$ as a finite sum of pure tensors as
in the proof of \Cref{lem:yzz-split-quaternion}, and choose a summand
$f$ in the $M_q$-rational automorphic representation $\pi_{A_q}$ with
\[
 Z=P_\xi(f)\ne0.
\]
By \Cref{lem:yzz-split-quaternion}, this is the split modular-curve case
of the Gross--Zagier formula \cite{YZZ}.  Let $\rho_q:A_q\to A_q^\vee$ be a
polarization over $\mathbf Q$ and put $f^\vee=\rho_q\circ f$, which is
again a pure tensor under the compatible local decompositions.  The
invariant pairing of $f$ and $f^\vee$ is non-zero: after either embedding
of $M_q$ in $\mathbf C$, it is the value on $f$ of the positive Hermitian
form induced by $\rho_q$.  The Rosati involution on $M_q\simeq K$ is
complex conjugation.  Thus
$\rho_q\circ[m]=[\bar m]^\vee\circ\rho_q$ for $m\in M_q$, and
$\overline{\xi_M(u)}=\xi_M(u)^{-1}$; consequently
$P_{\xi^{-1}}(f^\vee)=\rho_qZ$.  

The
$M_q\otimes_{\mathbf Q}\mathbf R$-valued N\'eron--Tate pairing
of \cite[\S1B]{CST14} between these two points is non-zero, since its
real trace is the positive N\'eron--Tate height attached to $\rho_q$
and $Z$ is non-torsion.  The explicit formula of
\cite[Proposition~2.5]{CST14} expresses this
height as the central derivative multiplied by the invariant pairing,
global normalizing factors, and a product of local toric factors.
Only finiteness of the local toric factors is needed here: if the
central derivative vanished, the right-hand side of the formula, and
hence the height, would vanish.  Since the height is non-zero, one gets
\[
 L'(\tfrac12,\pi_{g_q},\xi)\ne0.
\]
The result follows from \eqref{eq:unitary-arithmetic-shift} and
\Cref{lem:heegner-character-comparison}.
\end{proof}

\subsection{Proofs of main results}

\begin{theorem}\label{thm:rankin-derivative}
Let $p\equiv8\pmod9$, $q\equiv4\pmod9$, and
$\alpha_q(p)\ne1$.  For $j=1,2$,
\[
 L'(1,g_q\times\chi_{q,j})\ne0,
 \qquad
 L(1,E_{p^j})=0,
 \qquad
 L'(1,E_{p^j})\ne0.
\]
\end{theorem}

\begin{proof}
The first assertion is \Cref{prop:gross-zagier-nonvanishing}.  By
\Cref{rem:independent-rankin-sign}, the Rankin--Selberg root number is
$-1$.  Since \Cref{prop:complementary-rank-zero} gives
$\operatorname{ord}_{s=1}L(s,E_{D_{q,j}})=0$, taking orders in
\eqref{eq:rankin-factorization} yields
\[
 \operatorname{ord}_{s=1}L(s,g_q\times\chi_{q,j})
 =\operatorname{ord}_{s=1}L(s,E_{p^j})+0=1.
\]
This proves the remaining assertions.
\end{proof}

\begin{proof}[Proof of \Cref{thm:main}]
Choose $q$ as in \Cref{lem:chebotarev}.  Then
\Cref{thm:rankin-derivative}, for $j=1,2$, gives
\[
 \operatorname{ord}_{s=1}L(s,E_p)
 =\operatorname{ord}_{s=1}L(s,E_{p^2})=1.
\]
\end{proof}

\begin{proof}[Proof of \Cref{cor:arithmetic-rank} and \Cref{thm:sylvester}]
The theorems of Gross--Zagier and Kolyvagin turn analytic rank one into
Mordell--Weil rank one and finiteness of the Tate--Shafarevich group
\cite{GrossZagier,Kolyvagin,GrossKolyvagin}.  Hence
\Cref{thm:main} gives \Cref{cor:arithmetic-rank}.

Finally, the classical descent equivalence recalled at the start of the
introduction identifies positive rank of $E_m(\mathbf Q)$ with
non-trivial rational points on $X^3+Y^3=mZ^3$.  Thus both $p$ and $p^2$
are sums of two non-zero rational cubes.  Together with Hongbo Yin's results for
$p\equiv4,7\pmod9$ \cite{Yin47,YinGZ}, this proves
\Cref{thm:sylvester}.
\end{proof}

\appendix

\section{Examples}\label{app:examples}
We give two calculations illustrating the geometric aspects: the cubic descent class of the base CM point, the non-zero
boundary obtained after a first $\lambda$-division, and the passage
from that boundary to the cubic character components.  The case $q=13$ makes the division point and the boundary explicit,
while $q=31$ exhibits a non-trivial cube factor in the cubic
descent identity. 
\subsection{The cubic projector in the smallest quotient}

We first spell out the projector calculation that appears in the
examples.  After tracing away the prime-to-$3$ part of the
conductor-$p$ orbit, suppose that the remaining cyclic quotient is
$G=\langle X\rangle$ of order three.  For a point $Z$ on which $G$
acts, the two cubic components are
\[
 \Pi_{1,3}Z=Z+\omega^2Z^X+\omega Z^{X^2},
 \qquad
 \Pi_{2,3}Z=Z+\omega Z^X+\omega^2Z^{X^2},
\]
whereas the usual norm is
\[
 N_GZ=Z+Z^X+Z^{X^2}.
\]
Since $\omega\equiv1\pmod\lambda$, the difference between the norm and
each cubic projector is divisible by $\lambda$ in the group ring:
\[
 \frac{N_G-\Pi_{1,3}}{\lambda}=-\omega^2X+X^2,
 \qquad
 \frac{N_G-\Pi_{2,3}}{\lambda}=X-\omega^2X^2.
\]
This is the integral identity which allows the unweighted boundary to
detect the weighted cubic components.

Indeed, in the situation of \Cref{sec:cubic-components}, write
$S=N_GZ$.  If, say, $\Pi_{j,3}Z=O$, then
\[
 T_j:=\frac{N_G-\Pi_{j,3}}{\lambda}Z
\]
is $G$-fixed and satisfies $[\lambda]T_j=S$.  The boundary calculation
shows that
\[
 S\in E_\varpi[3]\setminus E_\varpi[\lambda],
\]
while cubic descent shows that such a point is not
$\lambda$-divisible over $F_q$.  This contradiction forces both cubic
components to be non-zero.  The weighted sums themselves reduce to
$O$ at $p$, because the points in the conductor-$p$ orbit have the
same reduction and the cubic weights sum to zero.  The local torsion
argument of \Cref{lem:torsion-reduction-kernel} then upgrades
non-vanishing to non-torsion.  Thus the projector transfers the single non-zero division boundary to the two cubic
character components.

\subsection{The example \texorpdfstring{$q=13$}{q=13}}\label{app:q13}
Take
\[
 q=13,\qquad
 \varpi=4+3\omega,\qquad
 \bar\varpi=1-3\omega,\qquad
 t^3=\varpi,
\]
and choose $r=23$, so that
\[
 r^2-r+1=13\cdot39,
 \qquad
 -r\equiv\omega\pmod\varpi.
\]
Put
\[
 \tau_{23}:=-\frac1{3(\omega+23)}=\tau_0.
\]

\begin{proposition}\label{prop:example-q13}
\begin{enumerate}[label=\textup{(\roman*)}]
\item The CM point and its first $\lambda$-division may be taken as
\[
 Q_{13}=
 \left((-2\omega-7)t^{-2},\frac{9\omega+7}{2}\right),
 \qquad
 R_{13}=[\omega-1:t:1],
\]
with $\ell_{\bar\varpi}(R_{13})=Q_{13}$ and
\[
 y(Q_{13})+\frac{\bar\varpi}{2}=\varpi,
 \qquad
 y(Q_{13})-\frac{\bar\varpi}{2}=(\omega-1)^3.
\]
Thus the descent class\footnote{In this example the
cube multiplier in the first factor happens to be $1$, so the Kummer
class can be read directly from the displayed identity.} is $[\varpi]_K$.

\item For $p=17$,
\[
 \alpha_{13}(17)=17^4\equiv\omega^2\pmod\varpi,
 \qquad
 \partial^c_{13,17}=T_E^\sharp=(0,\varpi/2).
\]
Here $|G_p|=6$ and
\[
 |\operatorname{Gal}(H_p^{(3)}/K)|=3.
\]
Thus, after the prime-to-$3$ trace, the cubic quotient is the
order-three situation described above.  Explicitly,
\[
 \Pi_{1,3}=1+\omega^2X+\omega X^2,
 \qquad
 \Pi_{2,3}=1+\omega X+\omega^2X^2,
\]
and
\[
 \frac{N_G-\Pi_{1,3}}{\lambda}=-\omega^2X+X^2,
 \qquad
 \frac{N_G-\Pi_{2,3}}{\lambda}=X-\omega^2X^2.
\]
By \Cref{thm:cubic-projection-criterion}, the non-zero boundary forces
both cubic components to be non-zero.  Since both reduce to $O$ at
$p$, \Cref{lem:torsion-reduction-kernel} shows that they are
non-torsion.

\item For $p=107$,
\[
 \alpha_{13}(107)=107^4\equiv\omega\pmod\varpi,
 \qquad
 \partial^c_{13,107}=-T_E^\sharp.
\]
Now
\[
 C_{p,3}\simeq\mathbf Z/27\mathbf Z,
 \qquad
 \operatorname{Gal}(H_p^{(3)}/K)\simeq\mathbf Z/9\mathbf Z,
\]
and, modulo $\lambda$,
\[
 \overline{\Pi_{1,9}}=\overline{\Pi_{2,9}}
 =1+X+\cdots+X^8=(X-1)^8.
\]
Thus the same norm--projector comparison persists when a higher power
of $3$ divides $p+1$.
\end{enumerate}
\end{proposition}

\begin{proof}
The formula for $Q_{13}=\varphi([\tau_{23},1])$ and the conjugate value
$\varphi^c([\tau_{23},1])=(0,\varpi/2)$ are
\cite[\S10, example~(2)]{Yin47}.  The two displayed descent factors are
direct substitutions, and $(\omega-1)\varpi=-7-2\omega$ gives the
stated Fermat lift.  These formulas specialize
\Cref{prop:uniform-kummer,prop:normalized-lift}.

For $p=17$ and $107$, the displayed cubic-residue calculations
determine the sign in \Cref{prop:frobenius-obstruction};
\Cref{prop:division-boundary} then gives the two boundary values.  The
group orders follow from \Cref{lem:ring-class-p-ramification}, and the
projector formulas are the specialization of
\Cref{lem:integral-projector-identities}.  The congruence for $n=9$
uses $X^9-1=(X-1)^9$ in characteristic $3$.
\end{proof}

For comparison with Part~I, when $p=17$ the residue parameter is
$\beta=2$ and Chan's two matrices are
\[
 R_1^T=
 \begin{pmatrix}
 2&2\\
 0&1
 \end{pmatrix},
 \qquad
 R_2^T=
 \begin{pmatrix}
 2&1\\
 0&2
 \end{pmatrix}.
\]
Both are invertible over $\mathbf F_3$ 
(cf.~\Cref{lem:chan-matrices}).  Thus the same cubic residue condition which
makes the Heegner boundary non-zero also proves the complementary rank-zero assertion in this example.

\subsection{A non-trivial cube factor for
\texorpdfstring{$q=31$}{q=31}}\label{app:q31}
Take $p=17$ and
\[
 \varpi=1+6\omega,\qquad
 \bar\varpi=-5-6\omega,\qquad
 t^3=\varpi.
\]
Choose $t$ with the cube-root normalization used in
\cite[\S10, example~(3)]{Yin47}.
Choose $r=26$; then
\[
 r^2-r+1=31\cdot21,
 \qquad
 \alpha_{31}(17)=\omega^2.
\]
Put
\[
 \tau_{26}:=-\frac1{3(\omega+26)}=\tau_0,\qquad
 u=\frac{-2-10\omega}{7},\qquad
 v=\frac{-5+3\omega}{7},\qquad
 R_{31}=[u:tv:1].
\]

\begin{lemma}\label{lem:example-q31}
One has
\[
 \varpi v^3-u^3=\bar\varpi,
\]
so $R_{31}\in\mathscr C_{\bar\varpi}(F_q)$.  For
$Q_{31}:=\varphi([\tau_{26},1])=\ell_{\bar\varpi}(R_{31})$,
\[
 y(Q_{31})+\frac{\bar\varpi}{2}=\varpi v^3,
 \qquad
 y(Q_{31})-\frac{\bar\varpi}{2}=u^3.
\]
In particular, the descent class is again $[\varpi]_K$, but 
the first descent value now contains a non-trivial cube factor.
 Moreover,
\[
 \partial^c_{31,17}=T_E^\sharp.
\]
Thus the Kummer class, rather than the particular representative of
the descent factor, is what enters the Frobenius boundary calculation.
\end{lemma}

\begin{proof}
The co-ordinates of $\varphi([\tau_{26},1])$ are given in
\cite[\S10, example~(3)]{Yin47}.  With the cube-root normalization
fixed above, the displayed $u$ and $v$ are the corresponding cube
roots.  Direct substitution in the formula
for $\ell_{\bar\varpi}$ identifies this point with
$\ell_{\bar\varpi}(R_{31})$.  The first identity is direct
multiplication in $K$; the two descent factors follow from the same
formula.  Thus this example realizes the general identity of
\Cref{prop:uniform-kummer} without the cube factor $1$
occurring for $q=13$.  Since $\alpha_{31}(17)=\omega^2$, the Frobenius
and boundary conclusions follow exactly as in
\Cref{prop:frobenius-obstruction,prop:division-boundary}.  Since $p=17$, the cubic quotient is again of order three, so the
projector calculation is the same as in
\Cref{prop:example-q13}(ii).
\end{proof}
\section{The construction for \texorpdfstring{$p<500$}{p < 500}}\label{app:table}

In this appendix we make explicit, for the fourteen primes
$p\equiv8\pmod9$ below $500$, the CM construction of rational points on
the cube-sum elliptic curves $E_p$, and record the resulting
representations of $p$ as a sum of two rational cubes.

\Cref{tab:construction} lists the auxiliary prime $q$, the cubic residue
symbol $\alpha_q(p)$, the division boundary, and the two group orders
entering the cubic projector.  The resulting identities are given in
\Cref{prop:table-construction}, while \Cref{tab:points} gives explicit
rational points on $E'_{p}$ together with size and height data for chosen
points on $E'_{p^2}$.  Finally, in \Cref{app:table-cm} we evaluate the
$E_p$-component of the CM construction numerically and identify it with
the corresponding rational point.

\subsection{Normalizations}\label{app:table-conventions}
Throughout, $\varpi=a+b\omega$ denotes the factor of $q$ determined by
\begin{equation}\label{eq:primary-normalization}
 a\equiv1\pmod 3,\qquad
 b\equiv0\pmod 3,\qquad
 b>0.
\end{equation}
The congruences choose the primary generator of either prime above $q$,
and $b>0$ chooses the factor with $\operatorname{Im}\varpi>0$.  This is
the convention of
\Cref{app:q13,app:q31}, where $13=(4+3\omega)(1-3\omega)$ and
$31=(1+6\omega)(-5-6\omega)$.

For each $p$ we take $q$ to be the least prime satisfying the two
conditions of \Cref{lem:chebotarev}.  Only $q=13$, $31$ and $67$ occur
below $500$, and $q=13$ is used for ten of the fourteen primes.

By \eqref{eq:affine-sign-table}, the sign $\varepsilon_{q,p}$ is the
product of $\eta_q$, which depends only on $q$, with $+1$ or $-1$
according as $\alpha_q(p)=\omega^2$ or $\alpha_q(p)=\omega$; we write
this as $\pm\eta_q$.  One has
$\eta_{13}=\eta_{31}=+1$ by \Cref{prop:example-q13} and
\Cref{lem:example-q31}.  We do not determine $\eta_{67}$, since only
non-vanishing of the boundary is used.  Here $n_0$ denotes the order of
$\operatorname{Gal}(H_p^{(3)}/K)$, the cyclic group on which the
projector of \Cref{thm:cubic-projection-criterion} acts.

\subsection{The division boundary}\label{app:table-boundary}

\begin{proposition}\label{prop:table-construction}
Let $p\equiv8\pmod 9$ be one of the fourteen primes below $500$, and
choose $q$ and $\varpi$ as in \Cref{app:table-conventions}.  Then
$q\in\{13,31,67\}$, and the values of $q$, $\varpi$, $\alpha_q(p)$ and
$\varepsilon_{q,p}$ are those recorded in \Cref{tab:construction}.
Moreover,
\[
 \partial^c_{q,p}=\varepsilon_{q,p}\Bigl(0,\frac{\varpi}{2}\Bigr)\ne O,
 \qquad
 |G_p|=\frac{p+1}{3},
 \qquad
 n_0=3^{\,\operatorname{ord}_3(p+1)-1}.
\]
In particular $3\mid n_0$.
\end{proposition}

\begin{proof}
The value of $\alpha_q(p)$ and the minimality of $q$ are checked directly
by evaluating
$\alpha_q(p)\equiv p^{(q-1)/3}\pmod{\varpi}$ over the primes
$q\equiv4\pmod 9$ up to the tabulated value, with $\varpi$ normalized by
\eqref{eq:primary-normalization}.  Since $\alpha_q(p)\ne1$,
\Cref{prop:division-boundary} gives
$\partial^c_{q,p}=\varepsilon_{q,p}T_E^\sharp$, and
$T_E^\sharp=(0,\varpi/2)$ by \eqref{eq:TE-explicit}; the sign is
determined by $\eta_q$ and $\alpha_q(p)$ through
\eqref{eq:affine-sign-table}.  The order of $G_p$ is
\eqref{eq:Gp-identification}, and since $G_p$ is cyclic its maximal
$3$-power quotient has order $3^{\operatorname{ord}_3(p+1)-1}$.  As
$p\equiv8\pmod 9$ forces $\operatorname{ord}_3(p+1)\ge2$, one has
$3\mid n_0$, which is the hypothesis
required in \Cref{thm:cubic-projection-criterion}.
\end{proof}

The coefficients $+\eta_q$ and $-\eta_q$ occur six and eight times,
respectively, and every boundary is non-zero.

\begin{table}[htbp]
\caption{The data of the construction.  Here $q$ is the least admissible
auxiliary prime, $\varpi$ its primary factor, and
$\partial^c_{q,p}=\varepsilon_{q,p}(0,\varpi/2)$.}
\label{tab:construction}
\begin{tabular}{r r l c c r r}
\toprule
$p$ & $q$ & $\varpi$ & $\alpha_q(p)$ & $\varepsilon_{q,p}$
 & $|G_p|$ & $n_0$\\
\midrule
 17 & 13 & $4+3\omega$ & $\omega^2$ & $+\eta_{13}$ &   6 & 3\\
 53 & 31 & $1+6\omega$ & $\omega$   & $-\eta_{31}$ &  18 & 9\\
 71 & 13 & $4+3\omega$ & $\omega^2$ & $+\eta_{13}$ &  24 & 3\\
 89 & 13 & $4+3\omega$ & $\omega$   & $-\eta_{13}$ &  30 & 3\\
107 & 13 & $4+3\omega$ & $\omega$   & $-\eta_{13}$ &  36 & 9\\
179 & 13 & $4+3\omega$ & $\omega$   & $-\eta_{13}$ &  60 & 3\\
197 & 13 & $4+3\omega$ & $\omega$   & $-\eta_{13}$ &  66 & 3\\
233 & 67 & $7+9\omega$ & $\omega$   & $-\eta_{67}$ &  78 & 3\\
251 & 13 & $4+3\omega$ & $\omega^2$ & $+\eta_{13}$ &  84 & 3\\
269 & 13 & $4+3\omega$ & $\omega^2$ & $+\eta_{13}$ &  90 & 9\\
359 & 31 & $1+6\omega$ & $\omega$   & $-\eta_{31}$ & 120 & 3\\
431 & 13 & $4+3\omega$ & $\omega$   & $-\eta_{13}$ & 144 & 9\\
449 & 13 & $4+3\omega$ & $\omega^2$ & $+\eta_{13}$ & 150 & 3\\
467 & 67 & $7+9\omega$ & $\omega^2$ & $+\eta_{67}$ & 156 & 3\\
\bottomrule
\end{tabular}
\end{table}

\subsection{The resulting points}\label{app:table-points}
The Jacobian of the diagonal cubic $X^3+Y^3=mZ^3$ is
\[
 E'_m:\quad y^2=x^3-432m^2,
\]
and $E'_m$ is $3$-isogenous over $\mathbf Q$ to $E_m$.  In particular,
$E'_m$ and $E_m$ have the same Mordell--Weil rank.
By
\Cref{thm:main} and \Cref{cor:arithmetic-rank}, $E'_p(\mathbf Q)$ and
$E'_{p^2}(\mathbf Q)$ have rank one.  For each such $m$, let $P_m$
denote the chosen non-torsion rational point used below.  Under the
birational correspondence
\[
 x=\frac{12m}{X+Y},\qquad y=\frac{36m(X-Y)}{X+Y},
\]
the point $P_m$ corresponds to a representation
\[
 m=\Bigl(\frac{\mathsf a}{\mathsf c}\Bigr)^3
  +\Bigl(\frac{\mathsf b}{\mathsf c}\Bigr)^3,
 \qquad
 \gcd(\mathsf a,\mathsf b,\mathsf c)=1,\quad \mathsf c>0.
\]
We write $\hat h$ for the N\'eron--Tate height on $E'_m$ normalized by
$\hat h(P)=\lim_{n}4^{-n}\log H(x(2^nP))$, with $H$ the naive height of a
rational number, and $d_{p^2}$ for the number of decimal digits of
$\max(|\mathsf a|,|\mathsf b|,\mathsf c)$ in the corresponding
representation of $p^2$.

\begin{table}[htbp]
\caption{The resulting points.  The columns $\mathsf a,\mathsf b,\mathsf c$
give $\mathsf a^3+\mathsf b^3=p\,\mathsf c^3$; $d_{p^2}$ and
$\hat h(P_{p^2})$ record the size and height of the corresponding point
for $p^2$.}
\label{tab:points}
\small
\begin{tabular}{r rrr r @{\qquad} r r}
\toprule
 & \multicolumn{4}{c}{$E'_p$} & \multicolumn{2}{c}{$E'_{p^2}$}\\
\cmidrule(lr){2-5}\cmidrule(lr){6-7}
$p$ & $\mathsf a$ & $\mathsf b$ & $\mathsf c$ & $\hat h(P_p)$
 & $d_{p^2}$ & $\hat h(P_{p^2})$\\
\midrule
 17 &        18 &          $-1$ &        7 &  1.9273 &  3 &   3.5644\\
 53 &      1872 &      $-1819$ &      217 &  5.0389 &  8 &  11.8337\\
 71 &       197 &       $-126$ &       43 &  3.5477 & 10 &  13.9706\\
 89 &        53 &           36 &       13 &  2.7322 & 18 &  27.1254\\
107 &        90 &           17 &       19 &  3.0052 &  5 &   6.7127\\
179 &   2184480 &   $-1305053$ &   357833 &  9.7551 & 12 &  17.7055\\
197 &      2339 &      $-2142$ &      247 &  5.1962 &  6 &   8.9103\\
233 &    124253 &    $-124020$ &     3589 &  7.8235 & 53 &  79.8702\\
251 &      4284 &      $-4033$ &      373 &  5.5967 & 42 &  63.9714\\
269 & 800059950 & $-786434293$ & 45728263 & 13.6790 & 28 &  42.4408\\
359 &  77517180 &     50972869 & 11855651 & 12.1904 & 76 & 116.2647\\
431 &       701 &       $-270$ &       91 &  4.3814 & 27 &  40.2415\\
449 &       323 &          126 &       43 &  3.8775 & 13 &  19.5327\\
467 &      1170 &       $-703$ &      139 &  4.7338 & 93 & 142.6023\\
\bottomrule
\end{tabular}
\end{table}

The identities $\mathsf a^3+\mathsf b^3=m\,\mathsf c^3$ were checked
exactly.  In each case the analytic rank of $E'_m$ was also computed numerically and found
to be one, in agreement with \Cref{thm:main}.

\begin{remark}\label{rem:table-sizes}
The link between the two tables is \Cref{thm:main}, rather than a direct
formula.  The construction gives the non-zero cubic component $\Pi Z$ of
the conductor-$p$ orbit on $E_\varpi$ over $L=F_qH_p^{(3)}$, while the
passage to \Cref{tab:points} uses the YZZ formula and the theorems of
Gross--Zagier and Kolyvagin.  The last two columns record the corresponding
$p^2$-data.  For $p=467$, the chosen point on $E'_{p^2}(\mathbf Q)$
corresponds to the representation
$p^2=(\mathsf a/\mathsf c)^3+(\mathsf b/\mathsf c)^3$ with
{\scriptsize
\[
\begin{aligned}
\mathsf a&=673856426282652472611591340002544597289173081567459219714429659481676561842102473757286093341,\\
\mathsf b&=506194562976553759385625547774811785929849248184944704311187269648509279438723916191212747992,\\
\mathsf c&=12594511180312553706327489029730767187582693859494639681094586806325043933759054535070301061,
\end{aligned}
\]
}%
while for $p=359$ the numerators have $76$ digits.  A direct search at this height is impractical; the entries for $p$ itself
in \Cref{tab:points} are much smaller.
\end{remark}

\subsection{Explicit CM evaluations for
\texorpdfstring{$p<500$}{p < 500}}\label{app:table-cm}
We now evaluate the $E_p$-component of the construction.  For a prime
$p$ in \Cref{tab:construction}, define the single-orbit cubic component
\[
 P^\circ_{q,p,1}
 :=\sum_{\sigma\in G_p}
 [\vartheta_{p,1}(\sigma)^{-1}](Q^c_{q,p})^\sigma.
\]
Since $Q_C=Q^c_{q,p}$ for $C\in C_q^{(3)}$ by
\eqref{eq:hecke-field}, one has the exact relation
\begin{equation}\label{eq:normalized-cm-component}
 P_{q,p,1}=[h_q]P^\circ_{q,p,1},
 \qquad h_q=\frac{q-1}{3}.
\end{equation}
Choose $s\in H_p$ with $s^3=p$, put $u=s/t$, and use the maps
\[
 \operatorname{Tw}_{p/\varpi}(x,y)=(u^2x,u^3y),
 \qquad
 \kappa_p(x,y)=(-12x,-24\sqrt{-3}\,y).
\]
Thus $\operatorname{Tw}_{p/\varpi}:E_\varpi\to E_p$, while
$\kappa_p:E_p\to E'_p$; the transport used below is
$\kappa_p\circ\operatorname{Tw}_{p/\varpi}$.
We use the same point symbols for their transported images.  For each
$p$, the numerical value of $P^\circ_{q,p,1}$ was recognized, up to sign,
as an integer multiple of the chosen point $P_p$.  We record the positive
integer $c_p$ for which
\begin{equation}\label{eq:cm-multiple}
 P^\circ_{q,p,1}\quad\text{is numerically identified with}\quad
 \pm[c_p]P_p.
\end{equation}
The sign is suppressed because only the multiplier and non-vanishing are
relevant here.  By \eqref{eq:normalized-cm-component}, the full component
is then numerically identified with $\pm[h_q c_p]P_p$.

\begin{table}[htbp]
\caption{Numerical CM identifications after transport to $E'_p$.  The
columns $c_p$ and $h_q c_p$ are the positive multipliers in
\eqref{eq:cm-multiple} and in the corresponding identification of
$P_{q,p,1}$, respectively; the omitted sign is $\pm$.}
\label{tab:cm-evaluations}
\begin{tabular}{r r r r r}
\toprule
$p$ & $q$ & $h_q$ & $c_p$ & $h_q c_p$\\
\midrule
 17 & 13 &  4 & 1 &   4\\
 53 & 31 & 10 & 1 &  10\\
 71 & 13 &  4 & 4 &  16\\
 89 & 13 &  4 & 1 &   4\\
107 & 13 &  4 & 5 &  20\\
179 & 13 &  4 & 2 &   8\\
197 & 13 &  4 & 1 &   4\\
233 & 67 & 22 & 1 &  22\\
251 & 13 &  4 & 1 &   4\\
269 & 13 &  4 & 4 &  16\\
359 & 31 & 10 & 1 &  10\\
431 & 13 &  4 & 4 &  16\\
449 & 13 &  4 & 4 &  16\\
467 & 67 & 22 & 5 & 110\\
\bottomrule
\end{tabular}
\end{table}

For example, when $(p,q)=(17,13)$ a numerical evaluation of the
conductor-$17$ orbit to more than one hundred decimal digits gives
\[
 S=N_GZ=\Bigl(-t^2,\frac{\sqrt{-3}\,\varpi}{2}\Bigr),
\]
in agreement with \Cref{prop:example-q13}\textup{(ii)}.  The single-orbit
cubic component was recognized as
\[
 (84,-684)\in E'_{17}(\mathbf Q),
 \qquad
 17=\left(\frac{18}{7}\right)^3+\left(-\frac17\right)^3,
\]
which is $\pm P_{17}$ with the convention of \Cref{tab:points}.  For
$p=269$, the single-orbit component was numerically identified with
$\pm[4]P_{269}$; the corresponding cube representation has a denominator
with $142$ decimal digits.  Moreover,
\[
 \hat h([4]P_{269})=16\hat h(P_{269}).
\]
Thus the single-orbit CM component need not be the chosen point itself
and can have greater height, even for $p<500$.

The CM integrals were evaluated numerically at high precision.  Before
truncating the Fourier series, each argument was moved by the relevant
$\Gamma_0(N)$-action to improve convergence.  The resulting co-ordinates
were recognized as rational points on $E'_p$ and compared with multiples
of the points in \Cref{tab:points}.  This identification is numerical
recognition rather than a symbolic proof.  Once a candidate point was
recognized, its curve equation, the corresponding cube identity, and its
group-law relation to $[c_p]P_p$ were checked in exact rational arithmetic.
Thus the displayed rational points and group-law relations are exact; only
their identification with the analytic CM evaluations is numerical.  The
non-torsion of the CM components is proved independently in
\Cref{thm:cubic-points}.

\end{document}